\documentclass[11pt,reqno]{amsart}
\usepackage{amsmath,amssymb,amsthm,amsfonts,mathrsfs,latexsym,times,color,bm}
\usepackage{amsmath,amsfonts}
\usepackage{amsthm}
\usepackage{graphicx}
\usepackage{color}
\usepackage{multicol}

\usepackage{hyperref}

\newtheorem{thm}{Theorem}[section]
\newtheorem{definition}{Definition}[section]

\newtheorem{Lemma}[thm]{Lemma}
\newtheorem{remark}{Remark}[section]
\newtheorem{theorem}[thm]{Theorem}
\newtheorem{proposition}[thm]{Proposition}

\newtheorem{corollary}[thm]{Corollary}

\newcommand{\na}{\nabla}

\numberwithin{equation}{section}

\newcommand{\beq}{\begin{equation}}
\newcommand{\eeq}{\end{equation}}
\newcommand{\ben}{\begin{eqnarray}}
\newcommand{\een}{\end{eqnarray}}
\newcommand{\beno}{\begin{eqnarray*}}
\newcommand{\eeno}{\end{eqnarray*}}

\newcommand{\bv}{\mathbf{v}}

\newcommand{\bQ}{\mathbf{Q}}
\newcommand{\bw}{\mathbf{w}}
\newcommand{\bU}{\mathbf{U}}

\numberwithin{equation}{section}

\begin{document}
\begin{sloppypar}
\title[two-dimensional compressible Euler equations]{Improved Local Well-Posedness in Sobolev Spaces for Two-Dimensional Compressible Euler Equations}


\author{Huali Zhang}
\address{Hunan University, School of Mathematics, Lushan South Road in Yuelu District, Changsha, 410882, People's Republic of China.}

\email{hualizhang@hnu.edu.cn}

\date{\today}


\begin{abstract}
We establish the local existence and uniqueness of solutions to the two-dimensional compressible Euler equations with initial velocity $\bv_0$, logarithmic density $\rho_0$, and specific vorticity \(w_0\), which satisfy $(\bv_0, \rho_0, w_0, \nabla w_0)\in H^{\frac74+}(\mathbb{R}^2)\times H^{\frac74+}(\mathbb{R}^2) \times H^{\frac32}(\mathbb{R}^2) \times L^{8}(\mathbb{R}^2)$.

The proof applies Smith-Tataru method \cite{ST} and the inherent wave-transport structure of the two-dimensional compressible Euler equations. The key observation is that Strichartz estimates hold when the regularity requirement for vorticity is lower than that for velocity and density, even though the gradient of vorticity appears as a source term in the velocity wave equation. Furthermore, our result presents an improvement of $\frac{1}{4}$-order regularity compared to previous results \cite{Z1} and \cite{Z2}.
\end{abstract}

\maketitle
\section{Introduction}
\subsection{Background}
In this paper, we study the two-dimensional compressible Euler equations, which are of the form
\begin{equation*}\label{CEE}
	\begin{cases}
	\varrho_t+\text{div}\left(\varrho \bv \right)=0, \quad t>0, x\in \mathbb{R}^2,
	\\
	\bv_t + \left(\bv\cdot \nabla \right)\bv+\frac{1}{\varrho}\nabla p(\varrho)=0,
\end{cases}
\end{equation*}
where $\bv=(v^1,v^2)^{\mathrm{T}}, \varrho$, and $p$ denote the fluid velocity, density, and pressure, respectively. Throughout the paper, we assume it's far away from vacuum. If we denote the logarithmic density $\rho$ by
\begin{equation}\label{rho}
	\rho=\log \varrho,
\end{equation}
the compressible Euler equations can be expressed as
\begin{equation}\label{CEE0}
	\begin{cases}
		\rho_t+ \left(\bv\cdot \nabla \right)\rho=-\textrm{div} \bv,\quad t>0, \ x\in \mathbb{R}^2,
		\\
		\bv_t + \left(\bv\cdot \nabla \right)\bv+p'(\mathrm{e}^{\rho})\nabla \rho=0.
	\end{cases}
\end{equation}
We consider the state function 
\begin{equation}\label{pq}
  p(\varrho)=\varrho^\gamma (\textrm{constant}\ \gamma\geq 1 ),
\end{equation}
and the initial data as
\begin{equation}\label{id}
	(\bv, \rho)|_{t=0}=(\bv_0, \rho_0).
\end{equation}
The compressible Euler equations \eqref{CEE0} describes the motion of compressible ideal fluids, which has wide applications in gas dynamics and astrophysics. In mathematics, a fundamental question concerns the sharp local well-posedness in Sobolev spaces for the equations \eqref{CEE0}-\eqref{id}. Although a few results \cite{AZ,DLS,  M, WQEuler,Z1, Z2} address this question, it remains an open and challenging problem in both two and three dimensions. In this paper, we further investigate the well-posedness of rough solutions to the equations \eqref{CEE0} in two dimensions, under lower regularity conditions than those established in previous studies \cite{Z1, Z2}. To better understand the equations \eqref{CEE0}, we first review some historical results related to two specific cases: the incompressible case and the irrotational case.

For the Cauchy problem of $n$-D incompressible Euler equations:
\begin{equation}\label{IEE}
	\begin{cases}
	\bv_t + \left(\bv\cdot \nabla \right)\bv+\nabla p=0,  \quad t>0, \ x\in \mathbb{R}^n,
	\\
	\mathrm{div} \bv=0,
\end{cases}
\end{equation}
Kato and Ponce \cite{KP} proved the local existence and uniqueness of solutions to the equations \eqref{IEE} for initial velocity $\bv_0 \in W^{s,p}(\mathbb{R}^n)$, where $s>1+\frac{n}{p}, 1<p<\infty$. Chae \cite{Chae,Ch2} extended the result of Kato and Ponce to the critical Triebel-Lizorkin space. For a discussion on the continuous dependence, we refer to the work of Tao \cite{Tao} and Guo-Li \cite{GL}. In contrast, Bourgain and Li \cite{BL, BL2} proved the ill-posedness of \eqref{IEE} by constructing initial data $\bv_0 \in W^{1+\frac{n}{p},p}(\mathbb{R}^n)$ (for $n=2,3$) that leads to instantaneous blow-up. Recently, Kim and Jeong \cite{KJ} proved the ill-posedness of \eqref{IEE} when $\bv_0 \in H^{\frac{n}{2}+1}(\mathbb{R}^n)$, $n\geq 3$.

In the case of irrotational flow, the compressible Euler equations can be written as a specific quasilinear wave equation. More generally, quasilinear wave equations are expressed in the following form:
\begin{equation}\label{qwe}
\begin{cases}
  &\square_{h(\phi)} \phi=q(d \phi, d \phi), \quad t>0, \ x\in \mathbb{R}^n,
  \\
  & (\phi,\partial_t \phi)|_{t=0}=(\phi_0, \phi_1) ,
  \end{cases}
\end{equation}
where $\phi$ is a scalar function, $d=(\partial_t, \partial_1, \partial_2, \cdots, \partial_n)$, and $h(\phi)$ a Lorentzian metric depending on $\phi$, and $q$ a quadratic term of $d \phi$. Set $(\phi_0, \phi_1) \in H^s(\mathbb{R}^n) \times H^{s-1}(\mathbb{R}^n)$. The local existence and uniqueness of \eqref{qwe} was established by Hughes-Kato-Marsden \cite{HKM} when $s > {n}/{2} + 1$. For the continuous dependence, one can refer Ifrim and Tataru's paper \cite{IT1}. On the other hand, Lindblad \cite{L} established the ill-posedness of \eqref{qwe} by constructing counterexamples when $s=2,n=3$. Based on the idea of \cite{L}, Ohlman \cite{Olm} also proved the ill-posedness in two dimensions for \( s = \frac{7}{4} \). A gap exists between the well-posedness results in \cite{HKM} and the ill-posedness demonstrated in \cite{L}. To bridge this regularity gap and relax the requirements on the initial data, a natural idea is to establish Strichartz-type estimates for the derivatives $d \phi$. 
As a first step, consider the linear wave equation with variable coefficients
\begin{equation}\label{qw0}
  \square_{h(t,x)}\phi=0,
\end{equation}
as \eqref{qw0} provides a framework for understanding the nonlinear problems. Kapitanskij \cite{K} and Mockenhaupt-Seeger-Sogge \cite{MSS} discussed the Strichartz estimates for \eqref{qw0} with smooth coefficients $h$. For rough coefficients $h \in C^2$, the study of Strichartz estimates for \eqref{qw0} in two or three dimensions began with Smith's result \cite{Sm}. In all dimensions for $h \in C^2$, the corresponding Strichartz estimates was proved by Tataru \cite{T2}. Conversely, Smith and Sogge \cite{SS} constructed counterexamples showing that for $\alpha<2$ there exists $h \in C^\alpha$ for which the Strichartz estimates do not hold.

Secondly, the above regularity thresholds for \eqref{qwe} were independently improved by Bahouri-Chemin \cite{BC2} and Tataru \cite{T1} for $s > \frac{n}{2} + \frac{7}{8}, n=2$ or $s > \frac{n}{2} + \frac{3}{4}, n\geq3$. Subsequently, Tataru \cite{T3}  refined the regularity requirements to $s>\frac{n+1}{2}+\frac{1}{6}, n \geq 3$. Meanwhile, Smith and Tataru \cite{ST0} showed that the loss of one-sixth order derivatives is sharp for general variable coefficients \( h \). Therefore, to further improve the results for \eqref{qwe}, a new approach is necessary. Klainerman and Rodnianski \cite{KR2} introduced a novel vector-field method and a remarkable Ricci curvature decomposition, establishing local existence and uniqueness for \eqref{qwe} when $s>2+\frac{2-\sqrt{3}}{2}$ and $n=3$. Based on \cite{KR2}, Geba \cite{Geba} successfully adapted the vector-field approach to the case $n=2$, where the regularity exponent satisfies $s > \frac{7}{4} + \frac{5-\sqrt{22}}{4}$. By representing solutions via wave packets, the sharp results in dimensions two and three were finally established by Smith and Tataru \cite{ST} (see also in Lindblad \cite{L} and Ohlman \cite{Olm}), where the regularity requires $s>\frac{7}{4}, n=2$ or $s>2, n=3$ or $s>\frac{n+1}{2}, 4 \leq n \leq 5$. An alternative proof of the 3D case was also obtained by Wang \cite{WQSharp} using the vector-field approach. It is also very interesting to consider axisymmetric initial data; for this, please refer to the work by C.B. Wang \cite{WCB}, Zha-Hidano \cite{ZH}, Zhou-Lei \cite{ZL}. Additionally, we should mention the significant progress made on low regularity solutions of the Einstein vacuum equations and membrane equations. We refer readers to the following papers: Ai-Ifrim-Tataru \cite{AIT}, Allen-Andersson-Restuccia \cite{AAR}, Andersson and Moncrief \cite{AM}, Ettinger and Lindblad \cite{EL}, Klainerman and Rodnianski \cite{KR}, Klainerman-Rodnianski-Szefel \cite{KR1}, Moschidis-Rodnianski \cite{MR}, Tataru \cite{T4}, Wang \cite{WQRough}, Wang-Zhou \cite{WZ} and among others.



In the general case, concerning to Cauchy problem of $n$-D compressible Euler equations \eqref{rho}-\eqref{id}, it's well-posed if $(\bv_0, \rho_0) \in H^{s}(\mathbb{R}^n), s>1+\frac{n}{2}$ and the density is far away from vacuum, please refer to Majda's book \cite{M}. Based on the pioneering work of Luk and Speck \cite{LS1,LS2}, significant progress has been made in low regularity well-posedness for \eqref{rho}-\eqref{id} via Strichartz estimates. In the case of 3D, Disconzi-Luo-Mazzone-Speck \cite{DLS} proved the existence and uniqueness of solutions with initial entropy $S_0$, velocity $\bv_0$, logarithmic density $\rho_0$ and specific vorticity $\bw_0$ satisfying $(\bv_0,\rho_0,\bw_0,S_0) \in   H^{2+} \times H^{2+} \times H^{2+} \times H^{3+}$ and $\Delta S_0, \mathrm{curl} \bw_0 \in C^{0, \delta}$, where $\delta$ is a small positive number. Independently, Wang \cite{WQEuler} also established the existence and uniqueness of solutions if $(\bv_0,\rho_0,\bw_0) \in H^{s}(\mathbb{R}^3) \times H^{s}(\mathbb{R}^3) \times H^{s'}(\mathbb{R}^3), \ 2<s'<s $. The works \cite{DLS} and \cite{WQEuler} are based on vector-field approach. Using a different approach--Smith and Tataru's method \cite{ST}, and combining with semiclassical analysis, Andersson and Zhang \cite{AZ} were able to relax the regularity requirements in \cite{DLS,WQEuler}, and proved the complete well-posedness if $(\bv_0, \rho_0,\bw_0) \in H^{2+}(\mathbb{R}^3) \times H^{2+}(\mathbb{R}^3) \times H^{2}(\mathbb{R}^3)$ or $(\bv_0, \rho_0,\bw_0,S_0) \in H^{\frac{5}{2}}(\mathbb{R}^3) \times H^{\frac{5}{2}}(\mathbb{R}^3) \times H^{\frac{3}{2}+}(\mathbb{R}^3)\times H^{\frac{5}{2}+}(\mathbb{R}^3)$. In the case of 2D, Zhang \cite{Z1,Z2} also proved the well-posedness of \eqref{rho}-\eqref{id} by setting $(\bv_0, \rho_0, w_0, \nabla w_0)  \in  H^{\frac{7}{4}+}(\mathbb{R}^2) \times H^{\frac{7}{4}+}(\mathbb{R}^2) \times H^{\frac{7}{4}+}(\mathbb{R}^2) \times L^\infty(\mathbb{R}^2)$ or $(\bv_0, \rho_0, w_0) \in H^{\frac{7}{4}+}(\mathbb{R}^2) \times H^{\frac{7}{4}+}(\mathbb{R}^2) \times H^2(\mathbb{R}^2)$. For a discussion on ill-posedness or shock waves, we refer readers to important results due to Abbrescia-Speck \cite{AS}, An-Chen-Yin \cite{ACY1,ACY2}, Christodoulou-Miao \cite{CM}, Huang-Kuang-Wang-Xiang \cite{HKWX}, Lei-Du-Zhang \cite{LDZ}, Luo-Yu \cite{LY1,LY2}, Merle-Rapha\"el-Rodnianski \cite{MRR}, Qu-Xin \cite{QX}, Sideris \cite{S}, Speck \cite{S1}, and Yin \cite{Yin}. Additionally, we also mention several remarkable results on free-boundary problems of the compressible Euler equations, particularly those by Avadanei \cite{Av}, Coutand-Lindblad-Shkoller \cite{CLS}, Jang-Masmoudi \cite{JM}, and Ifrim-Tataru \cite{IT}. 

\subsection{Motivation}
As established in \cite{Z1, Z2}, the regularity conditions for velocity and density are weaker than those presented in Majda's book \cite{M}, due to the Strichartz estimates. Furthermore, the regularity conditions imposed on velocity and density align with those for quasilinear wave equations \cite{ST}. The current result \cite{AZ} shows that 3D compressible Euler equation is well-posed if $(\bv_0, \rho_0,\bw_0) \in H^{2+}(\mathbb{R}^3) \times H^{2+}(\mathbb{R}^3) \times H^{2}(\mathbb{R}^3)$. However, the current 2D results \cite{Z1, Z2} require that $(\bv_0, \rho_0) \in H^{\frac74+}(\mathbb{R}^2), w_0 \in H^2(\mathbb{R}^2)$ or $w_0 \in H^{\frac74+}(\mathbb{R}^2)\cap \dot{W}^{1,\infty}(\mathbb{R}^2)$. Compared with the above results in 2D and 3D, the regularity requirements for vorticity in 2D results \cite{Z1, Z2} remain less desirable. This motivates our investigation into two-dimensional vorticity-specific regularity conditions, with the aim of establishing well-posedness at the same level in 3D. Indeed, due to the finite propagation speed, we can consider the problem as a perturbation of a flat Minkowski metric. Let us start with the standard linear wave equation:
\begin{equation*}
	\begin{cases}
		& \square f= 0, \quad t>0, x\in \mathbb{R}^2,
		\\
		& (f,\partial_t f)|_{t=0}=(f_0,0),
	\end{cases}
\end{equation*}
where $\square=\partial_t^2- (\partial^2_1 + \partial^2_2)$. By applying Keel-Tao's result \cite{KT} and Knapp's counterexample, the sharp Strichartz estimates hold: 
\begin{equation*}
	\| df \|_{L^4_t L^\infty_x} \lesssim \| f_0 \|_{\dot{H}^{\frac74+}}.
\end{equation*}
Therefore, for the following wave-transport system (see Lemma \ref{FC} below for details)
\begin{align}\label{wea}
		\square_g \bv\approx & \nabla w + (d\bv,d\rho) \cdot (d\bv,d\rho), 
		\\\label{web}
		\square_g \rho\approx & (d\bv,d\rho) \cdot (d\bv,d\rho),
		\\\label{wec}
		\mathbf{T} w=& 0,
\end{align}
we can expect the Strichartz estimates $\|d\bv,d\rho\|_{L^4_t L^\infty_x}$ by setting $(\bv_0,\rho_0) \in H^{\frac74+}$. For $\nabla w$ appears as a source term in \eqref{wea}, a general idea suggests the initial vorticity $w_0 \in H^{\frac74+}$. This is the result in \cite{Z1,Z2}. By careful analysis, we can find that the setting, $w_0 \in H^{\frac74+}$ is much higher than desired. Let us make a short explanation as follows.

We introduce a decomposition for the 2D velocity by\footnote{This is inspired by Wang's 3D result \cite{WQEuler}.}
$$v^i=v^i_{-}+v^i_{+}, \quad v^i_{-}=(-\Delta)^{-1} \left(-\epsilon^{ia} \mathrm{e}^{\rho} \partial_a w  \right), \quad i=1,2.$$ Then $\bv_{+}=(v_{+}^1,v_{+}^2)$ and $\bv_{-}=(v_{-}^1,v_{-}^2)$ satisfy
\begin{align}\label{key1}
	\square_g \bv_{+}\approx & \mathbf{T} \mathbf{T} \bv_{-}+ (d\bv,d\rho) \cdot (d\bv,d\rho),
	\\\label{key2}
	 \Delta \mathbf{T} \bv_{-}\approx & \nabla \bv \cdot \nabla w.
\end{align}
If $g$ is a flat metric, we can directly obtain
\begin{equation*}
	\begin{split}
		\| d\bv_{+} \|_{L^4_t L^\infty_x} \lesssim & \|\mathbf{T} \bv_{-} \|_{L^1_tH^s_x} + \|(\bv,\rho)\|_{L^1_tH^{\frac74+}_x}
		\lesssim  \|\bv\|_{L^\infty_tH^{\frac74+} } + \|\rho\|_{L^\infty_tH^{\frac74+} } + \|w\|_{L^\infty_tH_x^{1+} }.
	\end{split}
\end{equation*}
The above inequality, combining with $ \|d \bv_{-}\|_{L^\infty_x} \lesssim \|w\|_{H^{1+} } $, imply that
\begin{equation*}
	\begin{split}
		\| d\bv \|_{L^4_t L^\infty_x} \lesssim & \|\bv\|_{L^\infty_tH^{\frac74+} } + \|\rho\|_{L^\infty_tH^{\frac74+} } + \|w\|_{L^\infty_tH_x^{1+} }
		\\
		 \lesssim & \|\bv_0\|_{H^{\frac74+} } + \|\rho_0\|_{H^{\frac74+} } + \|w_0\|_{H^{1+} }.
	\end{split}
\end{equation*}
When the metric $g$ in \eqref{key1} is non-flat and depends on $\bv$ and $\rho$, to obtain Strichartz estimates, we need to study the geometric properties of null hypersurfaces. This is controlled by $\square_g g$. If we set $w\in H^{s'}_x$ ($s'\leq \frac74$), we have
\begin{equation*}
	\square_g g = \nabla w + (d\bv,d\rho) \cdot (d\bv,d\rho) \in H_x^{s'-1}.
\end{equation*}
Definitely, it's crucial for us to find a condition for $s'$ to establish the Strichartz estimates. Through a careful adaptation of Smith-Tataru's method \cite{ST}, we find that the Sobolev regularity indices $s'>\frac{3}{2}$ is sufficient. Therefore, applying equations \eqref{web} and \eqref{key1}, we can obtain a bound for Strichartz estimates $\|d\bv_{+},d\rho\|_{L^4_t L^\infty_x}$. As for $d\bv_{-}$, using \eqref{key2}, we also get a bound using Sobolev's imbeddings. As a result, we can derive the desired Strichartz estimates $\|d\bv,d\rho\|_{L^4_t L^\infty_x}$ if $(\bv,\rho,w) \in L^\infty_t H_x^{\frac74+} \times L^\infty_t H_x^{\frac74+} \times L^\infty_t H_x^{\frac32+}$. This implies that we should assume $(\bv_0,\rho_0,w_0)\in H^{\frac74+} \times H^{\frac74+} \times H^{\frac32+}$. Meanwhile, we also need to close the energy estimates with $w \in L^\infty_t H_x^{\frac32+}$ for the vorticity. Therefore, using \eqref{wec}, we have
\begin{equation*}
	\mathbf{T} \nabla w =\nabla \bv \cdot    \nabla w  .
\end{equation*}
Due to commutator and product estimates, the additional condition $\nabla w \in L_x^8$ is necessary.

In a word, we first prove the local existence and uniqueness of solutions for the two dimensional compressible Euler equations with initial data $(\bv_0, \rho_0, w_0, \nabla w_0)\in H^{s}(\mathbb{R}^2)\times H^{s}(\mathbb{R}^2) \times H^{s_0-\frac14}(\mathbb{R}^2) \times L^{8}(\mathbb{R}^2)$, where $\frac74 <s_0 \leq  s \leq 2$. We also establish two additional useful properties: \emph{(i)} a Strichartz estimate for linear wave equations endowed with an acoustic metric, and \emph{(ii)} a type of Strichartz estimate for solutions with a regularity of velocity $\frac74+$, density $\frac74+$, and vorticity $1+$. Secondly, if the initial data \((\mathbf{v}_0, \rho_0, w_0, \nabla w_0)\) belong to the spaces \(H^{\frac{7}{4}+}(\mathbb{R}^2) \times H^{\frac{7}{4}+}(\mathbb{R}^2) \times H^{\frac{3}{2}}(\mathbb{R}^2) \times L^{8}(\mathbb{R}^2)\), by applying frequency truncation on the initial data, a stronger Strichartz estimate can be established on a short time interval. Moreover, this allows us to obtain a sequence of solutions defined on these short intervals. After that, we extend this sequence of solutions from short time intervals to a fixed, regular time interval, and a Strichartz estimate also holds with a loss of derivatives\footnote{This argument is inspired by the work of Ai-Ifrim-Tataru \cite{AIT}, which has been developed by Andersson and Zhang \cite{AZ} in 3D compressible Euler equations. }. Based on this Strichartz estimates, we prove the local existence and uniqueness of solutions to the 2D compressible Euler equations for initial data $(\bv_0, \rho_0, w_0, \nabla w_0)\in H^{\frac74+}(\mathbb{R}^2)\times H^{\frac74+}(\mathbb{R}^2) \times H^{\frac32}(\mathbb{R}^2) \times L^{8}(\mathbb{R}^2)$. Compared to the earlier work \cite{Z1, Z2}, our result reduces the required regularity for vorticity by one-quarter order.

Next, we give some necessary notations in this paper.
\subsection{Notations}
\begin{itemize}
\item  We give the notations $d=(\partial_t, \partial_{1}, \partial_{2})^\mathrm{T}$ and $\partial_{0}=\partial_t$. 
\item We set $\left< \xi \right>=(1+|\xi|^2)^{\frac{1}{2}}, \ \xi \in \mathbb{R}^2$, and denote by $\left< \nabla \right>$ the corresponding Bessel potential multiplier. 
\item We denote the fractional Laplacian operator 
\begin{equation*}
	\Lambda_x=(-\Delta)^{\frac{1}{2}},\quad \Delta=\partial^2_1 + \partial^2_2.
\end{equation*}
\item 
Let $\zeta$ be a smooth function with support in the shell $\{ \xi\in \mathbb{R}^2: \frac{1}{2} \leq |\xi| \leq 2 \}$. Here, $\xi$ denotes the variable of the spatial Fourier transform. Let $\zeta$ also satisfy the
condition $\sum_{\lambda=2^k, k \in \mathbb{Z}} \zeta(\lambda \xi)=1$. 
\item Following the book \cite{BCD}, we introduce the Littlewood-Palay operator $P_{\lambda}$ with the frequency $\lambda=2^k (k \in \mathbb{Z})$, which satisfies
\begin{equation*}
	P_\lambda  f = \int_{\mathbb{R}^2} \mathrm{e}^{-\mathrm{i}x\cdot \xi} \zeta(\lambda^{-1}\xi) \hat{f}(\xi)d\xi.
\end{equation*}
\item We also set
\begin{equation*}
	f_{<\lambda}=\sum_{\lambda' < \lambda}P_{\lambda}f.
\end{equation*}
\item When the function $f$ is related to space variables, we use the notation $\|f\|_{H^s}=\|f\|_{H^s(\mathbb{R}^2)}$. When $f$ is related to both time and space variables, we use the notation $\|f\|_{H_x^s}=\|f(t,\cdot)\|_{H^s(\mathbb{R}^2)}$.
Similarly, we use notation $\|f\|_{L^p}=\|f\|_{L^p(\mathbb{R}^2)}$ if $f$ depends only the spatial variables, and use $\|f\|_{L_x^p}=\|f(t,\cdot)\|_{L^p(\mathbb{R}^2)}$ if $f$ depends on both time and space variables. For mixed norms shall be denoted by $\| f\|_{L^p_t L^q_x}$.
\item 
We use four small parameters
\begin{equation}\label{a0}
	\epsilon_3 \ll \epsilon_2 \ll \epsilon_1 \ll \epsilon_0 \ll 1,
\end{equation}
and three notations
\begin{equation*}
	\delta \in (0,s-\frac74), \quad \delta_0\in (0,s_0-\frac74), \quad \delta_1=\frac{s-\frac74}{10}.
\end{equation*}
\item Assume there holds
\begin{equation}\label{HEw}
	|\bv_0, \rho_0| \leq C_0, \qquad c_s|_{t=0}>c_0>0,
\end{equation}
where $C_0$ and $c_0$ are positive constants. 
\item In the following, constant $C$ depending only on $C_0, c_0$ shall be called universal. Unless otherwise stated, all constants that appear are universal in this sense. The notation $X \lesssim Y$ means $X \leq CY$, where $C$ is a universal constant, possibly depending on $C_0, c_0$. Similarly, we write  $X \simeq Y$ when $C_1 Y \leq X \leq C_2Y$, with $C_1$ and $C_2$ universal constants, and $X \ll Y$ when $X \leq CY$ for a sufficiently large constant $C$.
\end{itemize}
\subsection{Statement of the result}
Before stating our result, let us introduce a definition of the sound speed, logarithmic density, specific vorticity, and acoustic metric.
\begin{definition}(\cite{LS1}, Definition 2.1)\label{shengsu}
	We denote the speed of sound
	\begin{equation}\label{ss}
		c_s=\sqrt{\frac{dp}{d\varrho}}.
	\end{equation}
	In view of \eqref{pq} and \eqref{ss}, we have
	\begin{equation}\label{ss1}
		c_s=c_s(\varrho)
	\end{equation}
	and
	\begin{equation}\label{ssd}
		c'_s=\frac{\text{d}c_s}{\text{d} \varrho}.
	\end{equation}
\end{definition}
\begin{remark}
	Due to \eqref{ss1}, \eqref{ssd}, and \eqref{rho}, thus $c_s$ and $c'_s$ are also functions depending on $\rho$.
\end{remark}
\begin{definition}(\cite{LS1}, Definition 2.1)\label{pw}
We denote the specific vorticity $w$ as
\begin{equation}\label{sv}
	w = \mathrm{e}^{-\rho} { {\mathrm{curl}\bv} }.
\end{equation}
\end{definition}

\begin{definition}(\cite{LS1}, Definition 2.3)\label{metricd}
We define the acoustical metric $g$ and the inverse acoustical metric $g^{-1}$ relative to the Cartesian coordinates as follows:
\begin{equation}\label{MD}
	\begin{split}
		&g=-dt\otimes dt+c_s^{-2}\sum_{i=1}^{2}\left( dx^i-v^idt\right)\otimes\left( dx^i-v^idt\right),
		\\
		&g^{-1}=-(\partial_t+v^i \partial_i)\otimes (\partial_t+v^j \partial_j)+c_s^{2}\sum_{i=1}^{2}\partial_i \otimes \partial_i.
	\end{split}
\end{equation}	
\end{definition}
\begin{definition}\label{vfu}
	Define a decomposition for the velocity $\bv=\bv_{+}+\bv_{-}$ and
	\begin{equation}\label{dvc}
		v^i=v_{+}^i+ v_{-}^i.
	\end{equation}
	Above, the vector $\bv_{-}=(v_{-}^1, v_{-}^2)^{\mathrm{T}}$ is defined by
	\begin{equation}\label{etad}
		-\Delta v_{-}^{i}=\epsilon^{ia}\mathrm{e}^{{\rho}}\partial_a w .
	\end{equation}
\end{definition}
Based on these definitions, let us introduce the equivalent system of \eqref{CEE0} under new variables.
\begin{Lemma}\label{FC}(\cite{LS1}, Proposition 2.4)
The compressible Euler equations \eqref{rho}-\eqref{pq} also satisfies the following wave-transport system
	\begin{equation}\label{fc}
		\begin{cases}
			&\square_g v^i =-\epsilon^{ia}e^{\rho}c^2_s\partial_a w+Q^i,
			\\
			&\square_g \rho=\mathcal{D}, \qquad \qquad \qquad \qquad \quad t>0, x\in \mathbb{R}^2,
			\\
			& \mathbf{T}w=0.
		\end{cases}
	\end{equation}
	Above, $Q^i$($i=1,2$), $\mathcal{D}$, and $\mathbf{T}$ are defined by
	\begin{equation}\label{Di}
		\begin{split}
			& Q^i=2\epsilon^{ia}c^2_s w \partial_a \rho-\left( 1+c_s^{-1}c'_s\right)g^{\alpha \beta} \partial_\alpha \rho \partial_\beta v^i,
			\\
			&\mathcal{D}=-3c_s^{-1}c'_sg^{\alpha \beta} \partial_\alpha \rho \partial_\beta \rho+2\textstyle{\sum}_{1 \leq a < b \leq 2} \big\{ \partial_a v^a \partial_b v^b-\partial_a v^b \partial_b v^a \big\},
			\\
			& \mathbf{T}=\partial_t + \bv \cdot \nabla,
		\end{split}
	\end{equation}
	and
	\begin{equation*}
		\epsilon^{ia}=
		\begin{cases}
			0, & \text{ if \ $i=a$ },\\
			1, & \text{ if \ $i<a$ },\\
			-1, & \text{ if \ $i>a$ }.
		\end{cases}
	\end{equation*}
	We also define $\bQ =(Q^1,Q^2)^\mathrm{T}$.
\end{Lemma}
We next introduce the first theorem in the paper.
\begin{theorem}\label{dl2}
Assume $\frac74<s_0\leq s \leq 2$ and \eqref{HEw} hold. Let $M_0$ be any positive constant. Consider the Cauchy problem \eqref{fc}. If
\begin{equation*}\label{chuzhi}
	\| \bv_0\|_{H^{s}} +
	\| {\rho}_0\|_{H^{s}} + \| w_0\|_{H^{s_0-\frac14}}+ \| \nabla w_0\|_{L^{8}}
 \leq M_0,
	\end{equation*}
there exists two positive constants $T$(depending on $C_0,c_0, M_0, s, s_0$) such that the Cauchy problem \eqref{fc} is locally well-posed. Precisely,

\begin{enumerate}
\item\label{po1} there exists a unique solution $(\bv,{\rho}) \in C([0,T],H_x^{s}) \cap C^1([0,T],H_x^{s-1})$, $w \in C([0,T],H_x^{s_0-\frac14})$ $\cap C^1([0,T],H_x^{s_0-\frac54})$, $\nabla w \in C([0,T],L_x^{8})$, and $(d\bv, d{{\rho}}) \in L^4_{[0,T]} L^\infty_x$, 


\item\label{po2} for any $1 \leq r \leq s+1$, and for each $t_0 \in [0,T]$, the following linear equation
	\begin{equation}\label{linear}
	\begin{cases}
	& \square_g f=\mathbf{T}\Theta +B, \qquad (t,x) \in [0,T]\times \mathbb{R}^2,
	\\
	&f(t_0,\cdot)=f_0 \in H_x^r, 
	\\
	& \partial_t f(t_0,\cdot)=f_1 \in H_x^{r-1},
	\end{cases}
	\end{equation}
admits a solution $f \in C([0,T],H^r) \times C^1([0,T],H^{r-1})$ and the following estimates hold:
\begin{equation}\label{E0}
\| f\|_{L_t^\infty H_x^r}+ \|\partial_t f\|_{L_t^\infty H_x^{r-1}} \leq C_{M_0} (\|f_0\|_{H_x^r}+ \|f_1\|_{H_x^{r-1}}+\| \Theta \|_{L^\infty_tH_x^{r-1} \cap L^1_tH_x^r}+\|B\|_{L^1_tH_x^{r-1}} ).
\end{equation}
Additionally, the following estimates hold, provided $a<r-\frac{3}{4}$,
\begin{equation}\label{SE1}
\| \left<\nabla \right>^a f\|_{L^4_{t}L^\infty_x} \leq C_{M_0} ( \|f_0\|_{H_x^r}+ \|f_1\|_{H_x^{r-1}}+\| \Theta \|_{L^\infty_tH_x^{r-1} \cap L^1_tH_x^r}+\|B\|_{L^1_tH_x^{r-1}} ),
\end{equation}
and the same estimates hold with $\left<\nabla \right>^a$ replaced by $\left<\nabla \right>^{a-1} d$. Here $C_{M_0}$ is a positive constant depending on $C_0,c_0, s, s_0$ and $M_0$.
\end{enumerate}

\end{theorem}
\begin{remark}
\begin{enumerate}
	\item The initial condition $(\bv_0,\rho_0)\in H^{\frac74+}$ ensures the Strichartz estimate for 2D wave system. The vorticity regularity $ w_0\in H^{\frac32+}$ controls the geometry condition of null hypersurface in section \ref{ABA}. Meanwhile, the additional condition $\nabla w_0 \in L^8$ is imposed specifically to close the energy bounds for vorticity due to product estimates.
	\item For the free boundary problems, there is no Strichartz estimates of solutions. Therefore, on the study of low regularity well-posedness, the free boundary problems is much different from Cauchy problem. For more details, please refer to Avadanei \cite{Av} and Ifrim-Tataru \cite{IT}.
\end{enumerate}	
\end{remark}
Let us now state our second theorem.
\begin{theorem}\label{dl3}
	Assume $s\in (\frac74,2]$ and \eqref{HEw} hold.  Consider the problem \eqref{fc}. Let $M_0$ be any positive constant. If
	\begin{equation}\label{chuzhi2}
		\| \bv_0\|_{H^{s}} +
		\| {\rho}_0\|_{H^{s}} + \| w_0\|_{H^{\frac32}}+ \| \nabla w_0\|_{L^{8}}
		\leq M_0,
	\end{equation}
	then there exists two positive constants $T^*$(depending on $C_0,c_0, M_0,s$) such that the Cauchy problem \eqref{fc} is locally well-posed. Precisely,
	
	\begin{enumerate}
	\item \label{point:1} there exists a unique solution $(\bv,{\rho}) \in C([0,T^*],H_x^{s}) \cap C^1([0,T^*],H_x^{s-1})$, $w \in C([0,T^*],H_x^{\frac32})$ $\cap C^1([0,T^*],H_x^{\frac12})$, $\nabla w \in C([0,T^*],L_x^{8})$, and $(d\bv, d{{\rho}}) \in L^4_{[0,T^*]} L^\infty_x$, 


	\item \label{point:2} for any $s-\frac34 \leq r \leq \frac{11}{4}$, and for each $t_0 \in [0,T^*]$, the following linear equation
	\begin{equation}\label{linearC}
		\begin{cases}
			& \square_g f=\mathbf{T}\Theta +B, \qquad (t,x) \in [0,T^*]\times \mathbb{R}^2,
			\\
			&f(t_0,\cdot)=f_0 \in H_x^r, 
			\\
			& \partial_t f(t_0,\cdot)=f_1 \in H_x^{r-1},
		\end{cases}
	\end{equation}
	admits a solution $f \in C([0,T^*],H^r) \times C^1([0,T^*],H^{r-1})$ and the following estimates hold:
	\begin{equation*}\label{E0C}
		\| f\|_{L_t^\infty H_x^r}+ \|\partial_t f\|_{L_t^\infty H_x^{r-1}} \leq C_{M_0} (\|f_0\|_{H_x^r}+ \|f_1\|_{H_x^{r-1}}+\| \Theta \|_{L^\infty_tH_x^{r-1} \cap L^1_tH_x^r}+\|B\|_{L^1_tH_x^{r-1}} ).
	\end{equation*}
	Additionally, the following estimates hold, provided $a<r-(s-1)$,
	\begin{equation*}\label{SE1C}
		\| \left<\nabla \right>^a f\|_{L^4_{t}L^\infty_x} \leq C_{M_0} ( \|f_0\|_{H_x^r}+ \|f_1\|_{H_x^{r-1}}+\| \Theta \|_{L^\infty_tH_x^{r-1} \cap L^1_tH_x^r}+\|B\|_{L^1_tH_x^{r-1}} ),
	\end{equation*}
	and the same estimates hold with $\left<\nabla \right>^a$ replaced by $\left<\nabla \right>^{a-1} d$. Here $C_{M_0}$ is a positive constant depending on $C_0,c_0, s$ and $M_0$.
	\end{enumerate}
	
\end{theorem}
\begin{remark}
	Compared with Theorem \ref{dl2}, Theorem \ref{dl3} lowers the regularity exponent of the vorticity from $\frac32+$ to exactly $\frac{3}{2}$. This reduction is non-trivial because, if we apply Smith-Tataru's approach directly, we cannot obtain the Strichartz estimates for velocity and density when $w_0\in H^{\frac32}$. 
	
	When $(\bv_0, \rho_0, w_0, \nabla w_0)\in H^{\frac74+}(\mathbb{R}^2)\times H^{\frac74+}(\mathbb{R}^2) \times H^{\frac32}(\mathbb{R}^2) \times L^{8}(\mathbb{R}^2)$, we first use the frequency truncation to get a sequence of initial data belonging in $ H^{\frac74+}(\mathbb{R}^2)\times H^{\frac74+}(\mathbb{R}^2) \times H^{\frac32+}(\mathbb{R}^2) \times L^{8}(\mathbb{R}^2)$. Based on Theorem \ref{dl2}, a sequence of solutions on a short time-interval, and these intervals depends on the size of frequency.
	By extending the solutions from these short time intervals to a uniformly regular time-interval, a Strichartz estimate with a loss of derivatives can be obtained. This leads us to prove the local existence and uniqueness of solutions to the 2D compressible Euler equations for initial data $(\bv_0, \rho_0, w_0, \nabla w_0)\in H^{\frac74+}(\mathbb{R}^2)\times H^{\frac74+}(\mathbb{R}^2) \times H^{\frac32}(\mathbb{R}^2) \times L^{8}(\mathbb{R}^2)$.
\end{remark}
\begin{remark}
	\begin{enumerate}
			\item Since the regularity needed for vorticity is less stringent than that for velocity and density, the theorems \ref{dl2} and \ref{dl3} can be seen as a non-trivial extension from 2D quasilinear wave equation to the 2D quasilinear wave-transport system. Moreover, the 2D results presented in Theorems \ref{dl2} and \ref{dl3} of this paper may be comparable to the known 3D results \cite{AZ, WQEuler}.
		\item Referring to the ill-posedness result in \cite{ACY1}, we expect that the regularity of velocity and density in Theorems \ref{dl2} and \ref{dl3} is optimal. Indeed, the sharp regularity of velocity, density, and vorticity for the Cauchy problem of 2D compressible Euler equations remains an open and challenging problem.
		\item Once it's far away from vacuum, Theorems \ref{dl2} and \ref{dl3} also hold for a more general state \( p = p(\varrho) \) such that $ p'(\varrho) > 0$.
	\end{enumerate}
\end{remark}

\subsection{Outline of the paper}
The organization of the remainder of this paper is as follows. 
In Section 2, we prove the total energy estimates and uniqueness of solutions. In Section 3-7, we give a complete proof for Theorem \ref{dl2}. In the subsequent Section \ref{Sec8}, we present a proof for Theorem \ref{dl3}.
\section{Structures and energy estimates}
In this part we aim to give good formulations of \eqref{CEE0}, energy estimates and uniqueness theorems.
\subsection{Good formulations}
Let us start by introducing a hyperbolic system of compressible Euler equations \eqref{CEE0}.
\begin{Lemma} \label{sh}\cite{Li}
Let $\bv$ and $\rho$ be a solution of \eqref{CEE0}. Set $\bU=(v^1,v^2,p(\mathrm{e}^\rho))^{\mathrm{T}}$. Then \eqref{CEE0} satisfies the following symmetric hyperbolic system
\begin{equation}\label{sq}
  A^0(\bU)\partial_t\bU +{A}^1(\bU)\partial_1\bU+{A}^2(\bU)\partial_2\bU=0,
\end{equation}
where $A^\alpha (\alpha=0,1,2)$ is defined by
\begin{equation*}
A^0=
\left(
\begin{array}{ccc}
\mathrm{e}^{\rho} & 0 & 0\\
0 & \mathrm{e}^{\rho} & 0\\
0 & 0 & \mathrm{e}^{-\rho} c_s^{-2}
\end{array}
\right ),\quad
A^1=\left(
\begin{array}{ccc}
\mathrm{e}^{\rho} v^1 & 0 & 1\\
0 & \mathrm{e}^{\rho} v^1 & 0\\
1 & 0 & v^1 \mathrm{e}^{-\rho} c_s^{-2}
\end{array}
\right ),
\end{equation*}
\begin{equation*}
A^2=\left(
\begin{array}{ccc}
\mathrm{e}^{\rho} v^2 & 0 & 0\\
0 & \mathrm{e}^{\rho} v^2 & 1\\
0 & 1 & v^2 \mathrm{e}^{-\rho} c_s^{-2}
\end{array}
\right ).
\end{equation*}
\end{Lemma}
\begin{remark}
	We also mention that there is another symmetrization due to Lefloch and Ukai \cite{LU}.
\end{remark}
Next, let us introduce a good wave equation for $\bv_{+}$.
\begin{Lemma}\label{wte1}(\cite{Z2}, Lemma 2.3)
Let $\bv$ and $\rho$ be a solution of \eqref{CEE0} and $w$ be defined in \eqref{sv}. Let $\bv_{+}$ and $\bv_{-}$ be denoted in Definition \ref{vfu}. Then $\bv_{+}=(v_{+}^1, v_{+}^2)^{\mathrm{T}}$ satisfies
\begin{equation}\label{fcp}
\begin{split}
&\square_g v^i_{+}=\mathbf{T}\mathbf{T} v_{-}^i+Q^i.
\end{split}
\end{equation}
\end{Lemma}
\begin{Lemma}\label{tw}(\cite{Z2}, Lemma 2.8)
Let $\bv$ and $\rho$ be a solution of \eqref{CEE0} and $w$ be defined in \eqref{sv}. Then $\na w$ satisfies
\begin{equation}\label{w1}
\mathbf{T} ( \partial_i w ) =  \partial_i v^j \partial_j w, \quad i=1,2.
\end{equation}
\end{Lemma}
\subsection{Commutator and product estimates}
We first introduce a classical commutator estimate.
\begin{Lemma}\label{jiaohuan}\cite{KP}
	Let $a \geq 0$ and $\Lambda_x=(-\Delta)^{\frac12}$. Then for any scalar functions $h$ and $f$, we have
	\begin{equation*}\label{200}
		\|\Lambda_x^a(hf)-(\Lambda_x^a h)f\|_{L^2_x(\mathbb{R}^n)} \lesssim  \|\Lambda_x^{a-1}h\|_{L^2_x(\mathbb{R}^n)}\|\nabla f\|_{L^\infty_x(\mathbb{R}^n)}+ \|h\|_{L^p_x(\mathbb{R}^n)}\|\Lambda_x^af\|_{L^q_x(\mathbb{R}^n)},
	\end{equation*}
	where $\frac{1}{p}+\frac{1}{q}=\frac{1}{2}$.
\end{Lemma}
Next, let us introduce some product estimates.
\begin{Lemma}\label{jiaohuan0}\cite{KP}
Let $F(u)$ be a smooth function of $u$, $F(0)=0$ and $u \in L^\infty_x$. For any $s \geq 0$, we have
	\begin{equation*}\label{201}
		\|F(u)\|_{H_x^s} \lesssim  \|u\|_{H_x^{s}}(1+ \|u\|_{L^\infty_x}).
	\end{equation*}
\end{Lemma}

\begin{Lemma}\label{ps}\cite{ST}
	Suppose that $0 \leq r, r' < \frac{n}{2}$ and $r+r' > \frac{n}{2}$. Then
	\begin{equation*}\label{20000}
		\|hf\|_{H^{r+r'-\frac{n}{2}}(\mathbb{R}^n)} \leq C_{r,r'} \|h\|_{H^{r}(\mathbb{R}^n)}\|f\|_{H^{r'}(\mathbb{R}^n)}.
	\end{equation*}
Moreover, if $-r \leq r' \leq r$ and $r>\frac{n}{2}$, we have
\begin{equation*}\label{20001}
		\|hf\|_{H^{r'}(\mathbb{R}^n)} \leq C_{r,r'} \|h\|_{H^{r}(\mathbb{R}^n)}\|f\|_{H^{r'}(\mathbb{R}^n)}.
	\end{equation*}
\end{Lemma}
Finally, let us introduce a modified Duhamel's principle.
\begin{Lemma}\label{LD}(\cite{Z2}, Lemma 7.3)
	Let the metric $g$ be defined in \eqref{metricd}. If $f(t,x;\tau)$ is the solution of
	\begin{equation*}
		\begin{cases}
			&\square_{g}f=0, \quad t> \tau,
			\\
			&(f,\mathbf{T} f)|_{t=\tau}=-(\Theta,B)(\tau,x),
		\end{cases}
	\end{equation*}
	then the function
	\begin{equation*}
		V(t,x)=\int^t_0 f(t,x;\tau)d\tau,
	\end{equation*}
	is the solution of the linear wave equation
	\begin{equation*}
		\begin{cases}
			&\square_{g}V=\mathbf{T}\Theta+B,
			\\
			&(V, \mathbf{T}V)|_{t=0}=(0,-\Theta(0,x)).
		\end{cases}
	\end{equation*}
\end{Lemma}
\subsection{Energy estimates}
We first introduce some energy estimates for lower-order terms in \eqref{fc}.
\begin{Lemma}\label{yux}( \cite{Z2}, Lemma 2.12 )
Assume $s\in (\frac74,2]$. Let $(\bv,\rho)$ be a solution of \eqref{CEE0} and $w$ be defined in \eqref{sv}. Let $\bv_{-}$ be defined in \eqref{etad}. Let $\mathcal{D}$ and $\bQ$ be stated in \eqref{Di}. Then the following estimates
\begin{equation}\label{YYE}
  \| \mathcal{D}, \bQ\|_{ H_x^{s-1}} \lesssim \| d\rho, d\bv \|_{L_x^\infty} \| \bv, \rho \|_{H_x^{s}},
\end{equation}
and
\begin{equation}\label{eta}
  \| \mathbf{T} \bv_{-} \|_{H_x^s} \lesssim \|\bv, \rho \|_{H_x^{s}}\| w\|_{H_x^{1+}} ,
\end{equation}
hold. Moreover, the function $\bv_{-}$ satisfies
\begin{equation}\label{eee}
  \| \bv_{-} \|_{H_x^{2+}}  \lesssim (1+\| {\rho} \|_{H^{s}_x}) \| w \|_{H_x^{1+}}.
\end{equation}
\end{Lemma}
Finally, we will present a total energy theorem in the following, which will be very useful.
\begin{theorem}\label{be}{(Total energy estimate:type 1)}
Assume $\frac74<s_0\leq s \leq 2$. Let $(\bv, \rho)$ be a solution of \eqref{CEE0} and $w$ be defined as in \eqref{sv}. Set
\begin{equation*}
  E(t)= \| \bv\|^2_{H_x^{s}}+\| \rho\|^2_{H_x^{s}}+\|w\|^2_{H_x^{s_0-\frac14}}+\|\nabla w\|^2_{L^8_x},
\end{equation*}
and
\begin{equation*}
  E_0= \| \rho_0\|^2_{H^{s}}+\|\bv_0\|^2_{H^{s}}+\|w_0\|^2_{H^{s_0-\frac14}}+\|\nabla w_0\|^2_{L^8}.
\end{equation*}
If $0<t\leq 1$, then the following energy inequality holds:
\begin{equation}\label{E7}
\begin{split}
	E(t)\leq   & CE_0 \exp\left\{   \int^t_0 \| d \bv, d\rho \|_{L^\infty_x} d\tau + C E_0 \exp ( \int^t_0 \| d \bv, d\rho \|_{L^\infty_x} d\tau )    \right\}.
\end{split}
\end{equation}
\end{theorem}
\begin{proof}
By using \eqref{sq}, for any $t>0$, we have
\begin{equation}\label{AA0}
	\begin{split}
		&\| \rho\|^2_{H_x^s}+ \|\bv\|^2_{H_x^s} \leq  C\left( \| \rho_0\|^2_{H^s}+ \|\bv_0\|^2_{H^s} \right)+  C {\int^t_0} \|d\bv, d\rho\|_{L^\infty_x}(\| \rho\|^2_{H_x^s}+ \|\bv\|^2_{H_x^s})d\tau .
	\end{split}
\end{equation}
Due to Gronwall's inequality, it yields
\begin{equation}\label{AA00}
	\begin{split}
		&\| (\rho, \bv) \|^2_{H_x^s} \leq  C \| ( \rho_0, \bv_0 )\|^2_{H^s} \exp \left( \int^t_0 \|(d\bv, d\rho)\|_{L^\infty_x} d\tau \right) .
	\end{split}
\end{equation}
Multiplying with $w$ on the third equation in \eqref{fc} and integrating it on $[0,t] \times \mathbb{R}^2$, we find
\begin{equation}\label{AA1}
	\begin{split}
		&\| w \|^2_{L_x^2} \leq C\| w_0 \|^2_{L_x^2} + \int^t_0 \|\nabla \bv\|_{L^\infty_x} \| w \|^2_{L_x^2} d\tau.
	\end{split}
\end{equation}
Operating derivatives $\Lambda_x^{s_0-\frac{1}{4}}$ on $\mathbf{T}w=0$, we then have
\begin{equation}\label{AA2}
	\begin{split}
	\mathbf{T} (\Lambda_x^{s_0-\frac{1}{4}} w)=[v^i, \Lambda_x^{s_0-\frac{1}{4}}]\partial_i w.
	\end{split}
\end{equation}
Multiplying with $\Lambda_x^{s_0-\frac{1}{4}} w$ on \eqref{AA2} and integrating it on $[0,t] \times \mathbb{R}^2$, which implies that
\begin{equation}\label{AA3}
	\begin{split}
		\| \Lambda_x^{s_0-\frac{1}{4}} w \|^2_{L^2_x}\leq & \| \Lambda_x^{s_0-\frac{1}{4}} w_0\|^2_{L^2_x}+ C \left| \int^t_0 \int_{\mathbb{R}^2} \mathrm{div}\bv \cdot | \Lambda_x^{s_0-\frac{1}{4}} w|^2 dxd\tau \right|
		\\
		& +C \left| \int^t_0 \int_{\mathbb{R}^2} [v^i, \Lambda_x^{s_0-\frac{1}{4}}]\partial_i w \cdot \Lambda_x^{s_0-\frac{1}{4}} w dxd\tau \right|
		\\
		\leq & \|w_0\|^2_{H^{s_0-\frac{1}{4}}}+ C\int^t_0 \| \nabla \bv\|_{L^\infty_x} \| \Lambda_x^{s_0-\frac{1}{4}} w\|^2_{L^2_x}d\tau
		\\
		&+ C \int^t_0 \| \Lambda_x^{s_0-\frac{1}{4}} \bv\|_{L^{\frac{8}{3}}_x}\| \nabla w\|_{L^8_x} \| \Lambda_x^{s_0-\frac{1}{4}} w\|_{L^2_x} d\tau .
	\end{split}
\end{equation}
Combining \eqref{AA1} with \eqref{AA3} yields to
\begin{equation}\label{AA4}
	\begin{split}
		\| w \|^2_{H^{s_0-\frac{1}{4}}_x}\leq & \|w_0 \|^2_{H^{s_0-\frac{1}{4}}_x}+  C\int^t_0 \| \nabla \bv\|_{L^\infty_x}  
		\| w \|^2_{H^{s_0-\frac{1}{4}}_x} d\tau
		\\
		&+ \int^t_0 \| \bv\|_{H^{s}_x}\| \nabla w\|_{L^8_x} \|w \|_{H^{s_0-\frac{1}{4}}_x} d\tau .
	\end{split}
\end{equation}
Applying \eqref{w1}, we also obtain the energy inequality
\begin{equation}\label{AA5}
	\begin{split}
		\| \nabla w \|^2_{L^8_x}\leq & \|\nabla w_0 \|^2_{L^8_x}+  C\int^t_0 \| \nabla \bv\|_{L^\infty_x}  \|\nabla w \|^2_{L^{8}_x} d\tau.
	\end{split}
\end{equation}
Thanks to \eqref{AA0}, \eqref{AA4} and \eqref{AA5}, we get
\begin{equation*}
		E(t)\leq  CE_0+  C\int^t_0 \left( \| d \bv, d\rho \|_{L^\infty_x}+ \| \bv\|_{H^{s}_x} \right)  E(\tau) d\tau .
\end{equation*}
By Gronwall's inequality, we derive 
\begin{equation}\label{AA7}
	E(t)\leq  CE_0 \exp\left\{   \int^t_0 \left( \| d \bv, d\rho \|_{L^\infty_x}+ \| \bv\|_{H^{s}_x} \right)  d\tau \right\}  .
\end{equation}
Due to \eqref{AA00}, then \eqref{AA7} becomes
\begin{equation}\label{AA8}
	\begin{split}
		E(t)\leq  & CE_0 \exp\left\{   \int^t_0 \| d \bv, d\rho \|_{L^\infty_x} d\tau + C E_0 \int^t_0 \exp \{ \int^t_0 \| d \bv, d\rho \|_{L^\infty_x} d\tau' \}   d\tau \right\}  
		\\
		\leq  & CE_0 \exp\left\{   \int^t_0 \| d \bv, d\rho \|_{L^\infty_x} d\tau + C tE_0 \exp ( \int^t_0 \| d \bv, d\rho \|_{L^\infty_x} d\tau )    \right\}   .
	\end{split}
\end{equation}
Applying \eqref{AA8} and $0<t\leq 1$, we finally have
\begin{equation*}
	\begin{split}
		E(t)\leq   & CE_0 \exp\left\{   \int^t_0 \| d \bv, d\rho \|_{L^\infty_x} d\tau + C E_0 \exp ( \int^t_0 \| d \bv, d\rho \|_{L^\infty_x} d\tau )    \right\}   .
	\end{split}
\end{equation*}
This implies that \eqref{E7} holds. Therefore, we complete the proof of this theorem.
\end{proof}

\begin{theorem}\label{bBe}{(Total energy estimate:type 2)}
	Assume $s\in (\frac74,2]$. Let $(\bv, \rho)$ be a solution of \eqref{CEE0} and $w$ be defined as in \eqref{sv}. Set
	\begin{equation*}
		\bar{E}(t)= \| \bv\|^2_{H_x^{s}}+\| \rho\|^2_{H_x^{s}}+\|w\|^2_{H_x^{\frac32}}+\|\nabla w\|^2_{L^8_x},
	\end{equation*}
	and
	\begin{equation*}
		\bar{E}_0= \| \rho_0\|^2_{H^{s}}+\|\bv_0\|^2_{H^{s}}+\|w_0\|^2_{H^{\frac32}}+\|\nabla w_0\|^2_{L^8}.
	\end{equation*}
	If $0<t\leq 1$, then the following energy inequality holds:
	\begin{equation}\label{EB7}
		\begin{split}
			\bar{E}(t)\leq   & C\bar{E}_0 \exp\left\{   \int^t_0 \| d \bv, d\rho \|_{L^\infty_x} d\tau + C \bar{E}_0 \exp ( \int^t_0 \| d \bv, d\rho \|_{L^\infty_x} d\tau )    \right\}.
		\end{split}
	\end{equation}
\end{theorem}

\begin{proof}
	
	Taking derivatives $\Lambda_x^{\frac32}$ on $\mathbf{T}w=0$, we get
	\begin{equation}\label{AB2}
		\begin{split}
			\mathbf{T} (\Lambda_x^{\frac32} w)=[v^i, \Lambda_x^{\frac32}]\partial_i w.
		\end{split}
	\end{equation}
	Multiplying \eqref{AB2} with $\Lambda_x^{\frac32} w$ and integrating it on $[0,t] \times \mathbb{R}^2$, which implies that
	\begin{equation}\label{AB3}
		\begin{split}
			\| \Lambda_x^{\frac32} w \|^2_{L^2_x}\leq & \| \Lambda_x^{\frac32} w_0\|^2_{L^2_x}+ C \left| \int^t_0 \int_{\mathbb{R}^2} \mathrm{div}\bv \cdot | \Lambda_x^{\frac32} w|^2 dxd\tau \right|
			\\
			& +C \left| \int^t_0 \int_{\mathbb{R}^2} [v^i, \Lambda_x^{\frac32}]\partial_i w \cdot \Lambda_x^{\frac32} w dxd\tau \right|
			\\
			\leq & \|w_0\|^2_{H^{\frac32}}+ C\int^t_0 \| \nabla \bv\|_{L^\infty_x} \| \Lambda_x^{\frac32} w\|^2_{L^2_x}d\tau
			\\
			&+ C \int^t_0 \| \Lambda_x^{\frac32} \bv\|_{L^{\frac{8}{3}}_x}\| \nabla w\|_{L^8_x} \| \Lambda_x^{\frac32} w\|_{L^2_x} d\tau .
		\end{split}
	\end{equation}
Combining \eqref{AA1} with \eqref{AB3} yields to
\begin{equation}\label{AB4}
	\begin{split}
		\| w \|^2_{H^{\frac{3}{2}}_x}\leq & \|w_0 \|^2_{H^{\frac{3}{2}}_x}+  C\int^t_0 \| \nabla \bv\|_{L^\infty_x}  
		\| w \|^2_{H^{\frac{3}{2}}_x} d\tau
		\\
		&+ \int^t_0 \| \bv\|_{H^{s}_x}\| \nabla w\|_{L^8_x} \|w \|_{H^{\frac{3}{2}}_x} d\tau .
	\end{split}
\end{equation}
Applying \eqref{AA0}, \eqref{AA5} and \eqref{AB4}, we obtain
\begin{equation*}
	\bar{E}(t)\leq  C\bar{E}_0+  C\int^t_0 \left( \| d \bv, d\rho \|_{L^\infty_x}+ \| \bv\|_{H^{s}_x} \right)  \bar{E}(\tau) d\tau .
\end{equation*}
By Gronwall's inequality, then \eqref{EB7} holds. We thus complete the proof of theorem \ref{bBe}.
\end{proof}

\subsection{Uniqueness of the solution}
We now introduce two corollaries regarding the uniqueness of solutions, which can be obtained directly from Theorems \ref{be} and \ref{bBe}.
\begin{corollary}\label{uniq}
Assume $\frac74<s_0\leq s \leq 2$ and \eqref{HEw} hold. Consider the Cauchy problem \eqref{fc} with the initial data $(\bv_0, \rho_0, w_0,$ $\nabla w_0) \in H^{s} \times H^{s} \times H^{s_0-\frac14} \times L^8$. If there exists a solution $(\bv, \rho, w)$ for \eqref{fc} and
$(\bv,\rho) \in C([0,T],H_x^s) $, $w \in C([0,T],H_x^{s_0-\frac14}) $, $\nabla w\in C([0,T],L^8_x)$ and $(d\bv, d\rho) \in {L^4_{[0,T]} L^\infty_x}$, then it's unique.
\end{corollary}
\begin{corollary}\label{uniq2}
	Assume $s\in (\frac74,2]$ and \eqref{HEw} hold. Consider the Cauchy problem \eqref{fc} with the initial data $(\bv_0, \rho_0, w_0,$ $\nabla w_0) \in H^{s} \times H^{s} \times H^{\frac32} \times L^8$. If there exists a solution $(\bv, \rho, w)$ for \eqref{fc} and
	$(\bv,\rho) \in C([0,T^*],H_x^s) $, $w \in C([0,T^*],H_x^{\frac32}) $, $\nabla w\in C([0,T^*],L^8_x)$ and $(d\bv, d\rho) \in {L^4_{[0,T^*]} L^\infty_x}$, then it's unique.
\end{corollary}
\section{Proof of Theorem \ref{dl2}}
Observing the conclusions in Theorem \ref{dl2}, it includes existence, uniqueness of solutions, and Strichartz estimates. By using Corollary \ref{uniq}, the uniqueness of solution also holds once we prove the existence of solutions and Strichartz estimates. Therefore, we only need to discuss the existence of solutions and Strichartz estimates. Our idea here is to reduce the proofs of Theorem \ref{dl2} to the case of smooth, compactly supported and small initial data.
\subsection{A reduction to the case of smooth initial data}
\begin{proposition}\label{p3}
	Assume $\frac74<s_0\leq s \leq 2$, $\delta\in (0, s-\frac{7}{4})$, and \eqref{HEw} hold. Consider the Cauchy problem \eqref{fc}. For each $M_0>0$, there exists positives $T$ and $M_1$ (depending on $s, s_0, C_0, c_0, M_0$) such that, for each smooth initial data $(\bv_0, \rho_0,   w_0)$ which satisfies
	\begin{equation*}\label{p30}
		\begin{split}
			&\|\bv_0 \|_{H^s}+ \|\rho_0  \|_{H^s} + \| w_0\|_{H^{s_0-\frac14}}+ \|\nabla w_0\|_{L^{8}}  \leq M_0,
		\end{split}
	\end{equation*}
	there exists a smooth solution $(\bv, \rho, w)$ to \eqref{fc} on $[-T,T] \times \mathbb{R}^2$ satisfying
	\begin{equation}\label{p31}
		\|\bv\|_{L^\infty_tH_x^s}+ \| \rho\|_{L^\infty_tH_x^s}+ \| w\|_{H_x^{s_0-\frac14}}+ \|\nabla w\|_{L_x^{8}} \leq M_1.
	\end{equation}
	Furthermore, the solution satisfies

	\begin{enumerate}
		\item dispersive estimate for $\bv$, $\rho$ and $\bv_+$
	\begin{equation}\label{p32}
		\|d \bv, d \rho, d\bv_{+}\|_{L^4_t C^\delta_x} \leq M_1.
	\end{equation}

	\item Let a function $f$ satisfy equation \eqref{linear}. For each $1 \leq r \leq s+1$, the Cauchy problem \eqref{linear} is well-posed in $H_x^r \times H_x^{r-1}$, and the following estimates
	\begin{equation}\label{p33}
		\|\left< \nabla \right>^a f\|_{L^4_t L^\infty_x} \leq   C_{M_0}( \| f_0\|_{H^r}+ \| f_1\|_{H^{r-1}}+\| \Theta  \|_{L^\infty_tH_x^{r-1}\cap L^1_t H_x^r}+ \| B \|_{L^1_t H_x^{r-1}} ),
	\end{equation}
	and
		\begin{equation}\label{p34}
		\|f \|_{L^\infty_t H^r_x}+\| \partial_t f \|_{L^\infty_t H^{r-1}_x} \leq  C_{M_0} ( \|f_0\|_{H^r}+ \|f_1\|_{H^{r-1}}+\| \Theta \|_{L^\infty_tH^{r-1} \cap L^1_tH_x^r}+\|B\|_{L^1_tH_x^{r-1}} ).
	\end{equation}
	hold, where $a<r-\frac34$ and $C_{M_0}$ is a constant depending on $s, s_0, C_0, c_0, M_0$. The similar estimates hold if we replace $\left< \nabla \right>^a$ by $\left< \nabla \right>^{a-1}d$.
\end{enumerate}
\end{proposition}
Next, we will give a proof for Theorem \ref{dl2} based on Proposition \ref{p3}.
\begin{proof}[Proof of Theorem \ref{dl2} by using Proposition \ref{p3}]
	Consider the initial data $(\bv_0, \rho_0, w_0, \nabla w_0) \in H^s \times H^s \times H^{s_0-\frac14} \times L^8$ and
	\begin{equation*}
		\|\bv_0\|_{H^s} + \|{\rho}_0 \|_{H^s} + \|w_0\|_{H^{s_0-\frac14}} + \|\nabla w_0\|_{L^8} \leq M_0.
	\end{equation*}
	Let $\{(\bv_{0k}, \rho_{0k}, w_{0k}, \nabla w_{0k})\}_{k \in \mathbb{Z}^{+}}$ be a sequence of smooth data which converges to $(\bv_0, \rho_0, w_0, \nabla w_0)$ in $H^s \times H^s \times H^{s_0-\frac14} \times L^8$, where $w_{0k}$ is given by
	\begin{equation*}
		w_{0k}=\mathrm{e}^{-\rho_{0k}}\mathrm{curl}\bv_{0k}, \qquad k \in \mathbb{Z}^{+}.
	\end{equation*}
	By use of Proposition \ref{p3}, for each data $(\bv_{0k}, \rho_{0k}, w_{0k})$, there exists a solution $(\bv_k, \rho_k, w_k)$ to \eqref{fc}, and corresponds to the initial data
	\begin{equation*}
		(\bv_k, \rho_k, w_k)|_{t=0}=(\bv_{0k}, \rho_{0k}, w_{0k}).
	\end{equation*}
	We note that the solution of \eqref{fc} also satisfies the symmetric hyperbolic system \eqref{sh}. Let $\bU_k=(\bv_k, p(\mathrm{e}^{\rho_k}))^{\mathrm{T}}, k \in \mathbb{Z}^+$. For $k,l \in \mathbb{Z}^+$, we get
	\begin{equation*}
		\begin{split}
			&A^0( \bU_k )\partial_t \bU_k+ \sum^2_{i=1}A^i( \bU_k )\partial_{i}\bU_k = 0,
			\\
			&A^0( \bU_l ) \partial_t \bU_l+ \sum^2_{i=1}A^i(\bU_l)\partial_{i}\bU_l=0.
		\end{split}
	\end{equation*}
Applying the standard energy estimates to $\bU_k - \bU_l$, it yields
	\begin{equation*}
		\frac{d}{dt}\|\bU_k - \bU_l \|_{H_x^{s-1}} \leq C_{k, l} \left(\| d\bU_k, d\bU_l\|_{L^\infty_x}\|\bU_k-\bU_l\|_{H_x^{s-1}}+ \|\bU_k-\bU_l\|_{L^\infty_x}\| \nabla \bU_l\|_{H_x^{s-1}} \right),
	\end{equation*}
	where $C_{k, l}$ depends on the $L^\infty_x$ norm of $\bU_k$ and $\bU_l$. By using the Strichartz estimate of $d\bv_k, d\rho_k, k \in \mathbb{Z}^+$, i.e. \eqref{p32} in Proposition \ref{p3}, we infer that
	\begin{equation*}
		\begin{split}
			\|(\bU_k-\bU_l)(t,\cdot)\|_{H_x^{s-1}} &\lesssim \|(\bU_k-\bU_l)(0,\cdot)\|_{H^{s-1}}
			\\
			&\lesssim\|(\bv_{0k}-\bv_{0l}, \rho_{0k}-\rho_{0l}) \|_{H^{s-1}}.
		\end{split}
	\end{equation*}
	Thus $\{(\bv_k,p(\mathrm{e}^{\rho_k}))\}_{k \in \mathbb{Z}^+}$ is a Cauchy sequence in $C([-T,T];H_x^{s-1})$. Let $(\bv,p(\mathrm{e}^{\rho}))$ be the limit. Then
	\begin{equation}\label{lit}
		\lim_{k\rightarrow \infty}(\bv_k,p(\mathrm{e}^{\rho_k}) )=(\bv,p(\mathrm{e}^{\rho})) \ \mathrm{in} \ C([-T,T];H_x^{s-1}).
	\end{equation}
	Since $p(\mathrm{e}^{\rho_k})=\mathrm{e}^{\gamma \rho_k}$, we therefore have $\rho_k=\frac{1}{\gamma}\ln
	\left\{ p(\mathrm{e}^{\rho_k}) \right\}$. By using \eqref{lit} and Lemma \ref{jiaohuan0}, we obtain
	\begin{equation*}
		\lim_{k\rightarrow \infty}\rho_k=\rho \ \mathrm{in} \ C([-T,T];H_x^{s-1}).
	\end{equation*}
	Due to
	\begin{equation*}
		w_k=\mathrm{e}^{-\rho_k}\mathrm{curl}\bv_k, \quad k \in \mathbb{Z}^+,
	\end{equation*}
and
\begin{equation*}
	\partial_t w_k+ (\bv_k \cdot \nabla) w_k=0.
\end{equation*}
we can obtain
\begin{equation}\label{QQ}
	\partial_t ( w_k- w_l ) + \bv_k \cdot \nabla ( w_k - w_l)=(\bv_l-\bv_k)\cdot \nabla w_l.
\end{equation}
Multiplying $w_k- w_l$ on \eqref{QQ} and integrating it on $[0,t] \times \mathbb{R}^2$, we obtain
	\begin{equation}\label{cx}
		\begin{split}
			\|w_k-w_l\|^2_{L^2_x} &\leq C\|w_{0k}-w_{0l}\|^2_{L^2_x}+ C\int^t_0 \|\nabla \bv_k \|_{L^\infty_x} \|w_k-w_l\|^2_{L^2_x} d\tau
			\\
			& + C\int^t_0\|\bv_k-\bv_l\|_{L_x^{2}} \|w_l\|_{H^1_x} \|w_k-w_l\|_{L^2_x} d\tau.
		\end{split}
	\end{equation}
For \eqref{cx}, using Gronwall's inequality, then $\{w_k\}_{k \in \mathbb{Z}^+}$ is a Cauchy sequence in $C([-T,T];L_x^{2})$. We denote $w$ as the limit, i.e.
	\begin{equation*}
		\lim_{k\rightarrow \infty} w_k = w \ \mathrm{in} \ C([-T,T];L_x^{2}).
	\end{equation*}
	Since $(\bv_k,\rho_k)$ is uniformly bounded in $C([-T,T];H_x^{s})$ and $(w_k, \nabla w_k)$ uniformly bounded in $H_x^{s_0-\frac14}\times L^8_x$, we therefore deduce $(\bv, \rho) \in C([-T,T];H_x^{s}), w \in C([-T,T];H_x^{s_0-\frac14}),$ and $\nabla w \in L^8_x$.

	Using Proposition \ref{p3} again, the sequence $\{(d\bv_k,d\rho_k)\}_{k \in \mathbb{Z}^+}$ is uniformly bounded in the space $L^4([-T,T];C_x^\delta)$. Consequently, the sequence $\{(d\bv_k, d\rho_k)\}_{k \in \mathbb{Z}^+}$ converges to $(d\bv, d\rho)$ in the space $L^4([-T,T];L_x^\infty)$. That is,
	\begin{equation}\label{cc1}
		\lim_{k \rightarrow \infty} (d\bv_k, d\rho_k)=(d\bv, d\rho) \ \ \textrm{in} \ \ L^4([-T,T];L^\infty_x).
	\end{equation}
	Due to \eqref{dvc}, we get
	\begin{equation}\label{cc2}
		\begin{split}
			\| d \bv_{+}\|_{L^\infty_x} \leq \| d \bv\|_{L^\infty_x}+\| d \bv_{-}\|_{L^\infty_x}.
		\end{split}
	\end{equation}
By applying \eqref{etad} and elliptic estimates, Sobolev inequalities, for $s_0>\frac74$ we derive
	\begin{equation*}
		\begin{split}
			\| d \bv_{-}\|_{L^\infty_x} & \lesssim \| d (-\Delta)^{-1}(\mathrm{e}^{{\rho}} \nabla w ) \|_{H^{s_0-\frac{3}{4}}_{x}}
			\\
			& \lesssim \| w\|_{{H}^{s_0-\frac{3}{4}}} + \| \rho \|_{{H}^{s}}
			\\
			& \lesssim \|{\rho}_0\|_{H^{s}}+ \|\bv_0\|_{H^{s}} + \|w_0 \|_{H^{s_0-\frac{1}{4}}}+ \|\nabla w_0 \|_{L^{8}}.
		\end{split}
	\end{equation*}
	As a result, we have $d \bv_{-} \in L^4([-T,T];L^\infty_x)$ for finite $T$. This combining with \eqref{cc1}-\eqref{cc2} yields $d \bv_{+} \in L^4([-T,T];L^\infty_x)$.
	
	It remains for us to prove \eqref{E0} and \eqref{SE1}.  For $1 \leq r \leq s+1$, using Proposition \ref{p3}, there exists solutions $f_k$ satisfying 
	\begin{equation}\label{fk}
		\begin{cases}
			&\square_{g_k} f_k=\mathbf{T}\Theta +B
			\\
			& (f_k, \partial_tf_k)|_{t=0}=(f_0,f_1).
		\end{cases}
	\end{equation}
	Here the metric $g_k$ has the same formulation as in \eqref{MD} by replacing $(\bv,\rho)$ to $(\bv_k,\rho_k)$. Using \eqref{p33} and \eqref{p34}, we have
	\begin{equation}\label{sre}
		\|\left<\nabla \right>^{a} f_k \|_{L^4_t L^\infty_x} \leq C_{M_0} (  \|f_0\|_{H^r}+ \|f_1\|_{H^{r-1}}+\| \Theta  \|_{L^\infty_tH^{r-1} \cap L^1_tH_x^r}+\|B\|_{L^1_tH_x^{r-1}} ),
	\end{equation}
	and
	\begin{equation}\label{e9e}
		\|f_k \|_{L^\infty_t H^r_x}+\| \partial_t f_k \|_{L^\infty_t H^{r-1}_x} \leq C_{M_0} (   \|f_0\|_{H^r}+ \|f_1\|_{H^{r-1}}+\| \Theta  \|_{L^\infty_tH^{r-1} \cap L^1_tH_x^r}+\|B\|_{L^1_tH_x^{r-1}} ),
	\end{equation}
	where $a< r-\frac{3}{4}$. From \eqref{e9e}, we obtain that there exists a subsequence such that there is a limit $f$ satisfying
	\begin{equation*}\label{e8e}
		\|f \|_{L^\infty_t H^r_x}+\| \partial_t f\|_{L^\infty_t H^{r-1}_x} \lesssim   \|f_0\|_{H^r}+ \|f_1\|_{H^{r-1}}+\| \Theta  \|_{L^\infty_tH^{r-1} \cap L^1_tH_x^r}+\|B\|_{L^1_tH_x^{r-1}}.
	\end{equation*}
	By use of \eqref{sre}, we have
	\begin{equation*}\label{s7e}
		\|\left<\nabla \right>^{a} f \|_{L^4_t L^\infty_x} \lesssim   \|f_0\|_{H^r}+ \|f_1\|_{H^{r-1}}+\| \Theta  \|_{L^\infty_tH^{r-1} \cap L^1_tH_x^r}+\|B\|_{L^1_tH_x^{r-1}}, \quad a< r-\frac{3}{4}.
	\end{equation*}
	Furthermore, taking $k\rightarrow \infty$ for \eqref{fk}, the limit $f$ also satisfies
	\begin{equation*}\label{ff}
		\begin{cases}
			&\square_{g} f=\mathbf{T}\Theta +B,
			\\
			& (f, \partial_tf)|_{t=0}=(f_0,f_1).
		\end{cases}
	\end{equation*}
	Therefore, we have finished the proof of Theorem \ref{dl2}.
\end{proof}

We next state a result with smooth, small initial data.
\subsection{A reduction to the case of small initial data}
Since the propagation speed of system \eqref{fc} is finite, we therefore set $c>0$ being the largest one. Let us state a  result with small data:
\begin{proposition}\label{p1}
	Assume $\frac74<s_0\leq s \leq 2$, $ \delta\in (0, s-\frac{7}{4})$, \eqref{HEw} and \eqref{a0} hold. Consider the Cauchy problem \eqref{fc}. Suppose the initial data $(\bv_0, {\rho}_0, w_0)$ be smooth, supported in $B(0,c+2)$ and satisfying
	\begin{equation}\label{300}
		\begin{split}
			&\|\bv_0\|_{H^s} + \|\rho_0 \|_{H^s} + \|w_0\|_{H^{s_0-\frac14}}+ \|\nabla w_0\|_{L^8}  \leq \epsilon_3.
		\end{split}
	\end{equation}
	Then the Cauchy problem \eqref{fc} admits a smooth solution $(\bv,\rho,w)$ on $[-1,1] \times \mathbb{R}^2$, which have the following properties:
	
	$\mathrm{(1)}$ energy estimates
	\begin{equation*}\label{402}
		\begin{split}
			&\|\bv\|_{L^\infty_t H_x^{s}}+\| \rho \|_{L^\infty_t H_x^{s}} + \| w \|_{H_x^{s_0-\frac14}}+ \|\nabla w\|_{L_x^8}  \leq \epsilon_2.
		\end{split}
	\end{equation*}

	$\mathrm{(2)}$ dispersive estimate for $\bv$ and $\rho$
	\begin{equation*}\label{s403}
		\|d \bv, d \rho, d \bv_{+}\|_{L^4_t C^\delta_x} \leq \epsilon_2,
	\end{equation*}

	$\mathrm{(3)}$ dispersive estimate for the linear equation

	Let $f$ satisfy
	the equation \eqref{linear}. For each $1 \leq r \leq s+1$, the Cauchy problem \eqref{linear} is well-posed in $H_x^r \times H_x^{r-1}$. Moreover, for $a<r-\frac34$, the following estimates
\begin{equation}\label{304}
	\|\left< \nabla \right>^a f\|_{L^4_t L^\infty_x} \leq C_{M_0} ( \| f_0\|_{H_x^r}+ \| f_1\|_{H_x^{r-1}}+\| \Theta \|_{L^\infty_tH_x^{r-1}\cap L^1_t H_x^r}+ \| B \|_{L^1_t H_x^{r-1}} ),
\end{equation}
and
\begin{equation}\label{305}
	\| f   \|_{L^\infty_t H^s_x} + \|  \partial_t f   \|_{L^\infty_t H^{s-1}_x}
	\leq C_{M_0} ( \|f_0\|_{H_x^r}+ \|f_1\|_{H_x^{r-1}}+\| \Theta \|_{L^\infty_tH_x^{r-1} \cap L^1_tH_x^r}+\|B\|_{L^1_tH_x^{r-1}} ).
\end{equation}
hold, where $C_{M_0}$ is a constant depending on $s, s_0, C_0, c_0, M_0$. The similar estimates hold if we replace $\left< \nabla \right>^a$ by $\left< \nabla \right>^{a-1}d$.
\end{proposition}
In the following, we will give a proof for Proposition \ref{p3} based on Proposition \ref{p1}.
\begin{proof}[Proof of Proposition \ref{p3} by using Proposition \ref{p1}]
We divide the proof into three steps.

{\bf Step 1. Scaling}. From \eqref{p30}, we can see
\begin{equation}\label{a4}
\begin{split}
& \|\bv_0 \|_{H^s}+  \| \rho_0 \|_{H^s} +  \| w_0 \|_{H^{s_0-\frac14}} + \|\nabla w_0 \|_{L^8} \leq M_0.
\end{split}
\end{equation}
Take the scaling
\begin{equation*}
\begin{split}
&\widetilde{\bv}(t,x)=\bv(Tt,Tx),\quad \widetilde{\rho}(t,x)=\rho(Tt,Tx). 
\end{split}
\end{equation*}
Let
\begin{equation*}
	\widetilde{w}= \mathrm{e}^{-\widetilde{\rho}} \mathrm{curl} \widetilde{\bv}.
\end{equation*}
Due to \eqref{a4}, we can compute out
\begin{equation*}
\begin{split}
& \|\widetilde{\bv}_0 \|_{\dot{H}^s}+ \|\widetilde{\rho}_0\|_{\dot{H}^s} \leq M_0 T^{s-1},
\\
&  \|\widetilde{w}_0 \|_{\dot{H}^{s_0-\frac14}} \leq M_0 T^{s_0-\frac14}, \quad \| \nabla \widetilde{w}_0 \|_{L^8}\leq M_0 T^{\frac34}.
\end{split}
\end{equation*}
Choose sufficiently small $T$ such that
\begin{equation*}\label{xxs}
\max\{ M_0 T^{\frac34}, M_0 T^{s_0-\frac14}, M_0 T^{s-1} \} \ll \epsilon_3.
\end{equation*}
Then we have
\begin{equation*}
\begin{split}
  & \|\widetilde{\bv}_0 \|_{\dot{H}^s}+  \|\widetilde{\rho}_0 \|_{\dot{H}^s} + \|\widetilde{w}_0 \|_{\dot{H}^{s_0-\frac14}} + \| \nabla \widetilde{w}_0 \|_{L^8} \leq \epsilon_3.
\end{split}
\end{equation*}
Definitely, the above homogeneous norm is not enough for us to apply Proposition \ref{p1}. Next, we use the physical localization technique to reduce the data with small in-homogeneous norm.

{\bf Step 2. Localization}. Let $\chi$ be a smooth function supported in $B(0,c+2)$, and which equals $1$ in $B(0,c+1)$. For any given $y \in \mathbb{R}^2$, we define the localized initial data near $y$:
\begin{equation*}
\begin{split}
&\bv^y_0=\chi(x-y)\left( \bv_0(x)-\bv_0(y)\right),
\\
& \rho_0^y=\chi(x-y)\left( \rho_0(x)-\rho_0(y)\right).
\end{split}
\end{equation*}
Using \eqref{sv}, we next define
\begin{equation*}
  w^y_0= \mathrm{e}^{-(\rho^y_0+\rho_0(y))}\text{curl}\bv^y_0.
\end{equation*}
Since $\frac74<s_0\leq s \leq 2$, it is not difficult for us to verify
\begin{equation}\label{a5}
 \| (\bv^y_0, \rho_0^y)\|_{H_x^s}+ \| w_0^y \|_{H_x^{s_0-\frac14}} \lesssim  \|\bv_0, \rho_0\|_{\dot{H}^s}+ \| w_0  \|_{\dot{H}^{s_0-\frac14}} \lesssim \epsilon_3,
\end{equation}
and
\begin{equation}\label{AA9}
	\| \nabla w_0^y \|_{L^8} \lesssim  \| \nabla w_0 \|_{L^8}+ \|\bv_0\|_{\dot{H}^s}\|\rho_0\|_{\dot{H}^s} \lesssim \epsilon_3.
\end{equation}
Applying \eqref{a5} and \eqref{AA9}, we get
\begin{equation*}
	\| \bv^y_0\|_{H_x^s}+ \|\rho_0^y\|_{H_x^s}+ \| w_0^y \|_{H_x^{s_0-\frac14}}+ \| \nabla w_0^y \|_{L^8} \lesssim  \epsilon_3.
\end{equation*}
\quad {\bf Step 3. Using Proposition \ref{p1}.} Due to Proposition \ref{p1}, there is a smooth solution $(\bv^y, \rho^y, w^y)$ on $[-1,1]\times \mathbb{R}^2$ satisfying the following Cauchy problem
\begin{equation}\label{p}
\begin{cases}
& \square_g v^i =-\epsilon^{ia}e^{\rho}c^2_s\partial^a w+Q^i,
\\
& \square_g \rho=\mathcal{D},
\\
& \mathbf{T} w =0.
 \\
& (\bv,\rho,w)|_{t=0}=(\bv^y_0,\rho^y_0,w^y_0),
\\
& (\partial_t \bv, \partial_t {\rho})|_{t=0}=(-\bv_0^y\cdot \nabla \bv^y_0+c^2_s\nabla \rho^y_0,-\bv^y_0 \cdot \nabla \rho^y_0-\mathrm{div}\bv^y_0),
\end{cases}
\end{equation}
where $Q^i$ ($i=1,2$) and $\mathcal{D}$ are stated as \eqref{Di}. As a result, $\bv^y+\bv_0(y), \rho^y+\rho_0(y), w^y$ also solves \eqref{p}, where the initial data coincides with $(\bv_0, \rho_0, w_0)$ in $B(y,c+1)$. At the same time, the following Strichartz estimate also holds:
\begin{equation}\label{a6}
   \|d \bv^y, d \rho^y \|_{L^4_t C_x^\delta} \leq \epsilon_2.
\end{equation}
Consider the restriction, for $y\in \mathbb{R}^2$,
\begin{equation}\label{ppw}
  \left( \bv^y+\bv_0(y) \right)|_{\text{K}^y}, \quad \left( { \rho}^y+{\rho}_0(y) \right)|_{\text{K}^y}, \quad  w^y |_{\text{K}^y},
\end{equation}
where $\text{K}^y$ is defined by
\begin{equation*}
  \text{K}^y=\left\{(t,x):ct+|x-y| \leq c+1, |t|<1 \right\}.
\end{equation*}
Then the restriction \eqref{ppw} solves \eqref{p} on $\text{K}^y$. By finite speed of propagation and the uniqueness of solutions of \eqref{fc}, if we give
\begin{align*}
  \bv(t,x)  &= \bv^y(t,x)+\bv_0(y), \quad (t,x) \in \text{K}^y,
  \\
   {\rho}(t,x)  &={\rho}^y(t,x)+{\rho}_0(y), \quad \ (t,x) \in \text{K}^y,
  \\
   w(t,x)  &=w^y(t,x), \quad \qquad \quad  \ \ (t,x) \in \text{K}^y,
\end{align*}
then $(\bv, \rho, w)$ satisfies \eqref{fc} on $[-1,1] \times \mathbb{R}^2$. By use of Theorem \ref{be}, we have for $t \in [-1,1]$
\begin{equation}\label{a7}
\begin{split}
  &  \| (\bv, \rho)\|_{H_x^{s}}+\|w\|_{H_x^{s_0-\frac14}} +\|\nabla w\|_{L_x^8} 
  \\
 =  & \| (\bv^y, \rho^y)\|_{H_x^{s}}+\|w^y\|_{H_x^{s_0-\frac14}}+\|\nabla w^y\|_{L_x^8}
  \\
  \leq & C ( \| (\bv^y_0, \rho_0^y)\|_{H_x^s}+ \| w_0^y \|_{H^{s_0-\frac14}} + \|\nabla w_0^y\|_{L^8} ) 
   \cdot \exp\big \{   \int^t_0 \| d \bv^y, d\rho^y \|_{L^\infty_x} d\tau 
  \\
  & + C (\| (\bv^y_0, \rho_0^y)\|_{H_x^s}+ \| w_0^y \|_{H^{s_0-\frac14}} + \|\nabla w_0^y\|_{L^8}) \exp ( \int^t_0 \| d \bv^y, d\rho^y \|_{L^\infty_x} d\tau )    \big\}
  \\
  \leq & 2CM_0\exp\left( 1+\mathrm{e}^{CM_0}\right) .
\end{split}
\end{equation}
Owing to \eqref{a6}, we get
 \begin{equation}\label{a8}
  \|d \bv, d \rho \|_{L^4_t C_x^\delta}\leq C \|d \bv^y, d \rho^y \|_{L^4_t C_x^\delta} \leq C.
\end{equation}
Taking $M_1=\max\left\{  2CM_0\exp\left( 1+\mathrm{e}^{CM_0}\right) , C \right\}$, then \eqref{p31} and \eqref{p32} hold.

It remains for us to prove \eqref{p33} and \eqref{p34}. Let the cartesian grid $2^{-\frac{1}{2}}\mathbb{Z}^2$ be in $\mathbb{R}^2$, and a corresponding smooth partition of unity be
\begin{equation*}
  \textstyle{ \sum }_{y \in 2^{-\frac{1}{2}} \mathbb{Z}^2 } \psi(x-y)=1,
\end{equation*}
such that the function $\psi$ is supported in the unit ball. Consider a function $f^y$ being a solution of
\begin{equation*}
\begin{cases}
&\square_{g^y} f^y=0,
\\
&(f^y, \partial_t f^y)|_{t=0}=(\psi(x-y)f_0, \psi(x-y)f_1),
\end{cases}
\end{equation*}
where $g^y$ has the same formulation as in \eqref{metricd} with the velocity $\bv^y$ and $\rho^y$. We therefore get
\begin{equation*}\label{gy}
  g^y=g, \quad (t,x) \in \text{K}^y.
\end{equation*}
By finite speed of propagation, for $(t,x)\in \text{K}^y$, we deduce
\begin{equation*}
  f^y=f, \quad (t,x)\in \text{K}^y.
\end{equation*}
Write $f$ as
\begin{equation*}
f(t,x)=\textstyle{ \sum_{y \in 2^{-\frac{1}{2}}\mathbb{Z}^2} } \psi(x-y)f^y(x,t).
\end{equation*}
Due to \eqref{304} and \eqref{305}, for $a< r- \frac{3}{4}$, we obtain
\begin{equation}\label{a10}
\begin{split}
 \| \left< \nabla \right>^{a} f   \|^4_{L^4_t L^\infty_x}
 \leq  & C{ \sum_{y \in 2^{-\frac{1}{2}}\mathbb{Z}^2} }   \|\psi(x-y)  \left< \nabla \right>^a f^y(x,t)  \|^4_{L^4_t L^\infty_x}
\\
\leq  &  C{ \sum_{y \in 2^{-\frac{1}{2}}\mathbb{Z}^2} }   \| \psi(x-y) (f_0,f_1) \|^4_{H^r \times H^{r-1}}.
\\
 \lesssim & \left(  \|f_0\|_{H_x^r}+ \|f_1\|_{H_x^{r-1}}+\| \Theta \|_{L^\infty_tH_x^{r-1} \cap L^1_tH_x^r}+\|B\|_{L^1_tH_x^{r-1}} \right)^4,
\end{split}
\end{equation}
and
\begin{equation}\label{a11}
\begin{split}
 & \| f   \|_{L^\infty_t H^s_x} + \|  \partial_t f   \|_{L^\infty_t H^{s-1}_x}
 \\
 \leq  & C{ \sum_{y \in 2^{-\frac{1}{2}}\mathbb{Z}^2} }   ( \| \psi(x-y) f^y(t,x)  \|_{L^\infty_t H^{s-1}_x}  +  \|\psi(x-y)  \partial_t f^y   \|_{L^\infty_t H^{s-1}_x}  )
\\
 \lesssim & \|f_0\|_{H_x^r}+ \|f_1\|_{H_x^{r-1}}+\| \Theta \|_{L^\infty_tH_x^{r-1} \cap L^1_tH_x^r}+\|B\|_{L^1_tH_x^{r-1}}.
\end{split}
\end{equation}
Due to \eqref{a7}, \eqref{a8}, \eqref{a10}, and \eqref{a11}, we therefore finish the proof of Proposition \ref{p3}.
\end{proof}
It remains for us to verify Proposition \ref{p1}.

\section{The proof of Proposition \ref{p1}: a bootstrap argument}\label{ABA}
In this part, we will reduce the proof of Proposition \ref{p1} to a bootstrap argument. To start, we introduce some necessary notations and definitions. Let $\mathbf{m}$ be a standard Minkowski metric satisfying
\begin{equation*}
	\mathbf{m}^{00}=-1, \quad \mathbf{m}^{ij}=\delta^{ij}, \quad i, j=1,2.
\end{equation*}
Taking $\bv=0, \rho=0$ in $g$, which is record $g(0)$. Then the inverse matrix of the metric $g$ is
\begin{equation*}
	g^{-1}(0)=
	\left(
	\begin{array}{ccc}
		-1 & 0 & 0 \\
		0 & c^2_s(0) & 0 \\
		0 & 0 & c^2_s(0) 
	\end{array}
	\right ),
\end{equation*}
where $c_s(0)=c_s(\rho)|_{\rho=0}$. By a linear change of coordinates which preserves $dt$, we may assume that $g^{\alpha \beta}(0)=\mathbf{m}^{\alpha \beta}$. Let $\chi$ be a smooth cut-off function supported in the region $B(0,3+2c) \times [-\frac{3}{2}, \frac{3}{2}]$, which equals to $1$ in the region $B(0,2+2c) \times [-1, 1]$. We denote
\begin{equation}\label{boldg}
	\mathbf{g}=\chi(t,x)(g-g(0))+g(0),
\end{equation}
where $g$ is defined as in \eqref{metricd}. Consider the Cauchy problem of the following wave-transport system
\begin{equation}\label{CS}
	\begin{cases}
		\square_{\mathbf{g}} v^i=-\epsilon^{ia}\mathrm{e}^{\rho}c_s^2 \partial_a w+Q^i,
		\\
		\square_{\mathbf{g}} {\rho}= \mathcal{D},
		\\
		\mathbf{T}w=0.
		\\
		(\bv,\rho,w)|_{t=0}=(\bv_0,\rho_0,w_0),
	\end{cases}
\end{equation}
where $Q^i$ ($i=1,2$) and $\mathcal{D}$ are defined in \eqref{Di}, and $\mathbf{g}$ is denoted in \eqref{boldg}. 

\begin{definition}{(Definition for the set $\mathcal{H}$)}
	We denote by $\mathcal{H}$ the family of smooth solutions $(\bv, \rho, w)$ to the system \eqref{CS} for $t \in [-2,2]$, with the initial data $(\bv_0, \rho_0, w_0)$ supported in $B(0,2+c)$, and for which
	\begin{equation}\label{401}
		\|\bv_0\|_{H^s} + \|\rho_0\|_{H^s}+ \| w_0\|_{H^{s_0-\frac14}}+ \|\nabla w_0\|_{ L^8}  \leq \epsilon_3, \ \ \ \ \ \ \ \ \ \ \ \ \ \
	\end{equation}
	\begin{equation}\label{402a}
		\| \bv\|_{L^\infty_{[-2,2]} H_x^{s}}+\|\rho\|_{L^\infty_{[-2,2]} H_x^{s}}+ \| w_0\|_{L^\infty_{[-2,2]} H_x^{s_0-\frac14}} + \|\nabla w\|_{L^\infty_{[-2,2]} L^8_x} \leq 2 \epsilon_2,
	\end{equation}
	\begin{equation}\label{403}
		\| d \bv, d \rho, d \bv_{+}\|_{L^4_{[-2,2]} C_x^\delta} \leq 2 \epsilon_2.
	\end{equation}
\end{definition}
\subsection{Statement of bootstrap argument}
\begin{proposition}\label{p2}
Assume $\frac74<s_0\leq s \leq 2$, $ \delta\in (0, s-\frac{7}{4})$, \eqref{HEw} and \eqref{a0} hold. Then there is a continuous functional\footnote{For the definition of $G$, see \eqref{500} below.} $G: \mathcal{H} \rightarrow \mathbb{R}^{+}$, satisfying $G(\mathbf{0},0)=0$, so that for each $(\bv, \rho, w) \in \mathcal{H}$ satisfying $G(\bv, \rho) \leq 2 \epsilon_1$ the following properties hold:
\begin{enumerate}
	\item	The function $\bv, \rho$, and $w$ satisfies $G(\bv, \rho) \leq \epsilon_1$.

	\item The following estimate holds,
	\begin{equation}\label{404}
		\|\bv\|_{L^\infty_{[-2,2]} H_x^{s}}+ \|\rho\|_{L^\infty_{[-2,2]} H_x^{s}} + \|w\|_{L^\infty_{[-2,2]} H_x^{s_0-\frac14}} + \|\nabla w \|_{L^\infty_{[-2,2]} L^8_x} \leq \epsilon_2,
	\end{equation}
	\begin{equation}\label{405}
		\|d \bv, d \rho, d \bv_{+}\|_{L^2_{[-2,2]} C^\delta_x} \leq \epsilon_2. \ \ \ \ \ \ \ \ \
	\end{equation}

\item For $1 \leq r \leq s+1$, the problem 
		\begin{equation*}\label{406}
		\begin{cases}
			& \square_{\mathbf{g}} f=\mathbf{T}\Theta+B, \qquad (t,x) \in [-2,2]\times \mathbb{R}^2,
			\\
			&f(t_0,\cdot)=f_0 \in H_x^r(\mathbb{R}^2), 
			\\
			& \partial_t f(t_0,\cdot)=f_1 \in H_x^{r-1}(\mathbb{R}^2),
		\end{cases}
	\end{equation*}
	 is well-posed in $H_x^r \times H_x^{r-1}$. Moreover, on ${[-2,2]} \times \mathbb{R}^2$, for $a<r-1$, we have the Strichartz estimate 
	 \begin{equation*}\label{407}
	 	\|\left< \nabla \right>^a f\|_{L^4_{t} L^\infty_x} \lesssim  \| f_0\|_{H_x^r}+ \| f_1\|_{H_x^{r-1}}+\| \Theta \|_{L^\infty_tH_x^{r-1}\cap L^1_t H_x^r}+ \| B \|_{L^1_t H_x^{r-1}},
	 \end{equation*}
	 and the energy estimate
	 \begin{equation*}\label{408}
	 	\|f\|_{L^\infty_{t} H^r_x}+ \|\partial_t f\|_{L^\infty_{t} H^{r-1}_x} \lesssim  \| f_0\|_{H_x^r}+ \| f_1\|_{H_x^{r-1}}+\| \Theta \|_{L^\infty_tH_x^{r-1}\cap L^1_t H_x^r}+ \| B \|_{L^1_t H_x^{r-1}}.
	 \end{equation*}
\end{enumerate}
\end{proposition}
Next, we will prove Proposition \ref{p1} based on Proposition \ref{p2}.
\begin{proof}[{Proof of Proposition \ref{p1} by use of Proposition \ref{p2}}]
Observe that the initial data \eqref{300} in Proposition \ref{p1} satisfies
	\begin{equation*}
		\|\bv_0\|_{H^s} + \|\rho_0\|_{H^s}+\|w_0\|_{H_x^{s_0-\frac14}} + \|\nabla w_0\|_{L^8_x}   \leq \epsilon_3.
	\end{equation*}
	We denote by $\text{A}$ the subset of those $\gamma \in [0,1]$ such that the equation \eqref{CS} admits a smooth solution $(\bv_\gamma,\rho_\gamma,w_\gamma)$ with the initial data
	\begin{equation*}
		\begin{split}
			\bv_\gamma|_{t=0}=&\gamma \bv_0,
			\\
			\rho_\gamma|_{t=0}=&\gamma \rho_0,
			\\
			w_\gamma|_{t=0}=&  \mathrm{e}^{-\rho_\gamma(0)}\mathrm{curl}\bv_\gamma(0),
		\end{split}
	\end{equation*}
	and such that $G(\bv_\gamma, \rho_\gamma) \leq \epsilon_1$ and \eqref{404}-\eqref{405} hold. We claim that the set $\text{A}$ is not empty.

	Taking $\gamma=0$, then
	\begin{equation*}
		(\bv_\gamma, \rho_\gamma, w_\gamma)(t,x)=(\mathbf{0},0,0),
	\end{equation*}
	is a smooth solution of \eqref{CS} with the initial data
	\begin{equation*}
		(\bv_\gamma, \rho_\gamma, w_\gamma)|_{t=0}=(\mathbf{0},0,0).
	\end{equation*}
	Thus, $0\in \text{A}$. If we can prove $1 \in \text{A}$, the proposition \ref{p1} follows immediately by applying Proposition \ref{p2}. Next, we will prove that $\text{A} = [0,1]$. Thus, it's suffices for us to prove that $\text{A}$ is both open and closed in $[0,1]$.

  We assume $\gamma \in \text{A}$. Then $(\bv_\gamma, \rho_\gamma, w_\gamma)$ is a smooth solution to \eqref{CS}, where
	\begin{equation*}
		w_\gamma=  \mathrm{e}^{-\rho_\gamma}\mathrm{curl}\bv_\gamma.
	\end{equation*}
	Let $\beta$ be close to $\gamma$. By using the continuity of $G$, we have
	\begin{equation*}
		G(\bv_\beta, \rho_\beta) \leq 2\epsilon_1,
	\end{equation*}
	Similarly, \eqref{401}, \eqref{402a} and \eqref{403} also hold for $(\bv_\beta, \rho_\beta, w_\beta)$. Applying Proposition \ref{p2}, we derive that
	\begin{equation*}
		G(\bv_\beta, \rho_\beta) \leq \epsilon_1.
	\end{equation*}
	In a similar way, the estimates \eqref{404}, \eqref{405} also hold for $(\bv_\beta, \rho_\beta, w_\beta)$. Thus, $\beta \in \text{A}$ and $\text{A}$ is open set.

	Let $\gamma_k \in \text{A}, k \in \mathbb{N}^+$ and $\gamma$ be a limit satisfying $\lim_{k \rightarrow \infty} \gamma_k = \gamma$.
	Then there exists a sequence $\{(\bv_{\gamma_k}, \rho_{\gamma_k}, w_{\gamma_k}) \}_{k \in \mathbb{N}^+}$ being smooth solutions to \eqref{CS} and
	\begin{align*}
		& \|(\bv_{\gamma_k}, \rho_{\gamma_k})\|_{L^\infty_t H_x^{s}}
	+ \| w_{\gamma_k} \|_{L^\infty_t H_x^{s_0-\frac14}}+ \| \nabla w_{\gamma_k} \|_{L^\infty_t L_x^{8}} +  \| d \bv_{\gamma_k}, d \rho_{\gamma_k}\|_{L^4_t C_x^\delta} \leq \epsilon_2.
	\end{align*}
	Then there exists a subsequence of $\{(\bv_{\gamma_k}, \rho_{\gamma_k}, w_{\gamma_k}) \}_{k \in \mathbb{N}^+}$  such that $ (\bv_{\gamma}, \rho_{\gamma}, w_{\gamma}) $ is the limi. Moreover, $ (\bv_{\gamma}, \rho_{\gamma}, w_{\gamma}) $ satisfies
	\begin{equation*}
		\begin{split}
			 \|(\bv_{\gamma}, \rho_{\gamma})\|_{L^\infty_t H_x^{s}}+ 
		 \| w_{\gamma} \|_{L^\infty_t H_x^{s_0-\frac14}}+ \| \nabla w_{\gamma} \|_{L^\infty_t L_x^{8}} + \| d \bv_{\gamma}, d \rho_{\gamma}\|_{L^4_t C_x^\delta} \leq \epsilon_2.
		\end{split}
	\end{equation*}
	Similarly, we also get $G(\bv_\gamma, \rho_\gamma) \leq \epsilon_1$. This implies $\gamma \in \text{A}$. So $\text{A}$ is a closed set.  We therefore complete the proof of Proposition \ref{p1}.
\end{proof}
	It still remains for us to give a definition of $G$ and prove Proposition \ref{p2}. 
\subsection{Definition of the functional}
In order to define $G$, we need a new foliation for the space-time and introduce new norms. Assume $(\bv,\rho,w) \in \mathcal{H}$. Give an extension for $\mathbf{g}$ ($\mathbf{g}$ is defined in \eqref{boldg}), which equals to the Minkowski metric for $t \in [-2, -\frac{3}{2}]$. When no confusion can arise, we still record it $\mathbf{g}$.

For each $\theta \in \mathbb{S}^1$, we consider a foliation of the slice $t=-2$ by taking level sets of the function $r_\theta(-2,x)=\theta \cdot x+2$. Then $\theta \cdot dx-t$ is a null covector field over $t=-2$, and it is co-normal to the level sets of $r_\theta(-2)$. Let $\Gamma_{\theta}$ be the graph of a null covector field given by $dr_{\theta}$. We also define the hypersurface $\Sigma_{\theta,r}$ for $r \in \mathbb{R}$ as the level sets of $r_{\theta}$. Thus, the characteristic hypersurface $\Sigma_{\theta,r}$ is the flowout of the set $\theta \cdot x=r-2$ along the null geodesic flow in the direction $\theta$ at $t=-2$.

We introduce an orthonormal sets of coordinates on $\mathbb{R}^2$ by setting $x_{\theta}=\theta \cdot x$. Let $x'_{\theta}$ be given orthonormal coordinates on the hyperplane prependicular to $\theta$, which then define coordinates on $\mathbb{R}^2$ by projection along $\theta$. Then $(t,x'_{\theta})$ induce the coordinates on $\Sigma_{\theta,r}$, and $\Sigma_{\theta,r}$ is given by
\begin{equation*}
	\Sigma_{\theta,r}=\left\{ (t,x): x_{\theta}-\phi_{\theta, r}=0  \right\}
\end{equation*}
for a smooth function $\phi_{\theta, r}(t,x'_{\theta})$.

We now introduce two norms for functions defined on $[-2,2] \times \mathbb{R}^2$, for $a \geq 1$,
\begin{equation*}\label{d0}
	\begin{split}
		&\vert\kern-0.25ex\vert\kern-0.25ex\vert f\vert\kern-0.25ex\vert\kern-0.25ex\vert_{a, \infty} = \sup_{-2 \leq t \leq 2} \sup_{0 \leq j \leq 1} \| \partial_t^j f(t,\cdot)\|_{H^{a-j}(\mathbb{R}^2)},
		\\
		& \vert\kern-0.25ex\vert\kern-0.25ex\vert  f\vert\kern-0.25ex\vert\kern-0.25ex\vert_{a,2} = \big( \sup_{0 \leq j \leq 1} \int^{2}_{-2} \| \partial_t^j f(t,\cdot)\|^2_{H^{a-j}(\mathbb{R}^2)} dt \big)^{\frac{1}{2}}.
	\end{split}
\end{equation*}
We also denote
\begin{equation*}
	\vert\kern-0.25ex\vert\kern-0.25ex\vert f\vert\kern-0.25ex\vert\kern-0.25ex\vert_{a,2,\Sigma_{\theta,r}}=\vert\kern-0.25ex\vert\kern-0.25ex\vert f|_{\Sigma_{\theta,r}} \vert\kern-0.25ex\vert\kern-0.25ex\vert_{a,2},
\end{equation*}
where the right hand side is the norm of the restriction of $f$ to ${\Sigma_{\theta,r}}$, taken over the $(t,x'_{\theta})$ variables used to parametrise ${\Sigma_{\theta,r}}$. Similarly, the notation
\begin{equation*}
	\|f\|_{H^{a}(\Sigma_{\theta,r})},
\end{equation*}
denotes the $H^{a}(\mathbb{R})$ norm of $f$ restricted to the time $t$ slice of ${\Sigma_{\theta,r}}$ using the $x'_{\theta}$ coordinates on ${\Sigma^t_{\theta,r}}$.

For $\frac74<s_0<s\leq 2$, we now define the functional $G$ as
\begin{equation}\label{500}
	G(\bv, \rho)= \sup_{\theta, r} \vert\kern-0.25ex\vert\kern-0.25ex\vert d \phi_{\theta,r}-dt\vert\kern-0.25ex\vert\kern-0.25ex\vert_{s_0-\frac14,2,{\Sigma_{\theta,r}}}.
\end{equation}
\subsection{Proof of Proposition \ref{p2}} Before our proof, let us introduce two propositions as follows. 
\begin{proposition}\label{r2}
	Assume $\frac74<s_0\leq s \leq 2$ and $ \delta_0 \in (0, s_0-\frac{7}{4})$. 	Let $(\bv, \rho, w) \in \mathcal{H}$ such that $G(\bv, \rho) \leq 2 \epsilon_1$. Then we have

	\begin{equation}\label{G}
		G(\bv, \rho) \lesssim \epsilon_2.
	\end{equation}
	Furthermore, for each $t$ it holds that
	\begin{equation}\label{502}
		\|d \phi_{\theta,r}(t,\cdot)-dt \|_{C^{1,\delta_0}_{x'}} \lesssim \epsilon_2+  \| d {\mathbf{g}}(t,\cdot) \|_{C^{\delta_0}_x(\mathbb{R}^2)}.
	\end{equation}
\end{proposition}
\begin{proposition}\label{r3}
	Assume $\frac74<s_0\leq s \leq 2$. Suppose that $(\bv,\rho, w) \in \mathcal{{H}}$ and $G(\bv,\rho)\leq 2 \epsilon_1$.
	For any $1 \leq r \leq s+1$, and for each $t_0 \in [-2,2]$, the linear, non-homogenous equation
	\begin{equation*}
		\begin{cases}
			& \square_{\mathbf{g}} f=\mathbf{T}\Theta+B, \qquad (t,x) \in [t_0,2]\times \mathbb{R}^2,
			\\
			& f(t,x)|_{t=t_0}=f_0 \in H_x^r(\mathbb{R}^2), 
			\\
			& \partial_t f(t,x)|_{t=t_0}=f_1 \in H_x^{r-1}(\mathbb{R}^2),
		\end{cases}
	\end{equation*}
	admites a solution $f \in C([-2,2],H_x^r) \times C^1([-2,2],H_x^{r-1})$ and the following estimates holds:
	\begin{equation}\label{lw0}
		\| f\|_{L_t^\infty H_x^r}+ \|\partial_t f\|_{L_t^\infty H_x^{r-1}} \lesssim \|f_0\|_{H_x^r}+ \|f_1\|_{H_x^{r-1}}+\| \Theta \|_{L^\infty_tH_x^{r-1}\cap L^1_tH_x^r}+\|B\|_{L^1_tH_x^{r-1}}.
	\end{equation}
	Additionally, the following estimates hold, provided $a<r-\frac34$,
	\begin{equation}\label{lw1}
		\| \left<\nabla \right>^a f\|_{L^4_{t}L^\infty_x} \lesssim \|f_0\|_{H_x^r}+ \|f_1\|_{H_x^{r-1}}+\| \Theta \|_{L^\infty_tH_x^{r-1}\cap L^1_tH_x^r}+\|B\|_{L^1_tH_x^{r-1}}.
	\end{equation}
	The similar estimate \eqref{lw1} holds if we replace $\left<\nabla \right>^a$ by $\left<\nabla \right>^{a-1}d$.
\end{proposition}
Based on Propositions \ref{r2} and \ref{r3}, we are able to prove Proposition \ref{p2}.
\begin{proof}[Proof of Proposition \ref{p2}]
	By using Propositions \ref{r2} and \ref{r3}, we only need to verify \eqref{404} and \eqref{405}. Due to Theorem \ref{be} and \eqref{401}, we have
	\begin{equation*}\label{k00}
		\|\bv\|_{L^\infty_{[-2,2]} H_x^{s}}+ \|\rho\|_{L^\infty_{[-2,2]} H_x^{s}} + \|w\|_{L^\infty_{[-2,2]} H_x^{s_0-\frac14}} + \|\nabla w \|_{L^\infty_{[-2,2]} L^8_x} \lesssim \epsilon_3.
	\end{equation*}
	Note that $\epsilon_3 \ll \epsilon_2$. It implies that the estimate \eqref{404} holds. Applying Proposition \ref{r3} to \eqref{fcp}, for $\delta\in (0,s-\frac74)$, we can obtain
		\begin{equation}\label{k02}
		\|d \bv_{+}\|_{L^4_{[-2,2]} C^\delta_x}  \lesssim  \|\bv_0\|_{H^s}+ \|\mathbf{T}\bv_{-}\|_{L^\infty_tH_x^{s-1} }+ \|\mathbf{T}\bv_{-}\|_{L^1_tH_x^{s}} + \|\bQ\|_{L^1_t H_x^{s-1}} .
	\end{equation}
	Similarly, we also have
		\begin{equation}\label{ka}
		\|d \rho \|_{L^4_{[-2,2]} C^\delta_x} \lesssim  \|\rho_0\|_{H^s}+  \|\bv,\rho\|_{L^\infty_tH_x^s} \|d\bv,d\rho\|_{L^4_t L^\infty_x}.
	\end{equation}
	Using \eqref{YYE} and \eqref{eta}, the inequality \eqref{k02} becomes
	\begin{equation}\label{k04}
		\|d \bv_{+}\|_{L^4_{[-2,2]} C^\delta_x}  \lesssim \|\bv,\rho\|_{L^\infty_tH_x^s}  \|\bw\|_{L^\infty_t H_x^{s_0-\frac14}} + \|\bv,\rho\|_{L^\infty_tH_x^s} \|d\bv,d\rho\|_{L^4_t L^\infty_x}.
	\end{equation}
	The estimate \eqref{k04} combines with \eqref{403}-\eqref{404} yield 
		\begin{equation}\label{k06}
		\|d \bv_{+}\|_{L^4_{[-2,2]} C^\delta_x}  \lesssim \epsilon_3.
	\end{equation}
	Due to \eqref{eee} and \eqref{404}, we can obtain
	\begin{equation}\label{k08}
		\|d \bv_{-}\|_{L^4_{[-2,2]} C^\delta_x}  \lesssim \|\rho\|_{L^\infty_{[-2,2]} H_x^{s}} ( 1+ \|w\|_{L^\infty_{[-2,2]} H_x^{s_0-\frac14}} ) \lesssim \epsilon_3.
	\end{equation}
	Adding \eqref{k06} and \eqref{k08}, it implies
		\begin{equation}\label{k10}
		\|d \bv \|_{L^4_{[-2,2]} C^\delta_x}  \lesssim  \epsilon_3.
	\end{equation}
	In a similar way, using \eqref{ka}, and combining with \eqref{403}-\eqref{404}, we get
	\begin{equation}\label{k12}
		\|d \rho \|_{L^4_{[-2,2]} C^\delta_x}  \lesssim  \epsilon_3.
	\end{equation}
	Noting $\epsilon_3 \ll \epsilon_2$, and seeing \eqref{k06}, \eqref{k10}, and \eqref{k12}, we can obtain \eqref{405}. Therefore, we have finished the proof of Proposition \ref{p2}.
\end{proof}

It still remains for us to prove Proposition \ref{r2} and \ref{r3}, which will be presented in Section \ref{sec6} and \ref{sec7} respectively.

\section{The proof of Proposition \ref{r2} }\label{sec6}
To prove Proposition \ref{r2}, the key relies on discussing the regularity of the characteristic hypersurface. We begin by introducing the following result:
\begin{proposition}\label{r1}
Assume $\frac74<s_0\leq s \leq 2$. 	Let $(\bv, \rho, w) \in \mathcal{H}$ so that $G(\bv, \rho) \leq 2 \epsilon_1$. Then
	\begin{equation}\label{501}
		\vert\kern-0.25ex\vert\kern-0.25ex\vert  {\mathbf{g}}^{\alpha \beta}-\mathbf{m}^{\alpha \beta} \vert\kern-0.25ex\vert\kern-0.25ex\vert_{s_0-\frac14,2,{\Sigma_{\theta,r}}} + \vert\kern-0.25ex\vert\kern-0.25ex\vert \lambda ({\mathbf{g}}^{\alpha \beta}-P_{<\lambda} {\mathbf{g}}^{\alpha \beta}), d P_{ <\lambda } {\mathbf{g}}^{\alpha \beta}, \lambda^{-1} \nabla P_{<\lambda} d {\mathbf{g}}^{\alpha \beta}  \vert\kern-0.25ex\vert\kern-0.25ex\vert_{s_0-\frac54,2,{\Sigma_{\theta,r}}} \lesssim \epsilon_2.
	\end{equation}
\end{proposition}

%
%
To prove Proposition \ref{r1}, including \ref{r2}, it suffices to restrict our attention to the case where $\theta=(0,1)$ and $r=0$. We fix this choice, and suppress $\theta$ and $r$ in our notation. We use $(x_2, x')$ instead of $(x_{\theta}, x'_{\theta})$. Then $\Sigma$ is defined by
\begin{equation*}
	\Sigma=\left\{ x_2- \phi(t,x')=0 \right\}.
\end{equation*}
Let $ \Delta_{x'}$ be a Laplacian operator on $\Sigma$, and
\begin{equation}\label{Sig}
	\Lambda_{x'}= (- \Delta_{x'})^{\frac12}.
\end{equation}
The hypothesis $G(\bv,\rho) \leq 2 \epsilon_1$ implies that
\begin{equation}\label{503}
	\vert\kern-0.25ex\vert\kern-0.25ex\vert d \phi_{\theta,r}(t,\cdot)-dt \vert\kern-0.25ex\vert\kern-0.25ex\vert_{s_0-\frac14,2, \Sigma} \leq 2 \epsilon_1.
\end{equation}
By using \eqref{503} and Sobolev imbedding, we have
\begin{equation}\label{504}
	\|d \phi(t,x')-dt \|_{L^4_t C^{1,\delta_0}_{x'}} + \| \partial_t d \phi(t,x')\|_{L^4_t C^{\delta_0}_{x'}} \lesssim \epsilon_1,
\end{equation}
where we use $s_0>\frac74$ and $\delta_0 \in (0,s_0-\frac74)$. Noting that $(\bv, \rho, w) \in \mathcal{H}$, we can obtain
\begin{equation}\label{5021}
	\vert\kern-0.25ex\vert\kern-0.25ex\vert \bv \vert\kern-0.25ex\vert\kern-0.25ex\vert_{s,\infty}+\vert\kern-0.25ex\vert\kern-0.25ex\vert \rho\vert\kern-0.25ex\vert\kern-0.25ex\vert_{s,\infty}+\vert\kern-0.25ex\vert\kern-0.25ex\vert w \vert\kern-0.25ex\vert\kern-0.25ex\vert_{s_0-\frac14,\infty} +\|d\bv,d\rho, d\bv_{+}\|_{L^4_tC^\delta_x} \lesssim \epsilon_2.
\end{equation}
Next, we will prove Proposition \ref{r1}.
\subsection{Characteristic energy estimates and the proof of Proposition \ref{r1}}
To start, let us introduce two lemmas about characteristic energy estimates, which are referred to in \cite{ST}.
\begin{Lemma}\label{te0}(\cite{ST}, Lemma 5.5)
Assume $s\in (\frac74,2]$. Let $\tilde{f}(t,x)=f(t,x',x_2+\phi(t,x'))$. Then we have
	\begin{equation*}
		\vert\kern-0.25ex\vert\kern-0.25ex\vert \tilde{f}\vert\kern-0.25ex\vert\kern-0.25ex\vert_{s-\frac14,\infty}\lesssim \vert\kern-0.25ex\vert\kern-0.25ex\vert f\vert\kern-0.25ex\vert\kern-0.25ex\vert_{s-\frac14,\infty}, \quad \|d\tilde{f}\|_{L^4_tL_x^\infty}\lesssim \|d f\|_{{L^4_tL_x^\infty}}, \quad  \|\tilde{f}\|_{H^{s-\frac14}_{x}}\lesssim \|f\|_{H^{s-\frac14}_{x}}.
	\end{equation*}
\end{Lemma}
\begin{Lemma}\label{te2}(\cite{ST}, Lemma 5.4)
	For $r\geq 1$, we have
	\begin{equation*}
		\begin{split}
			\sup_{t\in[-2,2]} \| f\|_{H^{r-\frac{1}{2}}(\mathbb{R}^2)} & \lesssim \vert\kern-0.25ex\vert\kern-0.25ex\vert f \vert\kern-0.25ex\vert\kern-0.25ex\vert_{r,2},
			\\
			\sup_{t\in[-2,2]} \| f\|_{H^{r-\frac{1}{2}}(\Sigma^t)} & \lesssim \vert\kern-0.25ex\vert\kern-0.25ex\vert f \vert\kern-0.25ex\vert\kern-0.25ex\vert_{r,2,\Sigma}.
		\end{split}
	\end{equation*}
	If $r> \frac{3}{2}$, then
	\begin{equation*}
		\vert\kern-0.25ex\vert\kern-0.25ex\vert f_1f_2\vert\kern-0.25ex\vert\kern-0.25ex\vert_{r,2}\lesssim  \vert\kern-0.25ex\vert\kern-0.25ex\vert f_2 \vert\kern-0.25ex\vert\kern-0.25ex\vert_{r,2} \vert\kern-0.25ex\vert\kern-0.25ex\vert f_1\vert\kern-0.25ex\vert\kern-0.25ex\vert_{r,2}.
	\end{equation*}
	Similarly, if $r>1$, then
	\begin{equation*}
		\vert\kern-0.25ex\vert\kern-0.25ex\vert f_1f_2\vert\kern-0.25ex\vert\kern-0.25ex\vert_{r,2,\Sigma}\lesssim  \vert\kern-0.25ex\vert\kern-0.25ex\vert f_2 \vert\kern-0.25ex\vert\kern-0.25ex\vert_{r,2,\Sigma} \vert\kern-0.25ex\vert\kern-0.25ex\vert f_1 \vert\kern-0.25ex\vert\kern-0.25ex\vert_{r,2,\Sigma}.
	\end{equation*}
\end{Lemma}
Next, we prove two characteristic energy estimates for the hyperbolic system \eqref{sq}.
\begin{Lemma}\label{te1}
Assume $\frac74<s_0\leq s \leq 2$. Suppose that $\bU=(v^1,v^2,p(\mathrm{e}^{\rho}))^{\mathrm{T}}$ satisfies the hyperbolic symmetric system
	\begin{equation}\label{505}
		A^0(\bU) \partial_t\bU+ \sum^{2}_{i=1}A^i(\bU) \partial_i \bU= 0.
	\end{equation}
	Then
	\begin{equation}\label{te10}
		\begin{split}
			\vert\kern-0.25ex\vert\kern-0.25ex\vert  \bU\vert\kern-0.25ex\vert\kern-0.25ex\vert_{s_0-\frac14,2,\Sigma} & \lesssim \|d \bU \|_{L^4_t L^{\infty}_x}+ \| \bU\|_{L^{\infty}_tH_x^{s_0-\frac14}}.
		\end{split}
	\end{equation}
\end{Lemma}

\begin{proof}
	By a change of coordinates $x_2 \rightarrow x_2-\phi(t,x')$ and setting $\tilde{\bU}(t,x)=U(t,x',x_2+\phi(t,x'))$, the system \eqref{505} is transformed to
	\begin{equation*}
		A^0(\tilde{\bU}) \partial_t \tilde{\bU}+ \sum_{i=1}^2 A^i(\tilde{\bU}) \partial_{i} \tilde{\bU}= - \partial_t \phi  \partial_2 \tilde{\bU} - \sum_{\alpha=0}^2 A^\alpha(\tilde{\bU}) \partial_{\alpha}\phi \partial_\alpha \tilde{\bU}.
	\end{equation*}
	For $\phi$ is independent of $x_2$, we further get
	\begin{equation}\label{U}
		A^0(\tilde{\bU}) \partial_t \tilde{\bU}+ \sum_{i=1}^2 A^i(\tilde{\bU}) \partial_{i} \tilde{\bU}= - \partial_t \phi  \partial_2 \tilde{\bU} - \sum_{\alpha=0}^2 A^\alpha(\tilde{\bU}) \partial_{\alpha}\phi \partial_\alpha \tilde{\bU}.
	\end{equation}
	To prove \eqref{te10}, we first establish the $0$-order estimate. A direct calculation on$[-2,2]\times \mathbb{R}^2$ shows that
	\begin{equation*}
		\begin{split}
			\vert\kern-0.25ex\vert\kern-0.25ex\vert \tilde{\bU}\vert\kern-0.25ex\vert\kern-0.25ex\vert^2_{0,2,\Sigma} & \lesssim \| d\tilde{\bU} \|_{L^1_t L_x^\infty}\|\tilde{\bU}\|_{L_x^2} + \|\nabla d\phi\|_{L^1_t L_{x'}^\infty}\|\tilde{\bU}\|_{L_x^2}
			\\
			& \lesssim \| d\tilde{\bU} \|_{L^1_t L_x^\infty}\|\tilde{\bU}\|_{L_x^2} + \|\nabla d\phi\|_{L^1_t L_x^\infty}\|\tilde{\bU}\|_{L_x^2}.
		\end{split}
	\end{equation*}
	By using Lemma \ref{te0}, \eqref{504} and \eqref{5021}, we can prove that
	\begin{equation}\label{U0}
		\vert\kern-0.25ex\vert\kern-0.25ex\vert \bU\vert\kern-0.25ex\vert\kern-0.25ex\vert_{0,2,\Sigma} \lesssim \|d \bU \|_{L^2_t L^{\infty}_x}+ \| \bU\|_{L^{\infty}_tL_x^2}.
	\end{equation}
	Next, we will prove the $(s_0-\frac14)$-order estimate. Operating the derivative of $\partial^{\beta}_{x'}$($1 \leq |\beta| \leq s_0-\frac14$) on \eqref{U} and integrating it on $[-2,2]\times \mathbb{R}^2$, we get
	\begin{equation}\label{U1}
		\begin{split}
			\| \partial^{\beta}_{x'} \tilde{\bU}\|^2_{L^2_{\Sigma}} & \lesssim \| d \tilde{\bU} \|_{L^1_t L^{\infty}_x} \| \partial^{\beta}_{x} \tilde{\bU}\|_{L^{\infty}_tL_x^2} +|I_1|+|I_2|,
		\end{split}
	\end{equation}
	where
	\begin{equation*}
		\begin{split}
			&I_1= -\int_{-2}^2\int_{\mathbb{R}^2} \partial^{\beta}_{x'}  \big( \partial_t \phi  \partial_3 \tilde{\bU} \big) \cdot \Lambda^{\beta}_{x'} \tilde{\bU}  dxd\tau,
			\\
			& I_2= -\sum^2_{\alpha=0}\int_{-2}^2\int_{\mathbb{R}^2} \partial^{\beta}_{x'} \big( A^\alpha(\tilde{\bU}) \partial_{\alpha}\phi \partial_\alpha \tilde{\bU} \big) \cdot \partial^{\beta}_{x'} \tilde{\bU}   dxd\tau.
		\end{split}
	\end{equation*}
	Write $I_1$ as
	\begin{equation*}
		\begin{split}
			I_1 =& -\int_{-2}^2\int_{\mathbb{R}^2} \big( \partial^{\beta}_{x'}(\partial_t \phi  \partial_3 \tilde{\bU})-\partial_t \phi \partial_3 \partial^{\beta}_{x'} \tilde{\bU} \big)   \partial^{\beta}_{x'} \tilde{\bU}dxd\tau
			\\
			& + \int_{-2}^2\int_{\mathbb{R}^2} \partial_t \phi \cdot \partial_3 \partial^{\beta}_{x'} \tilde{\bU} \cdot \partial^{\beta}_{x'} \tilde{\bU}  dx d\tau,
			\\
			=& -\int_{-2}^2\int_{\mathbb{R}^2} [ \partial^{\beta}_{x'}, \partial_t \phi  \partial_3] \tilde{\bU} \cdot \partial^{\beta}_{x'} \tilde{\bU}dxd\tau
		\end{split}
	\end{equation*}
	Similarly, $I_2$ can be written by
	\begin{equation*}
		\begin{split}
			I_2 = & -\sum^2_{\alpha=0}\int_{-2}^2\int_{\mathbb{R}^2} \big( \partial^{\beta}_{x'} \big( A^\alpha(\tilde{\bU}) \partial_{\alpha}\phi \partial_\alpha \tilde{\bU}) -  A^\alpha(\tilde{\bU}) \partial_{\alpha}\phi \partial_\alpha \partial^{\beta}_{x'} \tilde{\bU} \big) \cdot \partial^{\beta}_{x'} \tilde{\bU}  dxd\tau
			\\
			& +\sum^2_{\alpha=0} \int_{-2}^2\int_{\mathbb{R}^2} \big( A^\alpha(\tilde{\bU}) \partial_{\alpha}\phi \big) \cdot \partial_\alpha(\partial^{\beta}_{x'} \tilde{\bU}) \cdot \partial^{\beta}_{x'} \tilde{\bU}dxd\tau.
		\end{split}
	\end{equation*}
	By using Lemma \ref{jiaohuan}, we infer that
	\begin{equation}\label{U2}
		\begin{split}
			|I_1|
			& \lesssim  \big( \|\partial^\beta \tilde \bU\|_{L^{\infty}_tL_x^2} \| d \partial_t \phi \|_{L^1_tL_{x}^\infty}+ \sup_{\theta, r} \|\partial^{\beta}_{x'} \partial_t \phi\|_{L^2(\Sigma_{\theta,r})} \| d \tilde{\bU}\|_{L^1_tL_x^\infty} \big)\cdot\|\partial^\beta \tilde \bU\|_{L^\infty_tL_x^2}
		\end{split}
	\end{equation}
	and
	\begin{equation}\label{U3}
		\begin{split}
			|I_2| \lesssim  & \big(\| \partial^{\beta}_{x'} \tilde{\bU} \|_{L^2_tL_x^2} \|d \nabla \phi\|_{L^1_t L^\infty_x} + \|d\tilde{\bU}\|_{L^1_tL_x^\infty} \sup_{\theta,r}\|\partial^{\beta}_{x'}d \phi\|_{L^2(\Sigma_{\theta,r})}  \big) \cdot \|\partial^{\beta}_{x'} \tilde{\bU} \|_{L^\infty_tL_x^2}
			\\
			& + \big( \| d \tilde{\bU} \|_{L^2_t L_x^\infty} \| \nabla \phi\| _{L^2_tL_x^\infty}+ \|\tilde{\bU}\|_{L^2_t L_x^\infty} \|\nabla^2\phi\|_{L^2_tL_{x}^\infty} \big)  \cdot \|\partial^\beta \tilde{\bU} \|^2_{L^\infty_t L_x^2}.
		\end{split}
	\end{equation}
	Take sum of $1\leq \beta \leq s_0-\frac14$ on \eqref{U1}. By using Lemma \ref{te0}, \eqref{U2}, \eqref{U3}, \eqref{5021}, and \eqref{504}, we obtain
	\begin{equation}\label{U4}
		\begin{split}
			\vert\kern-0.25ex\vert\kern-0.25ex\vert \partial_{x'} \bU\vert\kern-0.25ex\vert\kern-0.25ex\vert_{s_0-\frac54,2,\Sigma} & \lesssim \|d \bU \|_{L^4_t L^{\infty}_x}+\|d \bU\|_{L^{\infty}_tH_x^{s_0-1}}.
		\end{split}
	\end{equation}
	Operating $\nabla$ on \eqref{505}, we have
	\begin{equation*}
		\begin{split}
			A^0(\bU) \partial_t(\nabla \bU) + \sum^2_{i=1} A^i(\bU) \partial_i(\nabla \bU)=-\sum^2_{\alpha=0}\nabla (A^\alpha(\bU)) \partial_\alpha \bU,
		\end{split}
	\end{equation*}
	In a similar way, we can obtain
	\begin{equation}\label{U40}
		\begin{split}
			\vert\kern-0.25ex\vert\kern-0.25ex\vert \nabla \bU \vert\kern-0.25ex\vert\kern-0.25ex\vert_{s_0-\frac14,2,\Sigma} & \lesssim \|d \bU \|_{L^4_t L^{\infty}_x}+\|d \bU\|_{L^{\infty}_tH_x^{s_0-1}}.
		\end{split}
	\end{equation}
	Thanks to $\partial_t \bU=- \sum^2_{i=1}(A^0)^{-1}A^i(\bU)\partial_i \bU$ and Lemma \ref{te2}, we can carry out
	\begin{equation}\label{U5}
		\begin{split}
			\vert\kern-0.25ex\vert\kern-0.25ex\vert \partial_t \bU \vert\kern-0.25ex\vert\kern-0.25ex\vert_{s_0-\frac14,2,\Sigma}
			& \lesssim \vert\kern-0.25ex\vert\kern-0.25ex\vert \bU \vert\kern-0.25ex\vert\kern-0.25ex\vert_{s_0-\frac14,2,\Sigma} \vert\kern-0.25ex\vert\kern-0.25ex\vert\partial_t \bU\vert\kern-0.25ex\vert\kern-0.25ex\vert_{s_0-\frac14,2,\Sigma}
			\lesssim \|d \bU \|_{L^4_t L^{\infty}_x}.
		\end{split}
	\end{equation}
	The estimate \eqref{U5} together with \eqref{U0}, \eqref{U4}, \eqref{U40} thus gives \eqref{te10}. This conclude the proof of this lemma.
\end{proof}
\begin{Lemma}\label{fre}
	Assume $\frac74<s_0\leq s \leq 2$. Let $\bU$ satisfy the assumption in Lemma \ref{te1}. Then
	\begin{equation}\label{508}
		\vert\kern-0.25ex\vert\kern-0.25ex\vert  \lambda(\bU-P_{<\lambda} \bU), d P_{<\lambda} \bU, \lambda^{-1} d \nabla P_{<\lambda} \bU \vert\kern-0.25ex\vert\kern-0.25ex\vert_{s_0-\frac{5}{4},2,\Sigma} \lesssim \epsilon_2. 
	\end{equation}
\end{Lemma}
\begin{proof}
	Let $\Delta_0$ be a standard multiplier of order $0$ on $\mathbb{R}^2$, such that $\Delta_0$ is additionally bounded on $L^\infty(\mathbb{R}^2)$. Clearly, we have
	\begin{equation*}
		A^0(\bU) \partial_t(\Delta_0U)+ \sum^2_{i=1}A^i(\bU) \partial_i(\Delta_0 \bU)= -[\Delta_0, A^\alpha(\bU)]\partial_{\alpha}\bU.
	\end{equation*}
	Applying Lemma \ref{te1}, we derive
	\begin{equation}\label{60}
		\vert\kern-0.25ex\vert\kern-0.25ex\vert \Delta_0U\vert\kern-0.25ex\vert\kern-0.25ex\vert^2_{s_0-\frac{1}{4},2,\Sigma} \lesssim \|d \bU \|_{L^4_t L^{\infty}_x}\| \Delta_0 \bU\|^2_{L^{\infty}_tH_x^{s_0}}+\| [\Delta_0, A^\alpha(\bU)]\partial_{\alpha}\bU \|_{L^{2}_tH_x^{s_0-\frac{1}{4}}}\| \Delta_0 \bU\|_{L^{\infty}_tH_x^{s_0-\frac{1}{4}}}.
	\end{equation}
	Due to commutator estimates, we obtain
	\begin{equation*}
		\| [\Delta_0, A^\alpha(\bU)]\partial_{\alpha}\bU \|_{L^{2}_tH_x^{s_0-\frac{1}{4}}} \lesssim \|d \bU \|_{L^4_t L^{\infty}_x}\| \Delta_0 \bU\|_{L^{\infty}_tH_x^{s_0-\frac{1}{4}}}.
	\end{equation*}
	According to the inequality \eqref{60}, it turns out
	\begin{equation}\label{60e}
		\vert\kern-0.25ex\vert\kern-0.25ex\vert \Delta_0 \bU \vert\kern-0.25ex\vert\kern-0.25ex\vert^2_{s_0-\frac{1}{4},2,\Sigma} \lesssim \|d \bU \|_{L^4_t L^{\infty}_x}\| \Delta_0 \bU\|^2_{L^{\infty}_tH_x^{s_0-\frac{1}{4}}}.
	\end{equation}
	To get a bound for $\lambda(\bU-P_{<\lambda} \bU)$, we write
	\begin{equation*}
		\lambda(\bU-P_{<\lambda} \bU)= \lambda \sum_{\mu\geq \lambda}P_\mu \bU,
	\end{equation*}
	where $P_\mu \bU$ satisfies the above conditions for $\Delta_0 \bU$. Applying \eqref{60e} and replacing $s_0-\frac{1}{4}$ to $s_0-\frac{5}{4}$, we get
	\begin{equation*}
		\begin{split}
			\vert\kern-0.25ex\vert\kern-0.25ex\vert \lambda(\bU-P_{<\lambda} \bU) \vert\kern-0.25ex\vert\kern-0.25ex\vert^2_{s_0-\frac{5}{4},2,\Sigma}
		 =	& \sum_{\mu\geq \lambda} \vert\kern-0.25ex\vert\kern-0.25ex\vert \mu P_\mu \bU \cdot \frac{\lambda}{\mu}  \vert\kern-0.25ex\vert\kern-0.25ex\vert^2_{s_0-\frac{5}{4},2,\Sigma}
			\\
			\lesssim & \| \bU \|^2_{L^{\infty}_tH_x^{s_0-\frac{1}{4}}} \lesssim \epsilon^2_2.
		\end{split}
	\end{equation*}
	Taking square of the above inequality, which yields
	\begin{equation*}
		\vert\kern-0.25ex\vert\kern-0.25ex\vert \lambda(\bU-P_{<\lambda} \bU) \vert\kern-0.25ex\vert\kern-0.25ex\vert_{s_0-\frac{5}{4},2,\Sigma}
		\lesssim  \epsilon_2.
	\end{equation*}
	Finally, applying \eqref{60} to $\Delta_0=P_{<\lambda}$ and $\Delta_0=\lambda^{-1}\nabla P_{<\lambda}$ shows that
	\begin{equation*}
		\vert\kern-0.25ex\vert\kern-0.25ex\vert d P_{<\lambda} \bU\vert\kern-0.25ex\vert\kern-0.25ex\vert_{s_0-\frac{5}{4},2,\Sigma} +\vert\kern-0.25ex\vert\kern-0.25ex\vert \lambda^{-1} d \nabla P_{<\lambda} \bU \vert\kern-0.25ex\vert\kern-0.25ex\vert_{s_0-\frac{5}{4},2,\Sigma} \lesssim \epsilon_2.
	\end{equation*}
	Therefore, the estimate \eqref{508} holds. We have finished the proof of this lemma.
\end{proof}
Taking advantage of Lemma \ref{te1}, inequalities \eqref{402a} and \eqref{403}, we can directly obtain:
 \begin{corollary}\label{vte}
Assume $\frac74<s_0\leq s \leq 2$. 	Suppose $(\bv, \rho, w) \in \mathcal{H}$. Then the following estimate holds:
 	\begin{equation*}\label{e017}
 		\vert\kern-0.25ex\vert\kern-0.25ex\vert \bv \vert\kern-0.25ex\vert\kern-0.25ex\vert_{s_0-\frac14,2,\Sigma}+ \vert\kern-0.25ex\vert\kern-0.25ex\vert \rho \vert\kern-0.25ex\vert\kern-0.25ex\vert_{s_0-\frac14,2,\Sigma}  \lesssim \epsilon_2.
 	\end{equation*}
 \end{corollary}
Next, let us prove the characteristic energy estimates for vorticity.
\begin{Lemma}\label{te20}
Assume $\frac74<s_0\leq s \leq 2$. 	Suppose that $(\bv, \rho, w) \in \mathcal{H}$. Then we have
	\begin{equation}\label{te201}
		\begin{split}
			\vert\kern-0.25ex\vert\kern-0.25ex\vert  w  \vert\kern-0.25ex\vert\kern-0.25ex\vert_{s_0-\frac14,2,\Sigma}+
			\vert\kern-0.25ex\vert\kern-0.25ex\vert  \nabla w  \vert\kern-0.25ex\vert\kern-0.25ex\vert_{s_0-\frac54,2,\Sigma} \lesssim  \epsilon_2.
		\end{split}
	\end{equation}
\end{Lemma}
\begin{proof}
By a change of coordinates $x_2 \rightarrow x_2-\phi(t,x')$ and setting $\tilde{w}(t,x)= w(t,x',x_2+\phi(t,x'))$, the third equation in \eqref{fc} is transformed to
\begin{equation*}
	\partial_t \tilde{w}+ \tilde{\bv} \cdot \nabla \tilde{w}= - \partial_t \phi  \partial_2 \tilde{w} -  \tilde{v}^j \partial_{j}\phi \partial_j \tilde{w}.
\end{equation*}
For $\phi$ is independent of $x_2$, we then get
\begin{equation}\label{VW1}
	\partial_t \tilde{w}+ \tilde{\bv} \cdot \nabla \tilde{w}= - \partial_t \phi  \partial_2 \tilde{w} -  \tilde{v}^1 \partial_{1}\phi \partial_1 \tilde{w}.
\end{equation}
multiplying with $\Lambda_{x'}^{\alpha} \tilde{w}$, and integrating on $[0,t] \times \mathbb{R}^2$, we can obtain 
\begin{equation*}\label{VW2}
	\begin{split}
		\| \tilde{w} \|^2_{0,2,\Sigma}\leq & C \|\nabla \bv\|_{L^4_t L^\infty_x}\| \tilde{w} \|^2_{L^\infty_t L_x^{2}}+ C\| d \phi -dt\|_{s_0-\frac14,2,\Sigma} \| \tilde{w} \|^2_{L^\infty_t L_x^{2}}.
	\end{split}
\end{equation*}
Similarly, using \eqref{w1}, we can obtain 
\begin{equation*}\label{VW2a}
	\begin{split}
	\| \nabla {w} \|^2_{s_0-\frac54,2,\Sigma}\leq & C \|\nabla \bv\|_{L^4_t L^\infty_x}\| \tilde{w} \|^2_{L^\infty_t H_x^{s_0-\frac14}}+ C\| d \phi -dt\|_{s_0-\frac14,2,\Sigma} \| \tilde{w} \|^2_{L^\infty_t H_x^{s_0-\frac14}}
	\\
	&+  C\| d \phi -dt\|_{s_0-\frac14,2,\Sigma}\| \tilde{w} \|_{L^\infty_t H_x^{s_0-\frac14}} 
	{\| \nabla \tilde{w} \|_{L^\infty_t L^8_x} } (1+\|\bv\|_{L^\infty_t H^{s_0}_x}).
\end{split}
\end{equation*}
Taking advantage of inequalities \eqref{401} and \eqref{403}, we find
\begin{equation}\label{VW3}
	\| \nabla {w} \|_{s_0-\frac54,2,\Sigma} \lesssim \epsilon_2,
\end{equation}
and 
\begin{equation}\label{VW3a}
	\|  {w} \|_{0,2,\Sigma} \lesssim \epsilon_2.
\end{equation}
Due to $\partial_{x'} w  = \nabla w \partial_{x'}\phi$, we obtain
\begin{equation}\label{VW4}
	\|\partial_{x'} {w} \|_{s_0-\frac54,2,\Sigma} \lesssim \| \nabla {w} \|_{s_0-\frac54,2,\Sigma}.
\end{equation}
Combining \eqref{VW3}, \eqref{VW3a}, with \eqref{VW4} yields the estimate \eqref{te201}. We complete the proof of this lemma.
\end{proof}
Finally, we are ready to prove Proposition \ref{r1}.
\begin{proof}[Proof of Proposition \ref{r1}]
	From Lemma \ref{fre}, it only remains for us to verify that
	\begin{equation*}
		\begin{split}
			\vert\kern-0.25ex\vert\kern-0.25ex\vert {\mathbf{g}}^{\alpha \beta}-\mathbf{m}^{\alpha \beta}\vert\kern-0.25ex\vert\kern-0.25ex\vert_{s_0-\frac{1}{4},2,\Sigma_{\theta,r}} \lesssim \epsilon_2.
		\end{split}
	\end{equation*}
Due to Lemma \ref{te1}, \eqref{504}, and \eqref{5021}, we get
	\begin{equation*}
		\sup_{\theta,r}\vert\kern-0.25ex\vert\kern-0.25ex\vert \bv \vert\kern-0.25ex\vert\kern-0.25ex\vert_{s_0-\frac{1}{4},2,\Sigma_{\theta,r}}+ \sup_{\theta,r}\vert\kern-0.25ex\vert\kern-0.25ex\vert \rho \vert\kern-0.25ex\vert\kern-0.25ex\vert_{s_0-\frac{1}{4},2,\Sigma_{\theta,r}} \lesssim \epsilon_2.
	\end{equation*}
By using the definition of $\mathbf{g}$ (see \eqref{boldg}), and Lemma \ref{te2}, the following estimate
	\begin{equation*}
		\begin{split}
			\vert\kern-0.25ex\vert\kern-0.25ex\vert  {\mathbf{g}}^{\alpha \beta}-\mathbf{m}^{\alpha \beta}\vert\kern-0.25ex\vert\kern-0.25ex\vert _{s_0-\frac{1}{4},2,\Sigma_{\theta,r}} & \lesssim \vert\kern-0.25ex\vert\kern-0.25ex\vert \bv\vert\kern-0.25ex\vert\kern-0.25ex\vert_{s_0-\frac{1}{4},2,\Sigma_{\theta,r}}+\vert\kern-0.25ex\vert\kern-0.25ex\vert \bv \cdot \bv\vert\kern-0.25ex\vert\kern-0.25ex\vert_{s_0-\frac{1}{4},2,\Sigma_{\theta,r}}+\vert\kern-0.25ex\vert\kern-0.25ex\vert c_s^2-c_s^2(0)\vert\kern-0.25ex\vert\kern-0.25ex\vert_{s_0-\frac{1}{4},2,\Sigma_{\theta,r}}
			\\
			& \lesssim \epsilon_2,
		\end{split}
	\end{equation*}
holds for $s_0>\frac74$. Therefore, the estimate \eqref{501} holds. This concludes the proof of this lemma.
\end{proof}
It also remains for us to obtain Proposition \ref{r2}. To achieve the goal, we need to introduce a new frame on the null hypersurface $\Sigma$.
\subsection{Null frame}
We introduce a null frame along $\Sigma$ as follows. Let
\begin{equation*}
	V=(dr)^*,
\end{equation*}
where $r$ is the defining function of the foliation $\Sigma$, and where $*$ denotes the identification of covectors and vectors induced by $\mathbf{g}$. Then $V$ is the null geodesic flow field tangent to $\Sigma$. Let
\begin{equation*}\label{600}
	\sigma=dt(V), \qquad l=\sigma^{-1} V.
\end{equation*}
Thus $l$ is the $\mathbf{g}$-normal field to $\Sigma$ normalized so that $dt(l)=1$, hence
\begin{equation*}\label{601}
	l=\left< dt,dx_2-d\phi\right>^{-1}_{\mathbf{g}} \left( dx_2-d \phi \right)^*.
\end{equation*}
Therefore, the coefficients $l^j$ are smooth functions of $\bv, \rho$ and $d \phi$. Conversely, we have
\begin{equation}\label{602}
	dx_2-d \phi =\left< l,\partial_{2}\right>^{-1}_{\mathbf{g}} l^*.
\end{equation}
Seeing from \eqref{602}, $d \phi$ is also a smooth function of $\bv, \rho$ and the coefficients of $l$.

Next we introduce the vector fields $e_1$ tangent to the fixed-time slice $\Sigma^t$ of $\Sigma$. We do this by applying Grahm-Schmidt orthogonalization in the metric $\mathbf{g}$ to the $\Sigma^t$-tangent vector fields $\partial_{1}+ \partial_{1} \phi \partial_{2}$.

Finally, we denote
\begin{equation*}
	\underline{l}=l+2\partial_t.
\end{equation*}
It follows that $\{l, \underline{l}, e_1 \}$ form a null frame in the sense that
\begin{align*}
	& \left<l, \underline{l} \right>_{\mathbf{g}} =2, \qquad \qquad \ \ \left< e_1, e_1\right>_{\mathbf{g}}=1,
	\\
	& \left<l, l \right>_{\mathbf{g}} =\left<\underline{l}, \underline{l} \right>_{\mathbf{g}}=0, \quad \left<l, e_1 \right>_{\mathbf{g}}=\left<\underline{l}, e_1 \right>_{\mathbf{g}}=0 .
\end{align*}
The coefficients of each fields is a smooth function of $\bv,\rho$ and $d \phi$, and by assumption we also have the pointwise bound
\begin{equation*}
	| e_1 - \partial_{1} |  + | l- (\partial_t+\partial_{2}) | + | \underline{l} - (-\partial_t+\partial_{2})|  \lesssim \epsilon_1.
\end{equation*}
Based on the above setting, we introduce a result about the decomposition of curvature tensor.
\begin{Lemma}\label{LLQ}(\cite{ST}, Lemma 5.8)
Assume $\frac74<s_0\leq s \leq 2$. Suppose $f$ satisfying $$\mathbf{g}^{\alpha \beta} \partial^2_{\alpha \beta}f=F.$$
	Let $(t,x',\phi(t,x'))$ denote the projective parametrisation of $\Sigma$, and for $0 \leq \alpha, \beta \leq 1$, let $/\kern-0.55em \partial_\alpha$ denote differentiation along $\Sigma$ in the induced coordinates. Then, for $0 \leq \alpha, \beta \leq 1$, one can write
	\begin{equation*}
		/\kern-0.55em \partial_\alpha /\kern-0.55em \partial_\beta (f|_{\Sigma}) = l(f_2)+ f_1,
	\end{equation*}
	where
	\begin{equation*}
		\| f_2 \|_{L^2_t H^{s_0-\frac54}_{x'}(\Sigma)}+\| f_1 \|_{L^1_t H^{s_0-\frac54}_{x'}(\Sigma)} \lesssim \|df\|_{L^\infty_t H_x^{s_0-\frac54}}+ \|df\|_{L^4_t L_x^\infty}+  \| F\|_{L^1_t H^{s_0-\frac54}_{x'}(\Sigma)}.
	\end{equation*}
\end{Lemma}
\begin{corollary}\label{Rfenjie}
Assume $\frac74<s_0\leq s \leq 2$ and $\delta_0\in (0,s_0-\frac74)$. Let $R$ be the Riemann curvature tensor for the metric ${\mathbf{g}}$. Let $e_0=l$. Then for any $0 \leq a, b, c,d \leq 2$, we can write
	\begin{equation}\label{603}
		\left< R(e_a, e_b)e_c, e_d \right>_{\mathbf{g}}|_{\Sigma}=l(f_2)+f_1,
	\end{equation}
	where $|f_1|\lesssim |\nabla w|+ |d \mathbf{g} |^2$ and $|f_2| \lesssim |d {\mathbf{g}}|$. Moreover, the characteristic energy estimates
	\begin{equation}\label{604}
		\|f_2\|_{L^2_t H^{s_0-\frac54}_{x'}(\Sigma)}+\|f_1\|_{L^1_t H^{s_0-\frac54}_{x'}(\Sigma)} \lesssim \epsilon_2,
	\end{equation}
	holds. Additionally, for any $t \in [-2,2]$, it follows
	\begin{equation}\label{605}
		\|f_2(t,\cdot)\|_{C^{\delta_0}_{x'}(\Sigma^t)} \lesssim \|d \mathbf{g}\|_{C^{\delta_0}_x(\mathbb{R}^3)}.
	\end{equation}
\end{corollary}
\begin{proof}
	Due to the definition of curvature tensor, we have
	\begin{equation*}
		\left< R(e_a, e_b)e_c, e_d \right>_{\mathbf{g}}= R_{\alpha \beta \mu \nu}e^\alpha_a e^\beta_b e_c^\mu e_d^\nu,
	\end{equation*}
	where
	\begin{equation*}
		R_{\alpha \beta \mu \nu}= \frac12 \left[ \partial^2_{\alpha \mu} \mathbf{g}_{\beta \nu}+\partial^2_{\beta \nu} \mathbf{g}_{\alpha \mu}-\partial^2_{\beta \mu} \mathbf{g}_{\alpha \nu}-\partial^2_{\alpha \nu} \mathbf{g}_{\beta \mu} \right]+ F(\mathbf{g}^{\alpha \beta}, d \mathbf{g}_{\alpha \beta}),
	\end{equation*}
	where $F$ is a sum of products of coefficients of $\mathbf{g}^{\alpha \beta} $ with quadratic forms in $d \mathbf{g}_{\alpha \beta}$. Due to Lemma \ref{LLQ}, the estimate \eqref{603} holds. Next, we will prove \eqref{604} and \eqref{605}.

	Applying Proposition \ref{r1}, the term $F$ satisfies the bound required of $f_1$. It suffices for us to consider
	\begin{equation*}
		\frac12 e^\alpha_a e^\beta_b e_c^\mu e_d^\nu  \left[ \partial^2_{\alpha \mu} \mathbf{g}_{\beta \nu}+\partial^2_{\beta \nu} \mathbf{g}_{\alpha \mu}-\partial^2_{\beta \mu} \mathbf{g}_{\alpha \nu}-\partial^2_{\alpha \nu} \mathbf{g}_{\beta \mu} \right].
	\end{equation*}
	We therefore look at the term $ e^\alpha_a e_c^\mu \partial^2_{\alpha \mu} \mathbf{g}_{\beta \nu} $, which is typical. By \eqref{503} and Proposition \ref{r1}, we get
	\begin{equation}\label{LLL}
		\vert\kern-0.25ex\vert\kern-0.25ex\vert  l^\alpha - \delta^{\alpha 0} \vert\kern-0.25ex\vert\kern-0.25ex\vert _{s_0-\frac14,2,\Sigma} +\vert\kern-0.25ex\vert\kern-0.25ex\vert  \underline{l}^\alpha + \delta^{\alpha 0}-2\delta^{\alpha 2} \vert\kern-0.25ex\vert\kern-0.25ex\vert _{s_0-\frac14,2,\Sigma} + \vert\kern-0.25ex\vert\kern-0.25ex\vert  e^\alpha_1- \delta^{\alpha 1} \vert\kern-0.25ex\vert\kern-0.25ex\vert _{s_0-\frac14,2,\Sigma} \lesssim \epsilon_1.
	\end{equation}
	By using \eqref{LLL} and Proposition \ref{r1}, the term $ e_a (e_c^\mu) \partial_{ \mu} \mathbf{g}_{\beta \nu}$ satisfies the bound required of $f_1$, we therefore consider $e_a(e_c(\mathbf{g}_{\beta \nu}))$. Since the coefficients of $e_c$ in the basis $/\kern-0.55em \partial_\alpha$ have tangential derivatives bounded in $L^2_tH^{s_0-\frac54}_{x'}(\Sigma)$, we are reduced by Lemma \ref{LLQ} to verifying that
	\begin{equation*}
		\| \mathbf{g}^{\alpha \beta } \partial^2_{\alpha \beta} \mathbf{g}_{\mu \nu} \|_{L^1_t H^{s_0-\frac54}_{x'}(\Sigma)} \lesssim \epsilon_2.
	\end{equation*}
	Note $\mathbf{g}^{\alpha \beta } \partial^2_{\alpha \beta} \mathbf{g}_{\mu \nu}= \square_{\mathbf{g}} \mathbf{g}_{\mu \nu}$. By use of Corollary \ref{vte} and Lemma \ref{te20}, we have
	\begin{equation*}
		\begin{split}
			\| \square_{\mathbf{g}} \mathbf{g}_{\mu \nu}\|_{L^1_t H^{s_0-\frac54}_{x'}(\Sigma)}
			\lesssim & \ \| \square_{\mathbf{g}} \bv\|_{L^2_t H^{s_0-\frac54}_{x'}(\Sigma)}+\| \square_{\mathbf{g}} \rho \|_{L^2_t H^{s_0-\frac54}_{x'}(\Sigma)}
			\\
			\lesssim & \ \| \nabla w \|_{L^2_t H^{s_0-\frac54}_{x'}(\Sigma)}+ \| (d\bv,dg)\cdot (d\bv,dg) \|_{L^1_t H^{s_0-\frac54}_{x'}(\Sigma)}
			\\
			\lesssim & \ \| \nabla w \|_{L^2_t H^{s_0-\frac54}_{x'}(\Sigma)}+ [ 2-(-2)  ]^{\frac14}\| d\bv, d\rho \|_{L^4_t L^\infty_x }\| d\bv, d\rho\|_{L^2_t H^{s_0-\frac54}_{x'}(\Sigma)}
			\\
			\lesssim & \ \epsilon_2.
		\end{split}
	\end{equation*}
	Above, $\nabla w$ and $(d{\mathbf{g}})^2$ are included in $f_1$. Therefore, we have finished the proof of this corollary.
\end{proof}
Based on the above null frame, we can discuss the estimates for connection coefficients as follows.
\subsection{Estimates for connection coefficients}
Define
\begin{equation*}
	\chi= \left<D_{e_1}l,e_1 \right>_{\mathbf{g}}, \qquad l(\ln \sigma)=\frac{1}{2}\left<D_{l}\underline{l},l \right>_{\mathbf{g}}. \quad 
\end{equation*}
For $\sigma$, we set the initial data $\sigma=1$ at the time $-2$. Thanks to Proposition \ref{r1}, we have
\begin{equation}\label{606}
	\|\chi \|_{L^2_t H^{s_0-\frac54}_{x'}(\Sigma)} + \| l(\ln \sigma)\|_{L^2_t H^{s_0-\frac54}_{x'}(\Sigma)} \lesssim \epsilon_1.
\end{equation}
In a similar way, if we expand $l=l^\alpha /\kern-0.55em \partial_\alpha$ in the tangent frame $\partial_t, \partial_{x'}$ on $\Sigma$, then
\begin{equation}\label{607}
	l^0=1, \quad \|l^1\|_{s_0-\frac14,2,\Sigma} \lesssim \epsilon_1.
\end{equation}
Next, we will establish some estimates for connections along the hypersurface.
\begin{Lemma}\label{chi}
Assume $\frac74<s_0\leq s \leq 2$ and $ \delta_0 \in (0, s_0-\frac{7}{4})$. Let $\chi$ be defined as before. Then
	\begin{equation}\label{608}
		\|\chi\|_{L^2_t H^{s_0-\frac54}_{x'}(\Sigma)} \lesssim \epsilon_2.
	\end{equation}
	Furthermore, for any $t \in [-2,2]$,
	\begin{equation}\label{609}
		\| \chi \|_{C^{\delta_0}_{x'}(\Sigma^t)} \lesssim \epsilon_2+ \|d \mathbf{g} \|_{C^{\delta_0}_{x}(\mathbb{R}^2)}.
	\end{equation}
\end{Lemma}
\begin{proof}
	Along the null hypersurface, $\chi$ satisfies a transport equation (see Klainerman and Rodnianski \cite{KR2}). This tells us:
	\begin{equation*}
		l(\chi)=\left< R(l,e_1)l, e_1 \right>_{\mathbf{g}}-\chi^2-l(\ln \sigma)\chi.
	\end{equation*}
	Due to Corollary \ref{Rfenjie}, we can rewrite the above equation by
	\begin{equation}\label{610}
		l(\chi-f_2)=f_1-\chi^2-l(\ln \sigma)\chi,
	\end{equation}
	where
	\begin{equation}\label{611}
		\|f_2\|_{L^2_t H^{s_0-\frac54}_{x'}(\Sigma)}+\|f_1\|_{L^1_t H^{s_0-\frac54}_{x'}(\Sigma)} \lesssim \epsilon_2,
	\end{equation}
	and for any $t \in [0,T]$,
	\begin{equation}\label{612}
		\|f_2(t,\cdot)\|_{C^{\delta_0}_{x'}(\Sigma^t)} \lesssim \|d \mathbf{g}\|_{C^{\delta_0}_x(\mathbb{R}^2)}.
	\end{equation}
	Let $\Lambda_{x'}$ be defined in \eqref{Sig}. To be simple, we set
	\begin{equation*}\label{ef}
		F=f_1-\chi^2-l(\ln \sigma)\chi.
	\end{equation*}
Applying integration by parts on \eqref{610}, we obtain
	\begin{equation}\label{613}
		\begin{split}
			 \|\Lambda_{x'}^{s_0-\frac54}(\chi-f_2)(t,\cdot) \|_{L^2_{x'}(\Sigma^t)}
			\lesssim \ &\| [\Lambda_{x'}^{s_0-\frac54},l](\chi-f_2) \|_{L^1_tL^2_{x'}(\Sigma^t)}+ \| \Lambda_{x'}^{s_0-\frac54}F \|_{L^1_tL^2_{x'}(\Sigma^t)}.
		\end{split}
	\end{equation}
By using $s_0>\frac74$, $H^{s_0-\frac54}_{x'}(\Sigma^t)$ is an algebra. This ensures that
	\begin{equation}\label{614}
		\begin{split}
			\| \Lambda_{x'}^{s_0-\frac54} F \|_{L^1_tL^2_{x'}(\Sigma^t)} &\lesssim \|f_1\|_{L^1_tH^{s_0-\frac54}_{x'}(\Sigma^t)}+ \|\chi\|^2_{L^2_tH^{s_0-\frac54}_{x'}(\Sigma^t)}
			\\
			& \quad + \|\chi\|_{L^2_tH^{s_0-\frac54}_{x'}(\Sigma^t)}\cdot\|l(\ln \sigma)\|_{L^2_tH^{s_0-\frac54}_{x'}(\Sigma^t)}
			\\
			& \quad + \|\mu\|_{L^2_tH^{s_0-\frac54}_{x'}(\Sigma^t)}\cdot\|\chi\|_{L^2_tH^{s_0-\frac54}_{x'}(\Sigma^t)}.
		\end{split}
	\end{equation}
For $a \geq 0$, a direct calculation tells us
\begin{equation*}
	\begin{split}
		[\Lambda_{x'}^{a},l]f=& \Lambda_{x'}^{a} (l^{\alpha} /\kern-0.55em \partial_{\alpha}  f ) - l^{\alpha} /\kern-0.55em \partial_{\alpha} \Lambda_{x'}^{a} f
		\\
		=& \Lambda_{x'}^{a} /\kern-0.55em \partial_{\alpha}  ( l^{\alpha} f ) - \Lambda_{x'}^{a} ( /\kern-0.55em \partial_{\alpha} l^\alpha f)- l^\alpha \Lambda_{x'}^{a} /\kern-0.55em \partial_{\alpha} f
		\\
		=&  - \Lambda_{x'}^{a} ( /\kern-0.55em \partial_{\alpha} l^\alpha f) + [\Lambda_{x'}^{a} /\kern-0.55em \partial_{\alpha} , l^{\alpha} ]f.
	\end{split}
\end{equation*}
As a result, we can bound $\| [\Lambda_{x'}^{s_0-\frac54},l](\chi-f_2) \|_{L^2_{x'}(\Sigma^t)}$ by
	\begin{equation}\label{k20}
	\begin{split}
		\| [\Lambda_{x'}^{s_0-\frac54},l](\chi-f_2) \|_{L^2_{x'}(\Sigma^t)} 
		 \leq & \| /\kern-0.55em \partial_{\alpha} l^{\alpha} (\chi-f_2)(t,\cdot) \|_{H^{s_0-\frac54}_{x'}(\Sigma^t)}
		\\
		& + \|[\Lambda_{x'}^{s_0-\frac54} /\kern-0.55em \partial_{\alpha}, l^{\alpha}](\chi-f_2)(t,\cdot) \|_{L^{2}_{x'}(\Sigma^t)}.
	\end{split}	
	\end{equation}
Note $l^0=1$ and $s_0>\frac74$. Applying Kato-Ponce's commutator estimate and Sobolev embeddings, the right hand side of \eqref{k20} can be bounded by
	\begin{equation}\label{615}
		\|l^1(t,\cdot)\|_{H^{s_0-\frac54}_{x'}(\Sigma^t)} \| \Lambda_{x'}^{s_0-\frac54}(\chi-f_2)(t,\cdot) \|_{L^{2}_{x'}(\Sigma^t)} .
	\end{equation}
	Because of \eqref{606}, \eqref{607}, \eqref{611}, \eqref{613}, \eqref{614} and \eqref{615}, we thus prove that
	\begin{equation}\label{k22}
		\sup_t \|(\chi-f_2)(t,\cdot)\|_{H^{s_0-\frac54}_{x'}(\Sigma^t)}  \lesssim \epsilon_2.
	\end{equation}
	The estimate \eqref{k22} combining with \eqref{611} yield \eqref{608}. Due to \eqref{610}, we find
	\begin{equation}\label{616}
		\begin{split}
			\| \chi-f_2\|_{C^{\delta_0}_{x'}} & \lesssim \| f_1 \|_{L^1_tC^{\delta_0}_{x'}}+ \|\chi^2\|_{L^1_tC^{\delta_0}_{x'} }+\|l(\ln \sigma)\chi\|_{L^1_tC^{\delta_0}_{x'}}.
		\end{split}
	\end{equation}
	Using Sobolev's imbedding, we have
	\begin{equation}\label{619}
		H^{s_0-\frac54}(\mathbb{R})\hookrightarrow C^{\delta_0}(\mathbb{R}), \qquad \delta_0 \in (0,s_0-\frac74).
	\end{equation}
Taking advantage of \eqref{612}, \eqref{616}, and \eqref{619}, we get \eqref{609}. At this stage, we have finished the proof of this lemma.
\end{proof}
Finally, we present a proof for Proposition \ref{r2} as follows.
\subsection{The proof of Proposition \ref{r2}}
First, we will verify \eqref{G}. By using \eqref{501} and \eqref{602} for $\vert\kern-0.25ex\vert\kern-0.25ex\vert \mathbf{g}-\mathbf{m} \vert\kern-0.25ex\vert\kern-0.25ex\vert_{s_0-\frac54,2,\Sigma}$, then the inequality \eqref{G} follows from the following bound:
\begin{equation*}
	\vert\kern-0.25ex\vert\kern-0.25ex\vert l-(\partial_t-\partial_{2})\vert\kern-0.25ex\vert\kern-0.25ex\vert_{s_0-\frac14,2,\Sigma} \lesssim \epsilon_2.
\end{equation*}
Above, it is understood that one takes the norm of the coefficients of $l-(\partial_t-\partial_{2})$ in the standard frame on $\mathbb{R}^{1+2}$. The geodesic equation, together with the bound for Christoffel symbols $\|\Gamma^\alpha_{\beta \gamma}\|_{L^4_t L^\infty_x} \lesssim \|d {\mathbf{g}} \|_{L^4_t L^\infty_x}\lesssim \epsilon_2$, imply that
\begin{equation*}
	\|l-(\partial_t-\partial_{2})\|_{L^\infty_{t,x}} \lesssim \epsilon_2.
\end{equation*}
Therefore, it suffices for us to bound the tangential derivatives of the coefficients for $l-(\partial_t-\partial_{2})$ in the norm $L^2_t H^{s_0-\frac54}_{x'}(\Sigma)$. By using Proposition \ref{r1}, we can estimate the Christoffel symbols
\begin{equation*}
	\|\Gamma^\alpha_{\beta \gamma} \|_{L^2_t H^{s_0-\frac54}_{x'}(\Sigma^t)} \lesssim \epsilon_2.
\end{equation*}
Note that $H^{s_0-\frac54}_{x'}(\Sigma^t)$ is a algebra when $s_0>\frac74$. We thus derive
\begin{equation*}
	\|\Gamma^\alpha_{\beta \gamma} e_1^\beta l^\gamma\|_{L^2_t H^{s_0-\frac54}_{x'}(\Sigma^t)} \lesssim \epsilon_2.
\end{equation*}
In what follows, we can establish the following bound:
\begin{equation}\label{fb}
	\| \left< D_{e_1}l, e_1 \right>\|_{L^2_t H^{s_0-\frac54}_{x'}(\Sigma^t)}+ \| \left< D_{e_1}l, \underline{l} \right>\|_{L^2_t H^{s_0-\frac54}_{x'}(\Sigma^t)}+\|\left< D_{l}l, \underline{l} \right>\|_{L^2_t H^{s_0-\frac54}_{x'}(\Sigma^t)} \lesssim \epsilon_2.
\end{equation}
In fact, the first term in \eqref{fb} is $\chi$, which can be bounded by using Lemma \ref{chi}. For the second term in the left side of \eqref{fb}, consider
\begin{equation*}
	\left< D_{e_1}l, \underline{l} \right>=\left< D_{e_1}l, 2\partial_t \right>=-2\left< D_{e_1}\partial_t,l \right>.
\end{equation*}
Then it can be bounded by use of Proposition \ref{r1}. Similarly, using Proposition \ref{r1}, we can obtain the desired estimate for the last term in \eqref{fb}. Therefore, the rest of proof is to obtain
\begin{equation*}
	\| d \phi(t,x')-dt \|_{C^{1,\delta_0}_{x'}(\mathbb{R})}  \lesssim \epsilon_2+ \| d\mathbf{g}(t,\cdot)\|_{C^{\delta_0}_x(\mathbb{R}^2)}.
\end{equation*}
In order to do this, it suffices for us to show that
\begin{equation*}
	\|l(t,\cdot)-(\partial_t-\partial_{x_2})\|_{C^{1,\delta_0}_{x'}(\mathbb{R})} \lesssim \epsilon_2+ \| d \mathbf{g} (t,\cdot)\|_{C^{\delta_0}_x(\mathbb{R}^2)}.
\end{equation*}
It's clear that the coefficients of $e_1$ are small in $C^{\delta_0}_{x'}(\Sigma^t)$ perturbations of their constant coefficient analogs. Thus, it's left for us to prove
\begin{equation*}
	\|\left< D_{e_1}l, e_1 \right>(t,\cdot)\|_{C^{\delta_0}_{x'}(\Sigma^t)}
	+\|\left< D_{e_1}l, \underline{l} \right>(t,\cdot)\|_{C^{\delta_0}_{x'}(\Sigma^t)}  \lesssim \epsilon_2+ \| d\mathbf{g}(t,\cdot)\|_{C^{\delta_0}_x(\mathbb{R}^2)}.
\end{equation*}
Above, the first term is bounded by Lemma \ref{chi}, and the second by using
\begin{equation*}
	\|\left< D_{e_1}\partial_t, l \right>(t,\cdot)\|_{C^{\delta_0}_{x'}(\Sigma^t)} \lesssim  \| d\mathbf{g}(t,\cdot)\|_{C^{\delta_0}_x(\mathbb{R}^2)}.
\end{equation*}
Therefore, \eqref{502} holds. We complete the proof of Proposition \ref{r2}.

\section{The proof of Proposition \ref{r3} and Strichartz estimates}\label{sec7}
In this part, our goal is to give a proof of Proposition \ref{r3} and establish Strichartz estimates of solutions.
\subsection{The proof of Proposition \ref{r3}}
We first introduce Strichartz estimates for linear inhomogeneous wave equations.
\begin{proposition}\label{r5}
Assume $s\in (\frac74, 2]$. Suppose that $(\bv,\rho, w) \in \mathcal{{H}}$ and $G(\bv,\rho)\leq 2 \epsilon_1$. For each $1 \leq r \leq s+1$, then the linear, homogenous equation
	\begin{equation*}
		\begin{cases}
			& \square_{\mathbf{g}} F=0,
			\\
			&F(t,x)|_{t=t_0}=F_0, \quad \partial_t F(t,x)|_{t=t_0}=F_1,
		\end{cases}
	\end{equation*}
	is well-posed for the initial data $(F_0, F_1)$ in $H_x^r \times H_x^{r-1}$. Moreover, the solution satisfies, respectively, the energy estimates
	\begin{equation*}
		\| F\|_{L^\infty_tH_x^r}+ \| \partial_t F\|_{L^\infty_tH_x^{r-1}} \leq C \big( \|F_0\|_{H_x^r}+ \|F_1\|_{H_x^{r-1}} \big),
	\end{equation*}
	and the Strichartz estimates
	\begin{equation}\label{SL}
		\| \left<\nabla \right>^a F\|_{L^4_{t}L^\infty_x} \leq C \big( \|F_0\|_{H_x^r}+ \|F_1\|_{H_x^{r-1}} \big), \quad a<r-\frac34.
	\end{equation}
	The similar estimate \eqref{SL} also holds if we replace $\left<\nabla \right>^a$ by $\left<\nabla \right>^{a-1}d$.
\end{proposition}
\begin{remark}
	The proof of Proposition \ref{r5} will be given in Section \ref{Ap}. 
\end{remark}
Now we are ready to prove Proposition \ref{r3} by use of Proposition \ref{r5}.
\begin{proof}[Proof of Proposition \ref{r3} by using Proposition \ref{r5}]
	Let $\mathbf{g}$ be stated as in Proposition \ref{r3}. Let $V$ satisfy the linear, homogenous equation
	\begin{equation}\label{Vf}
		\begin{cases}
			& \square_{\mathbf{g}} V=0,
			\\
			&V(t,x)|_{t=t_0}=f_0, \quad \mathbf{T} V(t,x)|_{t=t_0}=f_1-\Theta(t_0,x),
		\end{cases}
	\end{equation}
	where $(f_0, f_1-G(t_0,x)) \in H_x^r \times H_x^{r-1}$.
	Let $Q$ satisfy the the linear, nonhomogenous equation
	\begin{equation}\label{Qf}
		\begin{cases}
			& \square_{\mathbf{g}} Q=\mathbf{T}\Theta+B,
			\\
			&Q(t,x)|_{t=t_0}=0, \quad \mathbf{T} Q(t,x)|_{t=t_0}=\Theta(t_0,x),
		\end{cases}
	\end{equation}
	where $\Theta(t_0,x) \in H_x^{r-1}$. Due to $\mathbf{T}=\partial_t+\bv \cdot \nabla$, we can rewrite \eqref{Vf} and \eqref{Qf} as 
	\begin{equation*}
		\begin{cases}
			& \square_{\mathbf{g}} V=0,
			\\
			&V(t,x)|_{t=t_0}=f_0, \quad \partial_t V(t,x)|_{t=t_0}=f_1-\Theta(t_0,x)-\bv\cdot \nabla f_0(t_0,x),
		\end{cases}
	\end{equation*}
	and
	\begin{equation*}
		\begin{cases}
			& \square_{\mathbf{g}} Q=\mathbf{T}\Theta+B,
			\\
			&Q(t,x)|_{t=t_0}=0, \quad \partial_t Q(t,x)|_{t=t_0}=\Theta(t_0,x).
		\end{cases}
	\end{equation*}
Therefore, the sum $f= V+Q$ satisfies
	\begin{equation*}
		\begin{cases}
			& \square_{\mathbf{g}} f=\mathbf{T}\Theta+B, 
			\\
			&f(t,x)|_{t=t_0}=f_0 \in H_x^r(\mathbb{R}^2), 
			\\
			& \mathbf{T} f(t,x)|_{t=t_0}=f_1 \in H_x^{r-1}(\mathbb{R}^2).
		\end{cases}
	\end{equation*}
	Observe that the proof for \eqref{lw0} and \eqref{lw1} relies on $V$ and $Q$. By Proposition \ref{r5}, we can infer
	\begin{equation*}
		\| V\|_{L^\infty_tH_x^r}+ \| \partial_t V\|_{L^\infty_tH_x^{r-1}} \leq C \big( \|f_0\|_{H_x^r}+ \|f_1-\Theta (t_0,\cdot)-(\bv\cdot \nabla) V(t_0,\cdot)\|_{H_x^{r-1}} \big),
	\end{equation*}
	and
	\begin{equation*}
		\| \left<\nabla \right>^a V\|_{L^4_{t}L^\infty_x} \leq C \big( \|f_0\|_{H_x^r}+ \|f_1-\Theta(t_0,\cdot)-(\bv\cdot \nabla) V(t_0,\cdot)\|_{H_x^{r-1}} \big), \quad a<r-\frac34.
	\end{equation*}
Thanks to $s>\frac74$ and $r<s+1$, it turns out
	\begin{equation}\label{VE}
		\begin{split}
			& \| V\|_{L^\infty_tH_x^r}+ \| \partial_t V\|_{L^\infty_tH_x^{r-1}}
			\\
			\leq & C \big( \|f_0\|_{H_x^r}+ \|f_1\|_{H_x^{r-1}}+\|\Theta(t_0,\cdot)\|_{H_x^{r-1}}+\| \bv\|_{H_x^{s_0}} \| \nabla f_0(t_0,\cdot)\|_{H_x^{r-1}} \big)
			\\
			\leq & C \big( \|f_0\|_{H_x^r}+ \|f_1\|_{H_x^{r-1}}+\|\Theta\|_{L^\infty_tH_x^{r-1}}\big),
		\end{split}
	\end{equation}
	and
	\begin{equation}\label{SV}
		\| \left< \nabla \right>^a V\|_{L^4_{t}L^\infty_x} \leq C \big( \|f_0\|_{H_x^r}+ \|f_1\|_{H_x^{r-1}}+\|\Theta\|_{L^\infty_tH_x^{r-1}} \big), \quad a<r-\frac34.
	\end{equation}
	By applying Lemma \ref{LD} and Proposition \ref{r5}, we can also prove
	\begin{equation}\label{QE}
		\| Q\|_{L^\infty_t H_x^r}+ \| \partial_t Q\|_{L^\infty_t H_x^{r-1}} \leq C \big( \|\Theta(t_0,\cdot)\|_{H_x^{r-1}}+ \|\Theta\|_{L^1_tH^{r}}+ \|B\|_{L^1_tH^{r-1}}\big),
	\end{equation}
	and
	\begin{equation}\label{SQ}
		\| \left<\nabla \right>^a Q\|_{L^4_{t}L^\infty_x} \leq C \big( \|\Theta(t_0,\cdot)\|_{H_x^{r-1}}+ \|\Theta\|_{L^1_tH_x^{r}}+ \|B\|_{L^1_tH_x^{r-1}}\big), \quad a<r-1.
	\end{equation}
	Adding \eqref{VE} and \eqref{QE}, we get \eqref{lw0}. Adding \eqref{SV} and \eqref{SQ}, we obtain \eqref{lw1}. Therefore, we complete the proof of Proposition \ref{r3}.
\end{proof}

We can also obtain the following Strichartz estimates of solutions.
\subsection{Strichartz estimates for solutions}
\begin{proposition}\label{r6}
Assume $\frac74<s_0\leq s\leq 2$. Suppose $(\bv, \rho, w) \in \mathcal{H}$ and $G(\bv, \rho)\leq 2 \epsilon_1$. Let $\bv_{+}$ be defined in \eqref{dvc}. Then for $a<s-\frac34$, we have
	\begin{equation}\label{fgh}
		\|\left< \nabla \right>^a \rho, \left< \nabla \right>^a \bv_{+}\|_{L^4_t L^\infty_x}
		\lesssim  \| \rho_0\|_{H^s}+\| \bv_0\|_{H^s}+\| w_0\|_{H^{1+}} .
	\end{equation}
	The similar estimate \eqref{fgh} holds if we replace $\left<\nabla \right>^a$ by $\left<\nabla \right>^{a-1}d$.
\end{proposition}
\begin{proof}
	We note the functions $\rho$ and $\bv_{+}$ satisfy the system
	\begin{equation}\label{fcr}
		\begin{cases}
			& \square_{\mathbf{g}} \rho= \mathcal{D},
			\\
			&\square_{\mathbf{g}} \bv_{+}=\mathbf{T}\mathbf{T} \bv_{-}+ \bQ.
		\end{cases}
	\end{equation}
	By using the Strichartz estimate in Proposition \ref{r5} (taking $r=s$ and $a<r-\frac34$) 
	, we obtain
	\begin{equation}\label{fgH}
		\begin{split}
			& \|\left< \nabla \right>^a \rho, \left< \nabla \right>^a \bv_{+}\|_{L^4_t L^\infty_x} 
			\\
			\lesssim & \| \rho_0 \|_{H^s}+ \| \bv_0\|_{H^s}+  \|\mathbf{T}\bv_{-}\|_{L^\infty_{[-2,2]}H_x^{s-1}}+\|\bQ, \mathbf{T}\bv_{-} \|_{L^1_{[-2,2]}H_x^s}
			\\
			\lesssim & \| \rho_0 \|_{H^s}+ \| \bv_0\|_{H^s} +  ( \|w\|_{L^\infty_{[-2,2]}H_x^{1+}} +  \|d\rho,d\bv \|_{L^4_{[-2,2]}L_x^\infty} ) \|\rho,\bv \|_{L^\infty_{[-2,2]}H_x^s}  .
		\end{split}
	\end{equation}
	Operating $\Lambda_x^{a'}$ on \eqref{w1}, we have
	\begin{equation*}
		\begin{split}
			\mathbf{T} (\Lambda_x^{a'} \nabla w) = & [\Lambda_x^{a'},\mathbf{T} ] \nabla w + \Lambda_x^{a'} (\nabla \bv  \cdot \nabla w)
			\\
			= & [\Lambda_x^{a'}, \bv \cdot \nabla ] \nabla w + \Lambda_x^{a'} (\nabla \bv  \cdot \nabla w).
		\end{split}
	\end{equation*}
By using commutator and product estimates, for $0\leq a'<1.$, we can obtain the following energy:
	\begin{equation*}
		\begin{split}
			\|w\|^2_{L^\infty_{[-2,2]}H_x^{1+a'}}  \lesssim & \|w_0 \|^2_{L^\infty_{[-2,2]}H_x^{1+a'}}+ \int^t_0  \|\nabla \bv\|_{L^\infty_x}\|w\|^2_{H_x^{1+a'}} d\tau 
		 + \int^t_0  \|\nabla \bv\|_{C^{a'}_x}\|w\|^2_{H_x^{1}} d\tau .
		\end{split}
	\end{equation*}
	Due to Gronwall's inequality, we get
	\begin{equation}\label{k24}
		\|w\|_{L^\infty_{[-2,2]}H_x^{1+a'}}  \lesssim \|w_0 \|_{L^\infty_{[-2,2]}H_x^{1+a'}} \exp\left(  \int^t_0 \|d\bv\|_{C^{a'}_x} d\tau \right), \quad 0\leq a'<1.
	\end{equation}
	Due to \eqref{402a}, \eqref{403} and \eqref{k24}, the estimate \eqref{fgH} becomes
	\begin{equation*}
		\begin{split}
			\|\left< \nabla \right>^a \rho, \left< \nabla \right>^a \bv_{+}\|_{L^4_t L^\infty_x} & \lesssim \| \rho_0\|_{H^s}+\| \bv_0\|_{H^s} +\| w_{0}\|_{H^{1+}}, \qquad a<s-\frac34.
		\end{split}
	\end{equation*}
\end{proof}
	The Proposition \ref{r6} states a Strichartz estimate of solutions with a very low regularity of the velocity, density, and vorticity. This also motivates the following proposition.
\begin{proposition}\label{r4}
Assume $\frac74<s_0\leq s\leq 2$ and $ \delta\in (0, s-\frac{7}{4})$. Suppose $(\bv, \rho, w) \in \mathcal{H}$ and $G(\bv, \rho)\leq 2 \epsilon_1$. Let $\bv_{+}$ be defined in \eqref{dvc}. Then we have
	\begin{equation}\label{SR0}
		\begin{split}
			\| d \rho, d \bv\|_{L^2_t C^\delta_x} & \lesssim \| \rho \|_{L^\infty_t H_x^s}+\|\bv\|_{L^\infty_t H_x^s}+ {\| w\|_{L^\infty_t H_x^{1+2\delta}}}.
		\end{split}
	\end{equation}
	Furthermore, the Strichartz estimates
	\begin{equation}\label{strr}
		\|d \bv, d \rho, d \bv_{+}\|_{L^4_t C^\delta_x} \leq \epsilon_2,
	\end{equation}
	and the energy estimates
	\begin{equation}\label{eef}
		\| \bv\|_{L^\infty_tH_x^s} +\|\rho\|_{L^\infty_tH_x^s}+\| w \|_{L^\infty_tH_x^{s_0-\frac14}}+ \| \nabla w \|_{L^\infty_tL^8_x} \leq \epsilon_2,
	\end{equation}
	hold.
\end{proposition}
\begin{proof}
Due to \eqref{fcr}, and using the Strichartz estimate in Proposition \ref{r5} (taking $r=s$ and $a=1$), we have
	\begin{equation}\label{rrr}
	\begin{split}
		\| d \rho, d \bv_+\|_{L^2_t C^\delta_x}  \lesssim \| \rho \|_{L^\infty_t H_x^s}+\|\bv\|_{L^\infty_t H_x^s}+ {\| w\|_{L^\infty_t H_x^{1+\delta}}}.
	\end{split}
\end{equation}
	Recall 
	\begin{equation*}
		\bv=\bv_{+}+\bv_{-}, \quad v^i_{-}=(-\Delta)^{-1} (\epsilon^{ia}\mathrm{e}^{\rho}\partial_a w).
	\end{equation*}
By Sobolev imbedding $C^\delta_x \hookrightarrow H_x^{1+2\delta}$, we then get
	\begin{equation}\label{pp3}
		\begin{split}
			\| \nabla \bv_{-}\|_{C^\delta_x} & \lesssim \| \nabla(-\Delta)^{-1} (\mathrm{e}^{\rho}\nabla w )\|_{H^{1+2\delta}}
			\\
			& \lesssim \|\mathrm{e}^{\rho} \nabla w \|_{H^{2\delta}}
			\\
			& \lesssim \|\rho\|_{H_x^{s}}+\| w \|_{H_x^{1+2\delta}}.
		\end{split}
	\end{equation} 
Thanks to \eqref{rrr} and \eqref{pp3}, we have proved \eqref{SR0}. For $s_0>\frac74$, using Theorem \ref{bBe}, we conclude that
	\begin{equation}\label{pp2}
		\begin{split}
			\| \nabla \rho,  \nabla \bv_{+}, \nabla \bv\|_{L^2_t C^\delta_x} 
			&  \lesssim  \| \rho \|_{L^\infty_t H_x^s}+\|\bv\|_{L^\infty_t H_x^s}+ {\| w\|_{L^\infty_t H_x^{s_0-\frac14}}}
			\\
			& \lesssim \| \rho_0\|_{H^s}+\|\bv_0\|_{H^s} + \| w_0\|_{H^{s_0-\frac14}} + \|\nabla w_0\|_{L_x^8}.
		\end{split}
	\end{equation}
	By using \eqref{401} and \eqref{pp2}, we have
	\begin{equation}\label{900}
		\begin{split}
			\| \nabla \rho, \nabla \bv_{+}, \nabla \bv\|_{L^2_t C^\delta_x} \lesssim \epsilon_3.
		\end{split}
	\end{equation}
	Due to \eqref{fc}, it yields
	\begin{equation}\label{Stwo}
		\| \mathbf{T} \rho,\mathbf{T} \bv\|_{L^4_t C^\delta_x} \lesssim \| \nabla \rho, \nabla \bv\|_{L^4_t C^\delta_x } .
	\end{equation}
	Note that $\epsilon_3 \ll \epsilon_2$. Taking advantage of \eqref{900} and \eqref{Stwo}, we therefore obtain \eqref{strr}. By using Theorem \ref{be} and \eqref{strr}, the estimate \eqref{eef} holds. Therefore, we have finished the proof of this proposition. 
\end{proof}
\begin{remark}
	The Strichartz estimate \eqref{SR0} requires a very low regularity of the velocity, density, and vorticity. We hope that the regularity exponent in \eqref{SR0} is sharp.
\end{remark}
It only remains for us to prove Proposition \ref{r5}. This will be discussed in the next section.

\section{Proof of Proposition \ref{r5}}\label{Ap}
In this section, following Smith-Tataru's paper \cite{ST}, we will present a proof for Proposition \ref{r5} by using wave-packet approximation, . 

\subsection{The proof of Proposition \ref{r5}}
In order to prove Proposition \ref{r5}, we first reduce it to phase space. Given a frequency scale $\lambda \geq 1$ ($\lambda=2^j, j\in \mathbb{N}^{+}$), we consider the smooth coefficients
\begin{equation*}
	\mathbf{g}_{\lambda}= P_{<\lambda} \mathbf{g}=\sum_{\lambda'<\lambda} P_{\lambda'},
\end{equation*}
where $\mathbf{g}$ is defined in \eqref{boldg}. We now introduce a proposition in the phase space, which will be very useful.
\begin{proposition}\label{A1}
	Suppose $(\bv,\rho, w) \in \mathcal{{H}}$ and $G(\bv,\rho)\leq 2 \epsilon_1$. Let $f$ satisfy
	\begin{equation}\label{linearA}
		\begin{cases}
			\square_{\mathbf{g}} f=0, \quad (t,x)\in [-2,2]\times \mathbb{R}^2,\\
			(f, \partial_t f)|_{t=t_0}=(f_0, f_1).
		\end{cases}
	\end{equation}
	Then for each $(f_0,f_1) \in H^1 \times L^2$ there exists a function $f_{\lambda} \in C^\infty([-2,2]\times \mathbb{R}^2)$, with
	\begin{equation*}
		\mathrm{supp} \widehat{f_\lambda(t,\cdot)} \subseteq \{ \xi \in \mathbb{R}^2: \frac{\lambda}{8} \leq |\xi| \leq 8\lambda \},
	\end{equation*}
	such that
	\begin{equation}\label{Yee}
		\begin{cases}
			& \| \square_{\mathbf{g}_\lambda} f_{\lambda} \|_{L^4_{[-2,2]} L^2_x} \lesssim \epsilon_0 (\| f_0\|_{H^1}+\| f_1 \|_{L^2} ),
			\\
			& f_\lambda(-2)=P_\lambda f_0, \quad \partial_t f_{\lambda} (-2)=P_{\lambda} f_1.
		\end{cases}
	\end{equation}
	Additionally, if $r>\frac34$, then the following Strichartz estimates holds:
	\begin{equation}\label{Ase}
	\| f_{\lambda} \|_{L^4_{[-2,2]} L^\infty_x} \lesssim \epsilon_0^{-\frac{1}{4}} \lambda^{r-1} ( \| P_\lambda f_0 \|_{H^1} + \| P_\lambda f_1 \|_{L^2} ) .
	\end{equation}
\end{proposition}
Next, we give a proof for Proposition \ref{r5} based on Proposition \ref{A1}.
\begin{proof}[Proof of Proposition \ref{r5} by use of Proposition \ref{A1}]
	We divide the proof into several cases.
	
	(i) $r=1$. Using basic energy estimates for \eqref{linear}, we have
	\begin{equation*}
		\begin{split}
			\| \partial_t f \|_{L^2_x} + \|\nabla f \|_{L^2_x}  \lesssim & \ (\| f_0\|_{H^1}+ \| f_1\|_{L^2}) \exp(\int^t_0 \| d \mathbf{g} \|_{L^\infty_x} d\tau)
			\\
			\lesssim & \ \| f_0\|_{H^1}+ \| f_1\|_{L^2}.
		\end{split}
	\end{equation*}
	Then the Cauchy problem \eqref{linearA} holds a unique solution $f \in C([-2,2],H_x^1)$ and $\partial_t f \in C([-2,2],L_x^2)$. It remains to show that the solution $f$ also satisfies the Strichartz estimate \eqref{SL}.

	Without loss of generality, we take $t_0=0$. For any given initial data $(f_0,f_1) \in H^1 \times L^2$, and $t_0 \in [-2,2]$, we take a Littlewood-Paley decomposition
	\begin{equation*}
		f_0=\sum_{\lambda}P_{\lambda}f_0, \qquad  f_1=\sum_{\lambda}P_{\lambda}f_1,
	\end{equation*}
	and for each $\lambda$ we take the corresponding $f_{\lambda}$ as in Proposition \ref{A1}.  If we set
	\begin{equation*}
		f=\sum_{\lambda}f_{\lambda},
	\end{equation*}
	then $f$ matches the initial data $(f_0,f_1)$ at the time $t=t_0$, and also satisfies the Strichartz estimates \eqref{Ase}. In fact, $f$ is also an approximate solution for $\square_{\mathbf{g}}$ in the sense that
	\begin{equation*}
		\| \square_{\mathbf{g}} f \|_{L^2_t L^2_x} \lesssim \epsilon_0(\| f_0 \|_{H^1}+\| f_1 \|_{L^2} ).
	\end{equation*}
	We can derive the above bound by using the decomposition
	\begin{equation*}
		\square_{\mathbf{g}} f=\textstyle{\sum_{\lambda}} \square_{\mathbf{g}_\lambda} f_\lambda+ \textstyle{\sum_{\lambda}} \square_{\mathbf{g}-\mathbf{g}_\lambda} f_\lambda.
	\end{equation*}
	The first term can be controlled by Proposition \ref{A1}. As for the second term, using $\mathbf{g}^{00}=-1$, we can rewrite
	\begin{equation*}
		\begin{split}
			\textstyle{\sum_{\lambda}} \square_{\mathbf{g}-\mathbf{g}_\lambda} f_\lambda= \textstyle{\sum_{\lambda}} ({\mathbf{g}-\mathbf{g}_\lambda}) \nabla  df_\lambda.
		\end{split}
	\end{equation*}
	By using H\"older's inequality, it follows that
	\begin{equation}\label{yv1}
		\begin{split}
			\| \textstyle{\sum_{\lambda}} \square_{\mathbf{g}-\mathbf{g}_\lambda} f_\lambda \|_{L^2_x} \lesssim \mathrm{sup}_{\lambda} \left(\lambda \| {\mathbf{g}-\mathbf{g}_\lambda} \|_{L^\infty_x} \right) \left( \sum_{\lambda}  \| df_\lambda \|^2_{L^2_x} \right)^{\frac12}.
		\end{split}
	\end{equation}
	Due to Bernstein's inequality, we have
	\begin{equation}\label{yv2}
		\begin{split}
			\mathrm{sup}_{\lambda} \left(\lambda \| {\mathbf{g}-\mathbf{g}_\lambda} \|_{L^\infty_x} \right) \lesssim & \ \mathrm{sup}_{\lambda} \big(\lambda \sum_{\mu > \lambda}\| \mathbf{g}_\mu \|_{L^\infty_x} \big)
			\\
			\lesssim & \ \mathrm{sup}_{\lambda} \big(\lambda \sum_{\mu > \lambda}\mu^{-(1+\delta)}\|d \mathbf{g}_\mu \|_{C^\delta_x} \big)
			\\
			\lesssim & \ \|d \mathbf{g} \|_{C^\delta_x}(\textstyle{\sum_{\mu}} \mu^{-\delta}) \lesssim \ \|d \mathbf{g} \|_{C^\delta_x}.
		\end{split}
	\end{equation}
	Combining \eqref{yv1} and \eqref{yv2}, we get
	\begin{equation*}
		\begin{split}
			\textstyle{\sum_{\lambda}} \square_{\mathbf{g}-\mathbf{g}_\lambda} f_\lambda \lesssim \epsilon_0(\| f_0 \|_{H^1}+\| f_1 \|_{L^2} ).
		\end{split}
	\end{equation*}
	For given $F\in L^1_t L^2_x$, we set
	\begin{equation*}
		\mathbf{M} F=\int^t_0 f^{\tau}(t,x)d\tau,
	\end{equation*}
	where $f^{\tau}(t,x)$ is the approximate solution formed above with the Cauchy data
	\begin{equation*}
		f^{\tau}(\tau,x)=0, \quad \partial_t f^{\tau}(\tau,x)=F(\tau,\cdot).
	\end{equation*}
	By calculating
	\begin{equation*}
		\square_{\mathbf{g}} \mathbf{M}F=\int^t_0 \square_{\mathbf{g}} f^{\tau}(t,x) d\tau+F,
	\end{equation*}
	it follows that
	\begin{equation*}
		\| \square_{\mathbf{g}} \mathbf{M}F-F\|_{L^2_t L^2_x} \lesssim \| \square_{\mathbf{g}}f^{\tau}\|_{L^1_t L^2_x} \lesssim \epsilon_0 \| F \|_{L^2_t L^2_x}.
	\end{equation*}
	Using the contraction principle, we can write the solution $f$ in the form
	\begin{equation*}
		f= \tilde{f}+ \mathbf{M}F,
	\end{equation*}
	where $\tilde{f}$ is the approximation solution formed above for initial data $(f_0,f_1)$ specified at time $t=0$, and
	\begin{equation*}
		\| F\|_{L^2_t L^2_x} \lesssim \epsilon_0 ( \| f_0 \|_{H^1}+\| f_1\|_{L^2}).
	\end{equation*}
	The Strichartz estimates now follow since they holds for each $f^\tau$, $\tau \in [0,t]$. By Duhamel's principle, we can also obtain the Strichartz estimates for the linear, nonhomogeneous wave equation
	\begin{equation*}
		\begin{cases}
			\square_{\mathbf{g}} f= F',
			\\
			f|_{t=0}=f_0, \quad \partial_tf|_{t=0}=f_1.
		\end{cases}
	\end{equation*}
	That means
	\begin{equation*}
		\| \left< \nabla \right>^a f \|_{L^4_t L^\infty_x} \lesssim  \|f_0\|_{H^1}+ \| f_1 \|_{L^2}+ \|F'\|_{L^1_t L^2_x} , \quad   a<\frac14.
	\end{equation*}
 (ii) $1 < r \leq s+1$. Based on the above result, we plan to transform the initial data in $H^1 \times L^2$. Operating $\left< \nabla \right>^{r-1}$ on \eqref{linearA}, we have 
	\begin{equation*}
		\square_{\mathbf{g}} \left< \nabla \right>^{r-1}f=-[\square_{\mathbf{g}}, \left< \nabla \right>^{r-1}]f.
	\end{equation*}
	Let $\left< \nabla \right>^{r-1}f=\bar{f}$. Then $\bar{f}$ is a solution to
	\begin{equation}\label{qd}
		\begin{cases}
			\square_{\mathbf{g}}\bar{f}=-[\square_{\mathbf{g}}, \left< \nabla \right>^{r-1}]\left< \nabla \right>^{1-r}\bar{f},
			\\
			(\bar{f}(t_0), \partial_t \bar{f}(t_0)) \in H^1 \times L^2.
		\end{cases}
	\end{equation}
	To handle this case, we need to bound the right term as
	\begin{equation}\label{qqd}
		\| [\square_{\mathbf{g}}, \left< \nabla \right>^{r-1}]\left< \nabla \right>^{1-r}\bar{f}\|_{L^4_t L^2_x} \lesssim \epsilon_0( \| d\bar{f}\|_{L^\infty_t L^2_x}+  \| \left< \nabla \right>^m d\bar{f}\|_{L^4_t L^\infty_x}),
	\end{equation}
	provided $m>1-s$. To prove it, we apply analytic interpolation to the family
	\begin{equation*}
		\bar{f} \rightarrow   [\square_{\mathbf{g}}, \left< \nabla \right>^{r-1}]\left< \nabla \right>^{1-r}\bar{f}.
	\end{equation*}
	For $\mathrm{Re}z=0$, noting $\mathbf{g}^{00}=-1$, we use the commutator estimate (c.f. (3.6.35) of \cite{KP}) to get
	\begin{equation*}
		\| [\mathbf{g}^{\alpha \beta}, \left< \nabla \right>^z] \partial^2_{\alpha \beta} \bar{f} \|_{L^2_x}=\| [\mathbf{g}^{\alpha i}, \left< \nabla \right>^z] \partial_i ( \partial_{\alpha} \bar{f} )\|_{L^2_x} \lesssim \| d \mathbf{g} \|_{L^\infty_x} \| d \bar{f} \|_{L^2_x}.
	\end{equation*}
	For $\mathrm{Re}z=s$, we use the Kato-Ponce commutator estimate
	\begin{equation}\label{qdy}
		\begin{split}
			\| [\mathbf{g}^{\alpha \beta}, \left< \nabla \right>^z] \left< \nabla \right>^{-z} \partial^2_{\alpha \beta} \bar{f} \|_{L^2_x} 
			= 	& \| [\mathbf{g}^{\alpha i}, \left< \nabla \right>^z] \left< \nabla \right>^{-z} \partial_{i} ( \partial_{\alpha } \bar{f} ) \|_{L^2_x} 
			\\
			\lesssim & \| d \mathbf{g} \|_{L^\infty_x} \| d \bar{f} \|_{L^2_x}+  \| \mathbf{g}^{\alpha \beta}- \mathbf{m}^{\alpha \beta}\|_{H^{s}_x}   \|\left< \nabla \right>^{-z} \nabla d\bar{f} \|_{L^\infty_x}.
		\end{split}
	\end{equation}
Taking advantage of
	\begin{equation*}
		\| d \mathbf{g} \|_{L^4_t L^\infty_x} +  \| \mathbf{g}^{\alpha \beta}- \mathbf{m}^{\alpha \beta}\|_{L^\infty_t H^{s}_x} \lesssim \epsilon_0,
	\end{equation*}
we can bound \eqref{qdy} by
	\begin{equation*}
		\| [\mathbf{g}^{\alpha \beta}, \left< \nabla \right>^z] \left< \nabla \right>^{-z} \nabla^2_{\alpha \beta} \bar{f} \|_{L^2_x} \lesssim \epsilon_0( \| d \bar{f} \|_{L^2_x}+  \|\left< \nabla \right>^{-z} \nabla d\bar{f}  \|_{L^\infty_x}).
	\end{equation*}
	Let us go back \eqref{qd}. Using the discussion in case $r=1$, for $ \theta<0$, we obtain
	\begin{equation}\label{eh}
		\begin{split}
			& \| \left< \nabla \right>^{(\theta-1)} d \bar{f} \|_{L^4_t L^\infty_x} 
			\\
			\lesssim  & \  \epsilon_0( \|\left< \nabla \right>^{r-1}f_0\|_{H^1}+ \| \left< \nabla \right>^{r-1} f_1 \|_{L^2}+ \| d \bar{f} \|_{L^2_x}+  \|\left< \nabla \right>^{-r} \nabla d \bar{f} \|_{L^4_t L^\infty_x} ) ,
			\\
			\lesssim & \ \epsilon_0( \|f_0\|_{H^r}+ \| f_1 \|_{H^{r-1}} + \| d \bar{f} \|_{L^2_x}+  \|\left< \nabla \right>^{-r} \nabla d \bar{f} \|_{L^4_t L^\infty_x} ).
		\end{split}
	\end{equation}
	Taking $\theta = -r+1$ in \eqref{eh}, we can see
	\begin{equation}\label{eeh}
		\|\left< \nabla \right>^{-r} \nabla d \bar{f} \|_{L^4_t L^\infty_x} \lesssim \epsilon_0( \|f_0\|_{H^r}+ \| f_1 \|_{H^{r-1}}).
	\end{equation}
By \eqref{qd}, \eqref{qqd}, and \eqref{eeh}, we have
	\begin{equation*}
		\begin{split}
			\| \left< \nabla \right>^a f, \left< \nabla \right>^{a-1} d f \|_{L^2_t L^\infty_x} \lesssim  
		  & \  \epsilon_0( \|f_0\|_{H^r}+ \| f_1 \|_{H^{r-1}} ) , \quad   a<r-1.
		\end{split}
	\end{equation*}
\end{proof}
\begin{remark}\label{rel}
From Proposition \ref{A1}, it implies that there is a good approximate solution $f_\lambda$ for the problem
	\begin{equation}\label{linearD}
		\begin{cases}
			\square_{\mathbf{g}_\lambda} f=0, \quad (t,x)\in [-2,2]\times \mathbb{R}^2,\\
			(f, \partial_t f)|_{t=t_0}=(P_\lambda f_0, P_\lambda f_1).
		\end{cases}
	\end{equation}
	In the case of $\epsilon_0 \lambda\leq 1$, we can take $f_\lambda=P_\lambda f$, where $f$ is the exact solution of \eqref{linearD}.
	Applying energy estimates for \eqref{linearD}, we can see
	\begin{equation}\label{eet}
		\|df\|_{L^\infty_{[-2,2]} L^2_x} \lesssim \| P_\lambda f_0\|_{H^1}+\|P_\lambda f_1\|_{L^2}.
	\end{equation}
	Moreover, for $\mathbf{g}^{00}_\lambda=-1$, so we have
	\begin{equation*}
		\begin{split}
			\| \square_{\mathbf{g}_\lambda} f_\lambda \|_{L^2_{[-2,2]} L^2_x} 
			\lesssim & 
			\| [ \mathbf{g}^{\alpha i}_\lambda, \partial_\alpha P_\lambda ] \partial_i f \|_{ L^2_{[-2,2]} L^2_x } 
			+ \| P_\lambda (\partial_\alpha \mathbf{g}^{\alpha i}_\lambda ) \partial_i f \|_{ L^2_{[-2,2]} L^2_x }
			\\
			\lesssim & \| d \mathbf{g}_\lambda \|_{L^2_{[-2,2]} L^\infty_x} \| d {f} \|_{L^\infty_{[-2,2]} L^2_x}
			\\
			\lesssim & \epsilon_0 ( \| P_\lambda f_0\|_{H^1}+\|P_\lambda f_1\|_{L^2} ).
		\end{split}
	\end{equation*}
	The Strichartz estimate \eqref{Ase} follows from Sobolev imbeddings and \eqref{eet}. Hence, the rest of the paper is to establishing Proposition \ref{AA1} in the case that
	\begin{equation*}
		\epsilon_0 \lambda \gg 1.
	\end{equation*}
\end{remark}
\subsection{Proof of Proposition \ref{A1}}
By using Remark \ref{rel}, it suffices for us to consider the problem if the frequency $\lambda$ is large enough, namely
\begin{equation*}
	\lambda \geq \epsilon_0^{-1}.
\end{equation*}
In this case, we can construct an approximate solution to \eqref{linearA} by using wave packets. We also note that the proposition \ref{A1} has two parts: \eqref{Yee} and \eqref{Ase}. The conclusion of the first part, i.e. \eqref{Yee} will be established in the following Proposition \ref{szy} and \ref{szi}. The inequality \eqref{Ase} will be proved in Proposition \ref{szt}.

Before our proof, we introduce a spatially localized mollifier $T_\lambda$ by
\begin{equation*}
	T_\lambda f = \chi_\lambda * f, \quad \chi_\lambda=\lambda^2 \chi(\lambda^{-1} y),
\end{equation*}
where $\chi \in C^\infty_0(\mathbb{R}^2)$ is supported in the ball $|x| \leq \frac{1}{32}$, and has integral $1$. By choosing $\chi$ appropriately, any function $\phi$ with frequency support contained in $\{\xi\in \mathbb{R}^2:|\xi|\leq 4\lambda\}$ can be factored $\phi=T_\lambda \widetilde{\phi}$, where $\| \widetilde{\phi} \|_{L^2_x} \approx \| \phi \|_{L^2_x}$.


\subsubsection{\textbf{A normalized wave packet}}\label{cwp}
Let's begin with the definition of a normalized wave packet.
\begin{definition}[\cite{ST},Definition 8.1]
	Let the hypersurface $\Sigma_{\omega,r}$ and the geodesic $\gamma$ be defined in Section \ref{sec6}. A normalized wave packet centered around $\gamma$ is a function $f$ of the form
	\begin{equation*}
		f=\epsilon_0^{\frac14} \lambda^{-\frac54} T_\lambda(u h),
	\end{equation*}
	where
	\begin{equation*}
		u(t,x)=\delta(x_{\omega}-\phi_{\omega,r}(t,x'_\omega)), \quad h=h_0( (\epsilon_0 \lambda)^{\frac12}(x'_\omega-\gamma'_\omega(t))  ).
	\end{equation*}
	Here, $h_0$ is a smooth function supported in the set $|x'| \leq 1$, with uniform bounds on its derivatives $|\partial^\alpha_{x'} h_0(x')| \leq c_\alpha$.
\end{definition}
\begin{remark}
The advantage of this definition is that the derivatives of wave packets involve only the tangential behavior of the restrictions of various functions to the characteristic surfaces $\Sigma_{\omega,r}$, as opposed to their regularity within the support of the wave packet. Therefore, we don't need to discuss the behavior of the null foliations $\cup_{r\in \mathbb{R}}\Sigma_{\omega,r}$ in
transversal directions.
\end{remark}
We give two notations here. We denote $L(\varphi,\psi)$ to denote a translation invariant bilinear operator of the form
\begin{equation*}
	L(\varphi,\psi)(x)=\int K(y,z)\varphi(x+y)\psi(x+z)dydz,
\end{equation*}
where $K(y,z)$ is a finite measure. If $X$ is a Sobolev spaces, we then denote $X_\kappa$ the same space but with the norm obtained by dimensionless rescaling by $\kappa$,
\begin{equation*}
	\| \varphi \|_{X_\kappa}=\| \varphi(\kappa \cdot)\|_{X}.
\end{equation*}
Since $\frac74<s_0\leq s \leq 2$, we get $s_0-\frac54>\frac12$. For $\kappa<1$, we have $\| \varphi \|_{H^{s_0-\frac54}_\kappa(\mathbb{R})} \lesssim \| \varphi \|_{H^{s_0-\frac54}(\mathbb{R})}$.

Based on the above definition, let us calculate what we will get when we operate $\square_{\mathbf{g}_\lambda}$ on wave packets.
\begin{proposition}\label{np}
	Let $f$ be a normalized packet. Then there exists another normalized wave packet $\tilde{f}$ and functions $\phi_m(t,x'_\omega), m=0,1,2$, so that
	\begin{equation}\label{np1}
		\square_{\mathbf{g}_\lambda} P_\lambda f= L(d \mathbf{g}, d \tilde{P}_\lambda \tilde{f})+ \epsilon_0^{\frac14}\lambda^{-\frac54}P_{\lambda}T_{\lambda}\sum_{m=0,1,2}
		\psi_m\delta^{(m)}(x'_\omega-\phi_{\omega,r}),
	\end{equation}
	where the functions $\psi_m=\psi_m(t,x'_\omega)$ satisfy the scaled Sobolev estimates
	\begin{equation}\label{np2}
		\| \psi_m\|_{L^2_t H^{s_0-\frac54}_{a,x'_\omega}} \lesssim \epsilon_0 \lambda^{1-m}, \quad m=0,1,2, \quad a=(\epsilon_0 \lambda)^{-\frac12}.
	\end{equation}
\end{proposition}
\begin{proof}[Proof of Proposition \ref{np}]
	For brevity, we consider the case $\omega=(0,1)$. Then $x_\omega=x_2$, and $x'_\omega=x'$. We write
	\begin{equation}\label{js0}
		\square_{\mathbf{g}_\lambda} P_\lambda f
		= \lambda^{-1} ( [\square_{\mathbf{g}_\lambda}, P_\lambda T_\lambda]+P_\lambda T_\lambda \square_{\mathbf{g}_\lambda} )(uh).
	\end{equation}
	For the first term in \eqref{js0}, noting $\mathbf{g}_\lambda$ supported at frequency $\leq 4 \lambda$, then we can write
	\begin{equation*}
		[\square_{\mathbf{g}_\lambda}, P_\lambda T_\lambda]=[\square_{\mathbf{g}_\lambda}, P_\lambda T_\lambda]\tilde{P}_\lambda \tilde{T}_\lambda
	\end{equation*}
	for some multipliers $\tilde{P}_\lambda$ and $\tilde{T}_\lambda$ which have the same properties as $P_\lambda$ and $T_\lambda$. Therefore, by using the kernal bounds for $P_\lambda T_\lambda$, we conclude that
	\begin{equation*}
		[\square_{\mathbf{g}_\lambda}, P_\lambda T_\lambda]f=L(d\mathbf{g}, df).
	\end{equation*}
	For the second term in \eqref{js0}, we use the Leibniz rule
	\begin{equation}\label{js1}
		\square_{\mathbf{g}_\lambda}(uh)=h \square_{\mathbf{g}_\lambda} u+(\mathbf{g}^{\alpha \beta}_\lambda+\mathbf{g}^{\beta \alpha}_\lambda)\partial_\alpha u \partial_\beta h + u \square_{\mathbf{g}_\lambda} h.
	\end{equation}
	Let $\nu$ denote the conormal vector field along $\Sigma$. We thus know $\nu=dx_2-d\phi(t,x')$. In the following, we take the greek indices $0 \leq \alpha, \beta \leq 1$.

	For the first term in \eqref{js1}, we can compute out
	\begin{equation*}
		\begin{split}
			\mathbf{g}_{\lambda}^{\alpha \beta} \partial^2_{\alpha \beta} u =& \mathbf{g}_{\lambda}^{\alpha \beta}(t,x',\phi)\nu_\alpha \nu_\beta \delta^{(2)}_{x_2-\phi}
			-2 (\partial_2 \mathbf{g}_{\lambda}^{\alpha \beta})(t,x',\phi) \nu_\alpha \nu_\beta \delta^{(1)}_{x_2-\phi}
			\\
			&+(\partial^2_2 \mathbf{g}_{\lambda}^{\alpha \beta})(t,x',\phi) \nu_\alpha \nu_\beta \delta^{(0)}_{x_2-\phi}- \mathbf{g}_{\lambda}^{\alpha \beta}(t,x',\phi) \partial^2_{\alpha \beta}\phi \delta^{(1)}_{x_2-\phi}
			\\
			&+ \partial_2\mathbf{g}_{\lambda}^{\alpha \beta}(t,x',\phi) \partial^2_{\alpha \beta}\phi \delta^{(0)}_{x_2-\phi}.
		\end{split}
	\end{equation*}
	Above, $\delta^{(m)}_{x_2-\phi}=(\partial^m \delta)(x_2-\phi) $. Due to Leibniz rule, we can take
	\begin{equation*}
		\begin{split}
			\psi_0&=h \big\{ (\partial^2_2 \mathbf{g}_{\lambda}^{\alpha \beta})(t,x',\phi)\nu_\alpha \nu_\beta+ (\partial_2 \mathbf{g}_{\lambda}^{\alpha \beta})(t,x',\phi)\partial^2_{\alpha \beta}\phi  \big\},
			\\
			\psi_1&=h \big\{ 2(\partial_2 \mathbf{g}_{\lambda}^{\alpha \beta})(t,x',\phi)\nu_\alpha \nu_\beta-  \mathbf{g}_{\lambda}^{\alpha \beta}(t,x',\phi)\partial^2_{\alpha \beta}\phi \big\},
			\\
			\psi_2&=h( \mathbf{g}_{\lambda}^{\alpha \beta}-\mathbf{g}^{\alpha \beta})\nu_\alpha \nu_\beta,
		\end{split}
	\end{equation*}
	Taking advantage of \eqref{G}, Proposition \ref{r1}, and Corollary \ref{vte}, we can conclude that this settings of $\psi_0, \psi_1,$ and $\psi_2$ satisfy the estimates \eqref{np2}.

	For the second term in \eqref{js0}, we have
	\begin{equation*}
		\begin{split}
			(\mathbf{g}^{\alpha \beta}_\lambda+\mathbf{g}^{\beta \alpha}_\lambda)\partial_\alpha u \partial_\beta h=& \frac12\nu_\alpha (\mathbf{g}^{\alpha \beta}_\lambda+\mathbf{g}^{\beta \alpha}_\lambda)(t,x',\phi)\partial_\beta h \delta^{(1)}_{x_2-\phi}
			\\
			&- \frac12\nu_\alpha \partial_2 (\mathbf{g}^{\alpha \beta}_\lambda+\mathbf{g}^{\beta \alpha}_\lambda)(t,x',\phi)\partial_\beta h \delta^{(0)}_{x_2-\phi}.
		\end{split}
	\end{equation*}
	Then we can take
	\begin{equation*}
		\psi_0= \frac12\nu_\alpha \partial_2 (\mathbf{g}^{\alpha \beta}_\lambda+\mathbf{g}^{\beta \alpha}_\lambda)(t,x',\phi)\partial_\beta h, \quad \psi_1= \frac12\nu_\alpha (\mathbf{g}^{\alpha \beta}_\lambda+\mathbf{g}^{\beta \alpha}_\lambda)(t,x',\phi)\partial_\beta h.
	\end{equation*}
	Thanks to \eqref{G}, Proposition \ref{r1}, and Corollary \ref{vte}, we conclude that this settings of $\psi_0$ and $\psi_1$ satisfy the estimate \eqref{np2}.

	For the third term in \eqref{js0}, we take
	\begin{equation*}
		\psi_0=\mathbf{g}^{\alpha \beta}_\lambda(t,x',\phi)\partial^2_{\alpha \beta}h.
	\end{equation*}
	By using \eqref{G}, Proposition \ref{r1}, and Corollary \ref{vte} again, this settings of $\psi_0$ satisfies the estimates \eqref{np2}.
\end{proof}
As a immediate consequence, we can obtain the following corollary.
\begin{corollary}\label{np0}
	Let $f$ be a normalized wave packet. Then the following estimates hold:
	\begin{equation*}\label{np3}
		\|d P_\lambda f \|_{L^\infty_t L^2_x} \lesssim 1, \quad \| \square_{\mathbf{g}} P_\lambda f \|_{L^2_t L^2_x} \lesssim \epsilon_0.
	\end{equation*}
\end{corollary}
From Proposition \ref{np} and Corollary \ref{np0}, a single normalized wave packet is not enough for us to construct approximate solutions to a linear wave equation. Therefore, we discuss the superposition of wave packets as follows.
\subsubsection{\textbf{Superpositions of wave packets}}
The index $\omega$ stands for the initial orientation of the wave packet at $t=-2$, which varies over a maximal collection of approximately $\epsilon_0^{-\frac12}\lambda^{\frac12}$ unit vectors separated by at least $\epsilon_0^{\frac12}\lambda^{-\frac12}$. For each $\omega$, we have the orthonormal coordinate system $(x_\omega, x'_\omega)$ of $\mathbb{R}^2$, where $x_\omega=x \cdot \omega$, and $x'_\omega$ are projective along $\omega$.


Next, we decompose $\mathbb{R}^2$ into a parallel tiling of rectangles of size $(8\lambda)^{-1}$ in the $x_{\omega}$ direction, and $(4\epsilon_0 \lambda)^{-\frac12}$ in the other directions $x'_{\omega}$. The index $j \in \mathbb{N}$ corresponds to a counting of these rectangles in this decomposition. We let $R_{\omega,j}$ denote the collection of the doubles of these rectangles, and let $\Sigma_{\omega,j}$ denote the element of the $\Sigma_{\omega}$ foliation upon which $R_{\omega,j}$ is centered. 
The distinct elements of the foliation, denoted by $\Sigma_{\omega,j}$ are thus separated by at least $(8\lambda)^{-1}$ at $t = −2$, and thus by $(9\lambda)^{-1}$ at all values of $t$, as shown in \eqref{pr0} below. Let $\gamma_{\omega,j}$ denote the null geodesic contained in $\Sigma_{\omega,j}$ which passes through the center of $R_{\omega,j}$ at time $t = −2$.

We denote the slab $T_{\omega,j}$ as
\begin{equation}\label{twj}
	T_{\omega,j}=\Sigma_{\omega,j} \cap \{ |x'_{\omega}-\gamma_{\omega,j}| \leq (\epsilon_0 \lambda)^{-\frac12}\},
\end{equation}
and let
\begin{equation*}
	\Sigma_\omega=\cup_{r\in\mathbb{R}}\Sigma_{\omega,r}.
\end{equation*}
For each $\omega$ the slabs $T_{\omega,j}$ satisfy a finite-overlap condition; indeed, slabs associated to different elements of $\Sigma_\omega$ are disjoint, and those associated to the
same $\Sigma_\omega$ have finite overlap in the $x'_{\omega}$ variable, since the flow restricted to any
$\Sigma_{\omega,r}$ is $C^1$ close to translation. We next introduce some geometry properties for these slabs.

By \eqref{600} and \eqref{606}, then the estimate
\begin{equation*}
	|dr_{\theta}-(\theta \cdot d x-dt)| \lesssim \epsilon_1,
\end{equation*}
holds pointwise uniformly on $[-2,2]\times \mathbb{R}^2$. This also implies that
\begin{equation}\label{pr0}
	| \phi_{\theta,r}(t,x'_{\theta})-\phi_{\theta,r'}(t,x'_{\theta})-(r-r')| \lesssim \epsilon_1|r-r'|.
\end{equation}
On the other hand, \eqref{502} tells us
\begin{equation}\label{pr1}
	\| d^2_{x'_{\omega}}\phi_{\omega,r}(t,x'_{\omega})-d^2_{x'_{\omega}}\phi_{\omega,r'}(t,x'_{\omega})\|_{L^\infty_{x'_{\omega}}} \lesssim \epsilon_2+ \bar{\rho}(t),
\end{equation}
where set $$\bar{\rho}(t)=\| d \mathbf{g} \|_{C^{\delta_0}_x}.$$
By using \eqref{pr0} and \eqref{pr1}, we get
\begin{equation*}\label{pr2}
	\| d_{x'_{\omega}}\phi_{\omega,r}(t,x'_{\omega})-d_{x'_{\omega}}\phi_{\omega,r'}(t,x'_{\omega})\|_{L^\infty_{x'_{\omega}}} \lesssim (\epsilon_2+ \bar{\rho}(t))^{\frac12}|r-r'|^{\frac12}.
\end{equation*}
For $dx_{\omega}-d\phi_{\omega,r}$ is null and also $|d \mathbf{g} | \leq \bar{\rho}(t)$, this also implies H\"older-$\frac12$ bounds on $d\phi_{\omega,r}$. Therefore, we suppose that $(t,x)\in \Sigma_{\omega,r}$ and $(t,y)\in \Sigma_{\omega,r'}$, that $|x'_{\omega}-y'_{\omega}| \leq 2(\epsilon_0 \lambda)^{-\frac12}$, and that $|r-r'|\leq 2 \lambda^{-1}$. Using \eqref{601}, we can obtain
\begin{equation*}
	|l_{\omega}(t,x)-l_{\omega}(t,y)| \lesssim \epsilon_0^{\frac12}\lambda^{-\frac12}+\epsilon_0^{-\frac12}\bar{\rho}(t)\lambda^{-\frac12}.
\end{equation*}
Due to $\dot{\gamma}_{\omega}=l_{\omega}$ and $\|\bar{\rho}\|_{L^4_t} \lesssim \epsilon_0$, any geodesic in $\Sigma_{\omega}$ which intersects a slab $T_{\omega,j}$ should be contained in the similar slab of half the scale.

We now introduce a lemma for superpositions of wave packets from a certain fixed time.
\begin{Lemma}\label{SWP}[\cite{ST},Lemma 8.5 and Lemma 8.6]
	Let $\frac74<s_0\leq s \leq 2$, $\delta_0 \in (0, s_0-\frac74)$ and $0<\mu < \delta_0$. Let a scalar function $\bar{v}(t,x)$ be formulated by
	\begin{equation*}\label{swp}
		\bar{v}(t,x)=\epsilon_0^{\frac14}\lambda^{-\frac14} P_{\lambda} \sum_{\omega,j}T_\lambda(\psi^{\omega,j}\delta_{x_{\omega}-\phi_{\omega,j}(t,x'_{\omega})}).
	\end{equation*}
	Set $a=(\epsilon_0 \lambda)^{-\frac12}$. Then we have
	\begin{equation}\label{swp0}
		\| \bar{v}(t) \|^2_{L^2_x}\lesssim \sum_{\omega,j}\| \psi^{\omega,j}\|^2_{H^{\frac12+\mu}_a}, \qquad \text{if} \ \ \ \ \bar{\rho}(t) \leq \epsilon_0,
	\end{equation}
	and
	\begin{equation}\label{swp1}
		\| \bar{v}(t) \|^2_{L^2_x}\lesssim \epsilon_0^{-1} \bar{\rho}(t)  \sum_{\omega,j}\| \psi^{\omega,j}\|^2_{H^{\frac12+\mu}_a},  \quad \text{if} \ \ \bar{\rho}(t) \geq \epsilon_0 .
	\end{equation}

	\begin{remark}
		Thanks to \eqref{swp0} and \eqref{swp1}, we can carry out
		\begin{equation}\label{swp3}
			\| \bar{v}(t) \|^2_{L^2_x}\lesssim \left(1+\epsilon_0^{-1} \bar{\rho}(t)\right)  \sum_{\omega,j}\| \psi^{\omega,j}\|^2_{H^{\frac12+\mu}_a}.
		\end{equation}
	\end{remark}
\end{Lemma}
\begin{proposition}[\cite{ST},Proposition 8.4]\label{szy}
	Let $f=\sum_{\omega,j}a_{\omega,j}f^{\omega,j}$, where $f^{\omega,j}$ are normalized wave packets supported in $T_{\omega,j}$. Then we have
	\begin{equation}\label{ese}
		\| d P_\lambda f\|_{L^\infty_t L^2_x} \lesssim (\sum_{\omega,j} a^2_{\omega,j})^{\frac12},
	\end{equation}
	and
	\begin{equation}\label{ese1}
		\| \square_{\mathbf{g}_{\lambda}} P_\lambda f \|_{L^1_t L^2_x} \lesssim \epsilon_0 (\sum_{\omega,j} a^2_{\omega,j})^{\frac12}.
	\end{equation}
\end{proposition}
\begin{proof}
	We first prove a weaker estimate comparing with \eqref{ese}
	\begin{equation}\label{ese2}
		\| d P_\lambda f\|_{L^2_t L^2_x} \lesssim (\sum_{\omega,j} a^2_{\omega,j})^{\frac12}.
	\end{equation}
	By using \eqref{swp3} and replacing $P_\lambda$ by $\lambda^{-1}\nabla P_\lambda$, and $\psi^{\omega,j}=a_{\omega,j}\zeta^{\omega,j}$, we have
	\begin{equation*}
		\| \nabla P_\lambda f(t)\|^2_{L^2_x} \lesssim (1+\epsilon_0^{-1}\bar{\rho}(t)) \sum_{\omega,j} a^2_{\omega,j}.
	\end{equation*}
	Due to the fact $\|\bar{\rho}\|_{L^4_t} \lesssim \epsilon_0$, we can see that
	\begin{equation}\label{ese3}
		\| \nabla P_\lambda f\|^2_{L^2_t L^2_x} \lesssim \sum_{\omega,j} a^2_{\omega,j} .
	\end{equation}
	We also need to get the similar estimate for the time derivatives. We can calculate
	\begin{equation*}
		\partial_t h = \dot{\gamma}(t) (\epsilon_0 \lambda)^{\frac12} \tilde{h}, \quad \partial_t \delta(x_\omega-\phi_{\omega,j})=\partial_t \phi_{\omega,j} \delta^{(1)}(x_\omega-\phi_{\omega,j}).
	\end{equation*}
	For $\dot{\gamma} \in L^\infty_t$ and $\partial_t \phi_{\omega,j} \in L^\infty_t$, then we have
	\begin{equation}\label{ese4}
		\| \partial_t P_\lambda f\|^2_{L^2_t L^2_x} \lesssim \sum_{\omega,j} a^2_{\omega,j} .
	\end{equation}
	Together with \eqref{ese3} and \eqref{ese4}, we have proved \eqref{ese2}.

	To prove \eqref{ese1}, we use the formula \eqref{np1}. Considering the right hand of \eqref{np1}, by using \eqref{ese2}, we can bound the first term by
	\begin{equation}\label{k26}
		\| L(d\mathbf{g}, d \tilde{P}_{\lambda} \tilde{f}) \|_{L^1_tL^2_x} \lesssim \|d\mathbf{g}\|_{L^4_tL^\infty_x} \|d \tilde{P}_{\lambda} \tilde{f} \|_{L^2_tL^2_x} \lesssim \epsilon_0 (\sum_{\omega,j} a^2_{\omega,j})^{\frac12}.
	\end{equation}
	It only remains for us to estimate the second right term on \eqref{np1}. If we set
	\begin{equation}\label{k28}
		\vartheta=\epsilon_0^{\frac12} \lambda^{-\frac54} P_\lambda T_\lambda \left(\sum_{\omega,j}a_{\omega,j}\cdot \sum_{m=0,1,2}\psi^{\omega,j}_m \delta^{(m)}_{x_{\omega}-\phi_{\omega,j}} \right),
	\end{equation}
	we have
	\begin{equation*}
		\| \vartheta \|^2_{L^2_x}=(1+\bar{\rho}(t)\epsilon_0^{-1}) \sum_{\omega,j}a^2_{\omega,j} \sum_{m=0,1,2} \lambda^{m-1}\|\psi^{\omega,j}_m (t) \|^2_{H^{1+\mu}_a}.
	\end{equation*}
	By \eqref{np2}, we therefore get
	\begin{equation}\label{k30}
		\begin{split}
			& \| \vartheta \|^2_{L^1_t L^2_x}
			\\
			\lesssim & \left(\int^2_{-2}[1+\bar{\rho}(t)\epsilon_0^{-1}]dt \right) \left( \int^{2}_{-2} \sum_{\omega,j}a^2_{\omega,j}  \sum_{m=0,1,2} \lambda^{2(m-1)}\|\psi^{\omega,j}_m (t) \|^2_{H^{1+\mu}_a} dt \right)
			\\
			\lesssim & \epsilon_0 (\sum_{\omega,j} a^2_{\omega,j})^{\frac12}.
		\end{split}
	\end{equation}
	Due to \eqref{np1}, using \eqref{k26}, \eqref{k28}, and \eqref{k30}, we have proved \eqref{ese1}. Using \eqref{ese2} and \eqref{ese1}, and classical energy estimates for linear wave equation, we obtain \eqref{ese}.
\end{proof}
\subsubsection{\textbf{Matching the initial data}}
Although we have constructed the approximate solutions using superpositions of normalized wave packets, we also need to complete this construction, i.e. matching the initial data for the solutions. 
\begin{proposition}[\cite{ST}, Proposition 8.7]\label{szi}
	Given any initial data $(f_0,f_1) \in H^1 \times L^2$, there exists a function of the form
	\begin{equation*}
		f=\sum_{\omega,j}a_{\omega,j}f^{\omega,j},
	\end{equation*}
	where the function $f^{\omega,j}$ are normalized wave packets, such that
	\begin{equation*}
		P_\lambda f(-2)=P_\lambda f_0, \quad \partial_t P_\lambda f(-2)=P_\lambda f_1.
	\end{equation*}
	Furthermore,
	\begin{equation*}
		\sum_{\omega,j}a_{\omega,j}^2 \lesssim \| f_0 \|^2_{H^1}+ \| f_1 \|^2_{L^2}.
	\end{equation*}
\end{proposition}
\subsubsection{\textbf{Overlap estimates}} Since the foliations $\Sigma_{\omega,r}$ varing with $\omega$ and $r$, so a fixed $\Sigma_{\omega,r}$ may intersect with other $\Sigma_{\omega',r'}$. As a result, we should be clear about the number of $\lambda$-slabs which contain two given points in the space-time $[-2,2]\times \mathbb{R}^2$.
\begin{corollary}[\cite{ST}, Proposition 9.2]\label{corl}
	For all points $P_1=(t_1,x_1)$ and $P_2=(t_2,x_2)$ in space-time $\mathbb{R}^{+} \times \mathbb{R}^2$ , and $\epsilon_0 \lambda \geq 1$, the number $N_{\lambda}(P_1,P_2)$ of slabs of scale $\lambda$ that contain both $P_1$ and $P_2$ satisfies the bound
	\begin{equation*}
		\begin{split}
			N_{\lambda}(P_1,P_2)\lesssim & \epsilon_0^{-\frac12} \lambda^{-\frac12} |t_1-t_2|^{-\frac12}.
		\end{split}
	\end{equation*}
\end{corollary}
After the construction of approximate solutions, we still need to prove the key estimate \eqref{Ase}. 
\subsubsection{\textbf{The proof of \eqref{Ase}}}
To start the proof, let us define $\mathcal{T}=\cup_{\omega,j}T_{\omega,j}$, where $T_{\omega,j}$ is set in \eqref{twj}. We also denote $\chi_{{J}}$ be a smooth cut-off function, and $\chi_J=1$ on a set $J$.
\begin{proposition}\label{szt}
	Let $t \in [-2,2]$ and
	\begin{equation*}
		f=\sum_{J \in \mathcal{T}}a_{J}\chi_{J}f_{J},
	\end{equation*}
	where $\sum_{J \in \mathcal{T}}a_{J}^2 \leq 1$ and $f_J$ are normalized wave packets in $J$. Then
	\begin{equation}\label{Aswr}
		\|f \|_{L^4_t L^\infty_x} \lesssim \epsilon_0^{-\frac14}  (\ln \lambda)^{\frac14}.
	\end{equation}
\end{proposition}
\begin{proof}
	This proof follows Smith-Tataru's paper \cite{ST}(Proposition 10.1 on page 353).
	
	Let us first make a partition of the time-interval $[-2,2]$. By decomposition, there exists a partition $\left\{ I_j \right\}$ of the  interval $[-2,2]$ into disjoint subintervals $I_j$ such that with the size of each $I_j$, $|I_j| \approx \lambda^{-1}$, and the number of subintervals $j \approx\lambda$. 
	We claim that, by the Mean Value Theorem, there exists a number $t_j$ such that
	\begin{equation}\label{Aswrq}
		\|f \|^4_{L^4_t L^\infty_x} \leq \sum_j \|f \|^4_{L^4_{I_j} L^\infty_x} \leq \sum_j \|f(t_j,\cdot) \|^4_{ L^\infty_x} ,
	\end{equation}
	where $t_j$ is located in $I_j$, and $|t_{j+1}-t_j| \approx \lambda^{-1}$. Let us explain the \eqref{Aswrq} as follows. To be simple, we let  $I_{j_0}=[0,\lambda^{-1}]$ and $I_{j_0+1}=[\lambda^{-1},2\lambda^{-1}]$.By mean value theorem, on $I_{j_0}$, we have
	\begin{equation*}
		\|f \|^4_{L^4_{I_{j_0}} L^\infty_x} = \|f(t_{j_0},\cdot) \|^4_{ L^\infty_x}\lambda^{-1}.
	\end{equation*}
	We also have
	\begin{equation*}
		\|f \|^4_{L^4_{I_{j_0+1}} L^\infty_x} = \|f(t_{j_0+1},\cdot) \|^4_{ L^\infty_x}\lambda^{-1}.
	\end{equation*}
	If $t_{j_0} \leq \frac12 \lambda^{-1}$ or $t_{j_0+1} \geq \frac32 \lambda^{-1}$, then $|t_{j_0+1}-t_{j_0}| \geq \frac12 \lambda^{-1}$. Otherwise, $t_{j_0} \in [\frac12 \lambda^{-1}, \lambda^{-1}] $ and $t_{j_0+1} \in [\lambda^{-1}, \frac32 \lambda^{-1}]$. In this case,
	we combine $I_{j_0}, I_{j_0+1}$ together and set $I^*_{j_0}= I_{j_0}\cup I_{j_0+1}$. On the new interval  $I^*_{j_0}$, we can see that
	\begin{equation*}
		\|f \|^4_{L^4_{I^*_{j_0}} L^\infty_x} = 2\lambda^{-1}\|f(t^*_{j_0},\cdot) \|^4_{ L^\infty_x},
	\end{equation*}
	and
	\begin{equation*}
		2\|f(t^*_{j_0},\cdot) \|^4_{ L^\infty_x} = \|f(t_{j_0},\cdot) \|^4_{ L^\infty_x}+\|f(t_{j_0+1},\cdot) \|^4_{ L^\infty_x}. 
	\end{equation*}
	For $f$ is a contituous function, we then get
	\begin{equation*}
		\|f(t^*_{j_0},\cdot) \|_{ L^\infty_x} = \left\{  \frac12( \|f(t_{j_0},\cdot) \|^4_{ L^\infty_x} +\|f(t_{j_0+1},\cdot) \|^4_{ L^\infty_x} ) \right\}^{\frac14}, \qquad t^*_{j_0} \in [\frac12\lambda^{-1}, \frac32\lambda^{-1}].
	\end{equation*}
	When no confusion arise, we still set $t_{j_0}=t^*_{j_0} \in [\frac12\lambda^{-1}, \frac32\lambda^{-1}]$. On the next time-interval $I_{j_0+2}=[2\lambda^{-1}, 3\lambda^{-1}]$, we have
	\begin{equation*}
		\|f \|^4_{L^4_{I_{j_0+2}} L^\infty_x} = \|f(t_{j_0+1},\cdot) \|^4_{ L^\infty_x}\lambda^{-1}, \quad t_{j_0+1} \in [2\lambda^{-1}, 3\lambda^{-1}].
	\end{equation*}
	We thus obtain $|t_{j_0+1}-t_{j_0}| \geq \frac12 \lambda^{-1}$. In this way, we can decompose $[-2,2]$.

	Therefore, to prove \eqref{Aswr}, and combining with \eqref{Aswrq}, we only need to show that
	\begin{equation}\label{wee0}
		\sum_j |f(t_j,x_j)|^4 \lesssim \epsilon_0^{-1} \ln \lambda,
	\end{equation}
	where $x_j$ is arbitrarily chosen. We then set the points $P_j=(t_j,x_j)$.

	Since each points lies in at most $\approx \epsilon_0^{-\frac12}\lambda^{\frac12}$ slabs, so we may assume that $|a_J| \geq \epsilon_0^{\frac12}\lambda^{-\frac12}$. Then we decompose the sum $f=\sum_{J \in \mathcal{T}}a_{J}\chi_{J}f_{J}$ dyadically with respect to the size of $a_J$. 
	We next decompose the sum over $j$ via a dyadic decomposition in the numbers of slabs containing $(t_j,x_j)$. We may assume that we are summing over $M$ points $(t_j,x_j)$, each of which is contained in approximately $L$ slabs. Then $|f(t_j,x_j)|\lesssim N^{-\frac12}L$ and
	\begin{equation}\label{wee1}
		\sum_j |f(t_j,x_j)|^4 \lesssim L^4 M N^{-2}.
	\end{equation}
	Combining \eqref{wee0} with \eqref{wee1}, so we only prove
	\begin{equation}\label{wee2}
		L^4 M N^{-1} \lesssim \epsilon_0^{-1} \ln \lambda .
	\end{equation}
	This is a counting problem, which we will prove it by calculating in two different ways. It's based on the number $K$ of pairs $(i,j)$ for which $P_i$ and $P_j$ are contained in a common slab, counted with multiplicity. For $J \in \mathcal{T}$, we denote by $n_J$ the number of points $P_j$ contained in $J$. Then
	\begin{equation*}
		K \approx \sum_{n_J \geq 2} n_J^2 \gtrsim N^{-1}  (\sum_{n_J \geq 2} n_J)^2.
	\end{equation*}
	Note that $ \sum_{J \in \mathcal{T}} n_J \approx ML$. We consider it into two cases. If
	\begin{equation*}
		\sum_{n_J \geq 2} n_J \leq \sum_{n_J =1 } n_J,
	\end{equation*}
	then $N\approx ML$. In this case, combining with the fact that $L \lesssim \epsilon_0^{-\frac12}\lambda^{\frac12}$, then \eqref{wee2} holds. Otherwise, we have
	\begin{equation}\label{wee3}
		K \gtrsim N^{-1}M^2L^2.
	\end{equation}
	In this case, by using Corollary \ref{corl}, we obtain
	\begin{equation}\label{wee4}
		K \lesssim \epsilon_0^{-\frac12} \lambda^{-\frac12} \sum_{1\leq i,j \leq M, i\neq j} |t_i-t_j|^{-\frac12}. 
	\end{equation}
	The sum is maximized in the case that $t_j$ are close as possible, i.e. if the $t_j$ are consecutive multiples of $\lambda^{-1}$. Therefore, for $M \approx \lambda$, we can update \eqref{wee4} as
	\begin{equation}\label{wee5}
		\begin{split}
				K \lesssim & \epsilon_0^{-\frac12} \lambda^{-\frac12} \lambda^{\frac12}  \sum_{1\leq i,j \leq M, i\neq j} |i-j|^{-\frac12} 
				\\
				\lesssim & \epsilon_0^{-\frac12}  (\sum_{1\leq i,j \leq M, i\neq j} |i-j|^{-1})^{\frac12}  (\sum_{1\leq i,j \leq M, i\neq j} 1 )^{\frac12}
				\\ 
				\lesssim & \epsilon_0^{-\frac12} ( M \ln M )^{\frac12} \cdot ( M^2 )^{\frac12}
				\\
				=& M^{\frac32} \epsilon_0^{-\frac12} ( \ln \lambda )^{\frac12}. 
		\end{split}
	\end{equation}
	Combining \eqref{wee5} and \eqref{wee3}, we get \eqref{wee2}. Therefore, we have finished the proof of Proposition \ref{szt}.
\end{proof}

\section{Proof of Theorem \ref{dl3}}\label{Sec8}
In this section, we will prove the existence of solutions for Theorem \ref{dl3} (uniqueness was already established in Corollary \ref{uniq2}). First, we reduce Theorem \ref{dl3} to the case of smooth initial data with bounded frequency support (see Proposition \ref{DDL2} below). Next, we present a small data result with smoother vorticity than that in Proposition \ref{DDL2} (see Proposition \ref{DDL3} below). Finally, based on Proposition \ref{DDL3}, we develop a semi-classical method\footnote{This approach is inspired by Bahouri-Chemin \cite{BC2}, Tataru \cite{T1}, and Ai-Ifrim-Tataru \cite{AIT}.} to prove the Strichartz estimate involving the original regular vorticity in Proposition \ref{DDL2}, thereby completing the proof of Proposition \ref{DDL2}.

\subsection{Proof of Theorem \ref{dl3}}\label{keypra}
Denote $P_{j}$ being the Littlewood-Paley operator with the frequency $2^j (j \in \mathbb{Z})$. By frequency truncation, we set a sequence of initial data $(\bv_{0j},\rho_{0j})$ satisfying
\begin{equation}\label{Dss0}
	\bv_{0j}=P_{\leq j}\bv_0, \quad \rho_{0j}=P_{\leq j}\rho_0,
\end{equation}
where $(\bv_0,\rho_0)$ is stated as \eqref{chuzhi2}, and $P_{\leq j}=\textstyle{\sum}_{k\leq j}P_k$. Following \eqref{sv}, we define
\begin{equation}\label{Qss0}
	w_{0j}=\mathrm{e}^{-\rho_{0j}}\mathrm{curl}\bv_{0j}.
\end{equation}
Due to \eqref{chuzhi2}, we can obtain
\begin{equation}\label{Ess0}
	\begin{split}
		\|w_{0j}\|_{H^{\frac32}}=& \| \mathrm{e}^{-\rho_{0j}}\mathrm{curl}\bv_{0j} \|_{H^{\frac32}}
		\\
		\leq  & \| \mathrm{e}^{-\rho_{0j}}P_{\leq j}(\mathrm{e}^{\rho_0}w_{0}) \|_{H^{\frac32}}
		\\
		\leq &  C M_0 \mathrm{e}^{2M_0} .
	\end{split}
\end{equation}
Similarly, we have
\begin{equation}\label{Essa}
	\begin{split}
		\|\nabla w_{0j}\|_{L^{8}}
		\leq &  C M_0  .
	\end{split}
\end{equation}
Adding \eqref{Dss0}, \eqref{Ess0}, and \eqref{Essa}, we can see
\begin{equation}\label{pu0}
\|\bv_{0j}\|_{H^s}+ \|\rho_{0j}\|_{H^s}+ \| w_{0j} \|_{H^{\frac32}}  + \|\nabla w_{0j}\|_{L^{8}}\leq CM_0 \mathrm{e}^{2M_0}=	{E}(0).
\end{equation}
Using \eqref{HEw}, \eqref{Dss0} and \eqref{Qss0}, we get
\begin{equation}\label{pu00}
	| \bv_{0j}, \rho_{0j} | \leq C_0, \quad c_s|_{t=0}\geq c_0>0.
\end{equation}

Before we give a proof of Theorem \ref{dl3}, let us now introduce Proposition \ref{DDL2}.

\begin{proposition}\label{DDL2}
	Let $s$ be stated as in Theorem \ref{dl3}. Let \eqref{HEw}-\eqref{chuzhi2} hold. Let $(\bv_{0j}, \rho_{0j}, w_{0j})$ be stated in \eqref{Dss0} and \eqref{Qss0}. For each $j\geq 1$, consider Cauchy problem \eqref{fc} with the initial data $(\bv_{0j}, \rho_{0j}, w_{0j})$. Then for all $j \geq 1$, there exists two positive constants $T^{*}>0$ and ${M}_{2}>0$ ($T^*$ and ${M}_{2}$ only depends on ${M}_{0},s, C_0$ and $c_0$) such that \eqref{fc} has a unique solution $(\bv_{j},\rho_{j})\in C([0,T^*];H_x^s)$, $w_{j}\in C([0,T^*];H_x^{\frac32}) $. To be precise,

	\begin{enumerate}
		\item the solution $\bv_j, \rho_j$ and $w_j$ satisfy the energy estimates
	\begin{equation}\label{Duu0}
		\|\bv_j\|_{L^\infty_{[0,T^*]}H_x^{s}}+\|\rho_j\|_{L^\infty_{[0,T^*]}H_x^{s}}+ \|w_j\|_{L^\infty_{[0,T^*]}H_x^{\frac32}}
		+ \|\nabla w_j\|_{L^\infty_{[0,T^*]}L_x^{8}} \leq {M}_{2},
	\end{equation}
	and
	\begin{equation}\label{Duu00}
		\|\bv_j, \rho_j\|_{L^\infty_{[0,T^*]\times \mathbb{R}^2}} \leq 2+C_0.
	\end{equation}
	\item the solution $\bv_j$ and $\rho_j$ satisfy the Strichartz estimate
	\begin{equation}\label{Duu2}
		\|d\bv_j, d\rho_j\|_{L^4_{[0,T^*]}L_x^\infty} \leq {M}_{2}.
	\end{equation}
	\item  for $s-\frac{3}{4} \leq r \leq \frac{11}{4}$, consider the following linear wave equation
	\begin{equation}\label{Duu21}
		\begin{cases}
			\square_{{g}_j} f_j=0, \qquad [0,T^*]\times \mathbb{R}^2,
			\\
			(f_j,\partial_t f_j)|_{t=0}=(f_{0j},f_{1j}),
		\end{cases}
	\end{equation}
	where $(f_{0j},f_{1j})=(P_{\leq j}f_0,P_{\leq j}f_1)$ and $(f_0,f_1)\in H_x^r \times H^{r-1}_x$. Then there is a unique solution $f_j$ on $[0,T^*]\times \mathbb{R}^2$. Moreover, for $a\leq r-(s-1)$, the following estimates
	\begin{equation}\label{Duu22}
		\begin{split}
			&\|\left< \nabla \right>^{a} {f}_j\|_{L^2_{[0,T^*]} L^\infty_x}
			\leq  {M}_3(\|{f}_0\|_{{H}_x^r}+ \|{f}_1 \|_{{H}_x^{r-1}}),
			\\
			&\|{f}_j\|_{L^\infty_{[0,T^*]} H^{r}_x}+ \|\partial_t {f}_j\|_{L^\infty_{[0,T^*]} H^{r-1}_x} \leq  {M}_3 (\| {f}_0\|_{H_x^r}+ \| {f}_1\|_{H_x^{r-1}}),
		\end{split}
	\end{equation}
	hold, where ${M}_3$ is a universal constant depends on $C_0, c_0, M_0, s$. Above, the similar estimate holds if we replace $\left< \nabla \right>^{a} $ by $\left< \nabla \right>^{a-1} d$.
	\end{enumerate}
\end{proposition}

Based on Proposition \ref{DDL2}, we are ready to prove Theorem \ref{dl3}.
\medskip\begin{proof}[Proof of Theorem \ref{dl3} by using Proposition \ref{DDL2}] 
	
If we set $\bU_j=(p(\rho_j), \bv_j)^{\mathrm{T}}$, by Lemma \ref{sh} and \ref{FC}, for $j \in \mathbb{N}^{+}$ we have
\begin{equation*}
	\begin{split}
		A^0 (\bU_j) \partial_t \bU_j+ A^i (\bU_j) \partial_i \bU_j=0,
		\\
		\partial_t w_j + (\bv_j \cdot \nabla)w_j=0.
	\end{split}
\end{equation*}
Then for any $j, l \in \mathbb{N}^{+}$, we obtain
\begin{equation}\label{ur}
	\begin{split}
		A^0 (\bU_j) \partial_t (\bU_j-\bU_l)+ A^i (\bU_j) \partial_i (\bU_j-\bU_l)=&-\{ A^\alpha (\bU_j)-A^\alpha (\bU_l)\} \partial_\alpha \bU_l,
		\\
		\partial_t (w_j-w_l) + (\bv_j \cdot \nabla)(w_j-w_l)=&-\{(\bv_j-\bv_l) \cdot \nabla\}w_l.
	\end{split}
\end{equation}
Due to \eqref{ur}, using \eqref{Duu0}-\eqref{Duu2}, we can show that
\begin{equation*}
	\| \bU_{j}-\bU_{l}\|_{L^\infty_{[0, T^*]}H^1_x}+\| w_{j}-w_{l}\|_{L^\infty_{[0, T^*]}H^{\frac12}_x} \leq C_{_2} (\| \bv_{0j}-\bv_{0l}\|_{H^{s}}+ \|\rho_{0j}-\rho_{0l}\|_{H^{s}}+ \| w_{0j}-w_{0l}  \|_{H^{\frac32}}).
\end{equation*}
Here $C_{2}$ is a positive constant depending on ${M}_2$. By Lemma \ref{jiaohuan0}, so we get
\begin{equation*}
	\| \bv_{j}-\bv_{l},\rho_{j}-\rho_{l}\|_{L^\infty_{[0, T^*]}H^1_x}+\| w_{j}-w_{l}\|_{L^\infty_{[0, T^*]}H^{\frac12}_x} \leq C_{2} (\| \bv_{0j}-\bv_{0l}\|_{H^{s}}+ \|\rho_{0j}-\rho_{0l}\|_{H^{s}}+ \| w_{0j}-w_{0l}  \|_{H^{\frac12}}).
\end{equation*}
Thus, $\{(\bv_{j}, \rho_{j}, w_{j})\}_{j\in \mathbb{N}^+}$ is a Cauchy sequence in $H^1_x\times H^1_x \times H^{\frac12}_x$. Let the limit be $( \bv, \rho, w)$. Then
\begin{equation}\label{rwe}
	\begin{split}
		(\bv_{j}, \rho_{j}, w_{j})\rightarrow & ( \bv, \rho, w) \quad \quad \ \ \quad \text{in} \ \ H_x^{1} \times H_x^{1} \times H_x^{\frac12}.
	\end{split}
\end{equation}
By using \eqref{Duu0}, a subsequence of $\{(\bv_{j}, \rho_{j}, w_{j})\}_{j\in \mathbb{N}^+}$ is weakly convergent. Therefore, if $m\rightarrow \infty$, then
\begin{equation}\label{rwe0}
	\begin{split}
		(\bv_{j_m}, \rho_{j_m}, w_{j_m}, \nabla w_{j_m})\rightharpoonup & ( \bv, \rho, w, \nabla w) \quad \qquad \ \ \text{in} \ \ H_x^{s} \times H_x^{s} \times H_x^{\frac32} \times L^8_x.
	\end{split}
\end{equation}
Due to \eqref{rwe} and \eqref{rwe0}, then $(\bv,\rho,w,\nabla w)$ is a strong solution of \eqref{fc} in $H_x^{s} \times H_x^{s} \times H_x^{\frac32}\times L^8_x$. Also, points (1), (2), and (3) of theorem \ref{dl3} hold. Therefore, we have finished the proof of Theorem \ref{dl3}.
	
\end{proof}
All that remains is to prove Proposition \ref{DDL2}. We will postpone this proof until Section \ref{keypro}. Instead, we turn to introduce Proposition \ref{DDL3}, which, compared to Proposition \ref{DDL2}, provides a small-data existence result with a smoother vorticity.
\subsection{Proposition \ref{DDL3} for small data}
By the finite propagation speed of system \eqref{fc}, we denote \( c>0 \) as the largest speed. 
\begin{proposition}\label{DDL3}
	Assume $s\in (\frac74,2], \delta\in (0, s-\frac74)$ and \eqref{a0} hold. Let $\delta_1=\frac{s-\frac74}{10}$. For each small, smooth initial data $(\bv_0, \rho_0, w_0)$ supported in $B(0,c+2)$ which satisfies
	\begin{equation}\label{DP30}
		\begin{split}
			&\|\bv_0 \|_{H^{s}}+ \|\rho_0 \|_{H^{s}} +  \| w_0\|_{H^{\frac32+\delta_1}}+ \|\nabla w_0\|_{L^{8}}  \leq \epsilon_3,
		\end{split}
	\end{equation}
	there exists a smooth solution $(\bv, \rho, w)$ to \eqref{fc} on $[-2,2] \times \mathbb{R}^2$ satisfying
	\begin{equation}\label{DP31}
		\|\bv\|_{L^\infty_{[-2,2]}H_x^{s}}+ \|\rho\|_{L^\infty_{[-2,2]}H_x^{s}} +\|w\|_{L^\infty_{[-2,2]}H_x^{\frac32+\delta_1}}+ \|\nabla w\|_{L^\infty_{[-2,2]}L_x^{8}} \leq \epsilon_2.
	\end{equation}
	Furthermore, the solution satisfies the following properties

	\begin{enumerate}
		\item dispersive estimate for $\bv$, $\rho$, and $\bv_+$
	\begin{equation}\label{DP32}
		\|d \bv, d \rho, d \bv_+\|_{L^4_{[-2,2]} C^{\delta}_x} \leq \epsilon_2.
	\end{equation}

	\item For any $t_0 \in [-2,2]$, let $f$ satisfy equation
	\begin{equation}\label{ppp}
		\begin{cases}
			&\square_{g} f=0,
			\\
			&(f,\partial_t f)|_{t=t_0}=(f_0,f_1).
		\end{cases}
	\end{equation} For each $1 \leq r \leq s+1$, the Cauchy problem \eqref{ppp} is well-posed in $H_x^r \times H_x^{r-1}$, and the following estimate holds:
	\begin{equation}\label{DP33}
		\|\left< \nabla \right>^a f\|_{L^4_{[-2,2]} L^\infty_x} \lesssim  \| f_0\|_{H_x^r}+ \| f_1\|_{H_x^{r-1}},\quad \ a<r-\frac34,
	\end{equation}
	and the same estimates hold if we replace $\left< \nabla \right>^a$ by $\left< \nabla \right>^{a-1}d$.
\end{enumerate}
	\end{proposition}

\begin{proof}
		For $\frac74+\delta_1<s$, by replacing the regularity exponents in Proposition \ref{p1} to $s=s, s_0=\frac74+\delta_1$, the conclusion of Proposition \ref{DDL3} follows immediately.
\end{proof}
\begin{remark}
	This Strichartz estimates is stronger than \eqref{Duu2}, since the regularity of vorticity in Proposition \ref{DDL3} is higher than that in Proposition \ref{DDL2}.
\end{remark}

\subsection{Proof of Proposition \ref{DDL2}}\label{keypro}
To prove Proposition \ref{DDL2}, we divide the proof into three parts. First, as presented in subsection \ref{esess}, we obtain a solution to Proposition \ref{DDL2} over a short time interval. Second, we derive robust Strichartz estimates for the low- and mid-frequency components of the solutions within these short intervals, albeit with a loss of derivatives, as demonstrated in subsection \ref{esest}. Finally, in subsection \ref{finalk}, we extend these solutions to a regular time interval. Similarly, the solutions to the linear wave equations can also be extended to this regular time interval, as shown in subsection \ref{finalq}.
\subsubsection{Energy estimates and Strichartz estimates on a short time-interval.} \label{esess}
Take the scaling
\begin{equation*}
	\begin{split}
		& \widetilde{\bv}_{0j}={\bv}_{0j}(Tt,Tx), \quad \widetilde{\rho}_{0j}=\rho_{0j}(Tt,Tx).
	\end{split}
\end{equation*}
Referring to \eqref{sv}, we define
\begin{equation*}
	\widetilde{w}_{0j}=\mathrm{e}^{-\widetilde{\rho}_{0j}}\mathrm{curl} \widetilde{\bv}_{0j}.
\end{equation*}
Thus, we can infer
\begin{equation}\label{Dss2}
	\begin{split}
		\| \widetilde{\bv}_{0j} \|_{\dot{H}^{s}} +  \| \widetilde{\rho}_{0j} \|_{\dot{H}^{s}} \leq  & T^{s-1}\| (\| {\bv}_{0j} \|_{\dot{H}^{s}}+ {\rho}_{0j} \|_{\dot{H}^{s}}),
		\\
		\| \widetilde{w}_{0j}\|_{\dot{H}^{\frac32+\delta_1}} =& \| \mathrm{e}^{-\widetilde{\rho}_{0j}}\mathrm{curl} \widetilde{\bv}_{0j} \|_{\dot{H}^{\frac32+\delta_1}}
		\\
		\leq & T^{\frac12+\delta_{1}} \|\mathrm{e}^{-{\rho}_{0j}}\mathrm{curl} {\bv}_{0j} \|_{\dot{H}^{2+\delta_1}}
		\\
		= & T^{\frac12+\delta_{1}} \|{w}_{0j} \|_{\dot{H}^{\frac32+\delta_1}}
		\\
		\leq & T^{\frac12+\delta_{1}} 2^{\delta_{1} j }\|{w}_{0j} \|_{\dot{H}^{\frac32}}.
	\end{split}
\end{equation}
Similarly, we can deduce
\begin{equation*}\label{Dssa}
	\begin{split}
		\| \nabla \widetilde{w}_{0j}\|_{L^{8}} 
		\leq & T^{\frac74} \|\nabla {w}_{0j} \|_{L^{8}}.
	\end{split}
\end{equation*}
Define
\begin{equation}\label{DTJ}
	T^*_j=2^{-\delta_1j}[1+E(0)]^{-3},
\end{equation}
where $E(0)$ is stated as in \eqref{pu0}. Taking $T$ in \eqref{Dss2} as $T^*_j$, for $s>\frac74$,  \eqref{Dss2} yields
\begin{equation*}\label{pp7}
	\begin{split}
		& \| \widetilde{\bv}_{0j} \|_{\dot{H}^{s}}+\| \widetilde{\rho}_{0j}\|_{\dot{H}^{s}} + \| \widetilde{w}_{0j}\|_{\dot{H}^{\frac32+\delta_1}}+ \|\nabla \widetilde{w}_{0j}\|_{L^{8}}  \leq 2^{-\frac{\delta^2_1}2 j}[1+E(0)]^{-\frac12} .
	\end{split}
\end{equation*}
Note $\| \widetilde{\bv}_{0j}\|_{L^\infty}\leq  \|{\bv}_{0j}\|_{L^\infty}$ and $\| \widetilde{\rho}_{0j}\|_{L^\infty}  \leq  \|{\rho}_{0j}\|_{L^\infty} $. Due to \eqref{pu00}, we have
\begin{equation*}\label{pp60}
	\begin{split}
		\| \widetilde{\bv}_{0j}\|_{L^\infty}+ \| \widetilde{\rho}_{0j}\|_{L^\infty}  \leq  \|{\bv}_{0j}\|_{L^\infty} + \|{\rho}_{0j}\|_{L^\infty} \leq  C_0.
	\end{split}
\end{equation*}
For a small parameter $\epsilon_3$ stated in Proposition \ref{DDL3}, we choose $N_0=N_0(s,M_0,c_0,C_0)$ such that
\begin{equation}\label{pp8}
	\begin{split}
		& 2^{-\frac{\delta^2_1 }2N_0 }[1+E(0)]^{-\frac12} \ll \epsilon_3,
		\\
		&  2^{-\delta_{1} N_0 } (1+C_*)^3\{ 1+ {C^3_*} [E(0)]^{-3}[1+ E(0)]^{-3} \} \leq 1 ,
		\\
		&  C2^{-\frac{3}{4}\delta_1N_0}[E(0)]^{-\frac94} [1+E(0)]^3  [\frac{1}{3}(1-2^{-\delta_{1}})]^{-2} \leq 2 .
	\end{split}	
\end{equation}
Above, $C_*$ is denoted by
\begin{equation}\label{Cstar}
	\begin{split}
		C_*=CE(0)\mathrm{e}^2\exp\{  CE(0) \textrm{e}^2 \}  .
	\end{split}	
\end{equation}
Therefore, for $j \geq N_0$, we have
\begin{equation}\label{Dss3}
	\begin{split}
		\| \widetilde{\bv}_{0j} \|_{\dot{H}^{s}}+\| \widetilde{\rho}_{0j}\|_{\dot{H}^{s}}+\| \widetilde{w}_{0j}\|_{\dot{H}^{\frac32+\delta_1}}+\| \nabla \widetilde{w}_{0j}\|_{L^8}  \leq  \epsilon_3.
	\end{split}
\end{equation}
Since the estimate \eqref{Dss3} is formulated in terms of homogeneous norms, we need to use physical localization techniques to extend these bounds to the inhomogeneous case.

Define a smooth function $\chi\in C_0^\infty(\mathbb{R}^2)$ supported in $B(0,c+2)$, and $\chi$ equals $1$ in $B(0,c+1)$. Given $y \in \mathbb{R}^2$, we define the localized initial data for the velocity and density near $y$:
\begin{equation}\label{yyy0}
	\begin{split}
		\bar{\bv}_{0j}(x)=&\chi(x-y)\left( \widetilde{\bv}_{0j}(x)- \widetilde{\bv}_{0j}(y)\right),
		\\
		\bar{\rho}_{0j}(x)=&\chi(x-y)\left( \widetilde{\rho}_{0j}(x)-\widetilde{\rho}_{0j}(y)\right).
	\end{split}
\end{equation}
Referring to \eqref{sv}, we set
\begin{equation}\label{yyy1}
	\bar{w}_{0j}=
	\mathrm{e}^{-\bar{\rho}_{0j}}\mathrm{curl}\bar{\bv}_{0j}.
\end{equation}
Taking advantage of \eqref{Dss3}, \eqref{yyy0} and \eqref{yyy1}, we obtain
\begin{equation}\label{sS4}
	\begin{split}
		& \| \bar{\bv}_{0j} \|_{{H}^{s}}+\| \bar{\rho}_{0j}\|_{{H}^{s}}+\| \bar{w}_{0j} \|_{{H}^{\frac32+\delta_1}}+\| \nabla \bar{w}_{0j} \|_{{L}^{8}}  
		\\
		\leq  & C(\| \widetilde{\bv}_{0j} \|_{\dot{H}^{s}}+\| \widetilde{\rho}_{0j}\|_{\dot{H}^{s}}+\| \widetilde{w}_{0j} \|_{\dot{H}^{\frac32+\delta_1}}+\| \nabla\widetilde{w}_{0j} \|_{L^{8}})
		\\
		\leq & \epsilon_3.
	\end{split}
\end{equation}
Using Proposition \ref{DDL3}, for every $j$, there is a smooth solution $(\bar{\bv}_j, \bar{\rho}_j, \bar{w}_j)$ on $[-2,2]\times \mathbb{R}^2$ satisfying
\begin{equation}\label{yz0}
	\begin{cases}
		\square_{\widetilde{g}_j} \bar{v}_j^i=-\epsilon^{ia}\mathrm{e}^{\bar{\rho}_j+\tilde{\rho}_{0j}(y)}\widetilde{c}_j^2 \partial_a \bar{w}_j+\widetilde{{Q}}^i,
		\\
		\square_{\widetilde{g}_j} \bar{\rho}_j=\widetilde{\mathcal{D}},
		\\
		\widetilde{\mathbf{T}} \bar{w}_j=0,
		\\
		(\bar{\bv}_j, \bar{\rho}_j, \bar{w}_j)|_{t=0}=(\bar{\bv}_{0j}, \bar{\rho}_{0j}, \bar{w}_{0j}).
	\end{cases}
\end{equation}
Above, the quantities $\widetilde{c}^2_j$ and $\widetilde{{g}_j}$ are defined by
\begin{equation}\label{QRY}
	\begin{split}
		\widetilde{c}^2_j &= c^2_s (\bar{\rho}_j+\tilde{\rho}_{0j}(y)),
		\\
		\widetilde{g}_j &=g({\bar\bv}_j+\tilde{\bv}_{0j}(y), {\bar\rho}_j+\tilde{\rho}_{0j}(y)),
		\\
		\widetilde{\mathbf{T}} &=\partial_t+ ({\bar\bv}_j+\tilde{\bv}_{0j}(y))\cdot \nabla,
	\end{split}
\end{equation}
and $\widetilde{{Q}}^{i}, i=1,2$ and $\widetilde{\mathcal{D}}$ have the same formulations with ${Q}^{i}, \mathcal{D}$ by replacing $({\bv}, {\rho})$ to $(\bar{\bv}_j+\tilde{\bv}_0(y), \bar{\rho}_j+\tilde{\rho}_0(y) )$. Using Proposition \ref{DDL3} again, on $[-2,2]\times \mathbb{R}^2$, we infer
\begin{equation*}\label{Dsee0}
	\|\bar{\bv}_j\|_{L^\infty_{[-2,2]}H_x^{s}}+ \|\bar{\rho}_j\|_{L^\infty_{[-2,2]}H_x^{s}}+ \| \bar{w}_j \|_{L^\infty_{[-2,2]}H_x^{\frac32+\delta_1}}+ \|\nabla \bar{w}_j \|_{L^\infty_{[-2,2]}L_x^{8}} \leq \epsilon_2,
\end{equation*}
and
\begin{equation*}\label{Dsee1}
	\|d\bar{\bv}_j, d\bar{\rho}_j \|_{L^4_{[-2,2]}C^{\delta}_x} \leq \epsilon_2.
\end{equation*}
Furthermore, using \eqref{DP33}, for $1\leq r \leq s+1$, the linear equation\footnote{Here $\widetilde{g}_j$ is given by \eqref{QRY}.}
\begin{equation}\label{Dsee2}
	\begin{cases}
		&\square_{ \widetilde{g}_j } F=0,\qquad \qquad \ \ \qquad [-2,2]\times \mathbb{R}^2,
		\\
		&({F}, \partial_t {F})|_{t=t_0}=({F}_0, {F}_1), \ \ \ t_0 \in [-2,2],
	\end{cases}
\end{equation}
admits a solution ${F} \in C([-2,2],H_x^{r})\times C^1([-2,2],H_x^{r-1})$. Moreover, for $k<r-1$, the following estimate holds\footnote{For all $j$, the initial norm of the initial data \eqref{sS4} is uniformly controlled by the same small parameter $\epsilon_3$, and the regularity of the initial data \eqref{sS4} only depends on $s$, so the constant in \eqref{3s1} is uniform for all $\widetilde{{g}}$ depending on $j$.}:
\begin{equation}\label{3s1}
	\begin{split}
		\|\left< \nabla \right>^{k}F, \left< \nabla \right>^{k-1}d{F}\|_{L^4_{[-2,2]} L^\infty_x} \leq  & C(\| {F}_0\|_{H_x^r}+ \| {F}_1\|_{H_x^{r-1}} ).
	\end{split}
\end{equation}
Note \eqref{yz0}. Then the function $(\bar{\bv}_j+\widetilde{\bv}_{0j}(y), \bar{\rho}_j+\widetilde{\rho}_{0j}(y),  \bar{w}_j)$ is also a solution of the following system
\begin{equation*}\label{yz1}
	\begin{cases}
		\square_{\widetilde{g}_j} ( \bar{\bv}_j+\widetilde{\bv}_{0j}(y) )=-\epsilon^{ia}\mathrm{e}^{\bar{\rho}_j+\tilde{\rho}_{0j}(y)}\widetilde{c}_s^2 \partial_a \bar{w}_j+\widetilde{{Q}}^i,
		\\
		\square_{\widetilde{g}_j} ( \bar{\rho}_j+\widetilde{\rho}_{0j}(y) )=\widetilde{\mathcal{D}},
		\\
		\widetilde{\mathbf{T}}  \bar{w}_j=0,
		\\
		(\bar{\bv}_j+\widetilde{\bv}_{0j}(y), \bar{\rho}_j+\widetilde{\rho}_{0j}(y),  \bar{w}_j )|_{t=0}=(\widetilde{\bv}_{0j}, \widetilde{\rho}_{0j}, \widetilde{w}_{0j}).
	\end{cases}
\end{equation*}
For $y\in \mathbb{R}^2$, we consider the restrictions 
\begin{equation}\label{Dsee3}
	\left( \bar{\bv}_j+\widetilde{\bv}_{0j} (y) \right)|_{\mathrm{K}^y},
	\quad (\bar{\rho}_j+\widetilde{\rho}_{0j} (y) )|_{\mathrm{K}^y},
	\\
	\quad \bar{w}_j|_{\mathrm{K}^y},
\end{equation}
where $$\mathrm{K}^y=\left\{ (t,x): ct+|x-y| \leq c+1, |t| <1 \right\}.$$ Then the restrictions \eqref{Dsee3} solve \eqref{fc} on $\mathrm{K}^y$. Due to finite speed of propagation, a smooth solution $(\bar{\bv}_j+\widetilde{\bv}_{0j}(y), \bar{\rho}_j+ \widetilde{\rho}_{0j}(y), \bar{\bw}_j)$ solves \eqref{fc} on $\mathrm{K}^y$.

Let a function $\psi$ be supported in the unit ball such that
\begin{equation*}
	\textstyle{\sum}_{y \in 3^{-\frac12} \mathbb{Z}^2 } \psi(x-y)=1.
\end{equation*}
Therefore, the function
\begin{equation}\label{Dsee4}
	\begin{split}
		\widetilde{\bv}_j(t,x)  &=\textstyle{\sum}_{y \in 3^{-\frac12} \mathbb{Z}^2 }\psi(x-y) (\bar{\bv}_j+\widetilde{\bv}_{0j}(y)),
		\\
		\widetilde{\rho}_j(t,x)  &=\textstyle{\sum}_{y \in 3^{-\frac12} \mathbb{Z}^2}\psi(x-y) ( \bar{\rho}_j+ \widetilde{\rho}_{0j}(y)),
		\\
		\widetilde{w}_j(t,x)  &=\mathrm{e}^{-\widetilde{\rho}_j}\mathrm{curl}\widetilde{\bv}_j,
	\end{split}
\end{equation}
is a smooth solution of \eqref{fc} on $[-1,1]\times \mathbb{R}^2$ with the initial data $(\widetilde{\bv}_j, \widetilde{\rho}_j, \widetilde{w}_j )|_{t=0}=(\widetilde{\bv}_{0j}, \widetilde{\rho}_{0j}, \widetilde{w}_{0j})$,

On the other hand, the initial data $(\widetilde{\bv}_{0j}, \widetilde{\rho}_{0j}, \widetilde{w}_{0j})$ is a scaling of $(\bv_{0j}, \rho_{0j}, w_{0j})$ with the space-time scale $T^*_j$, and the system \eqref{fc} is scaling-invariant. Therefore, the function
\begin{equation*}
	({\bv}_{j}, {\rho}_{j}, {w}_{j})=(\widetilde{\bv}_{j},\widetilde{\rho}_{j}, \widetilde{w}_{j}) ((T^*_j)^{-1}t,(T^*_j)^{-1}x),
\end{equation*}
is a solution of \eqref{fc} on $[0,T^*_j]\times \mathbb{R}^2$ ($T^*_j$ is defined as in \eqref{DTJ}) with the initial data
\begin{equation*}
	({\bv}_{j}, {\rho}_{j}, {w}_{j})|_{t=0}=({\bv}_{0j}, {\rho}_{0j}, {w}_{0j}).
\end{equation*}
Referring \eqref{Dsee4} and using \eqref{Dsee2}-\eqref{3s1}, we can see
\begin{equation}\label{Dsee5}
	\begin{split}
		\|d\widetilde{\bv}_j, d\widetilde{\rho}_j\|_{L^4_{[0,1]}C^{\delta}_x}
		\leq & \sup_{y \in 3^{-\frac12} \mathbb{Z}^2}  \|d\bar{\bv}_j, d\bar{\rho}_j\|_{L^4_{[0,1]}{C^{\delta}_x}}
		\\
		\leq & C(\|\bar{\bv}_{0j}\|_{H^s}+ \|\bar{\rho}_{0j}\|_{H^s}+ \| \bar{w}_{0j} \|_{H^{\frac32}}+  \| \nabla \bar{w}_{0j} \|_{L^{8}}).
	\end{split}
\end{equation}
By changing of coordinates $(t,x)\rightarrow ((T_j^*)^{-1}t,(T_j^*)^{-1}x)$ for each $j\geq 1$, we get
\begin{equation}\label{Dsee6}
	\begin{split}
		\|d{\bv}_j, d{\rho}_j\|_{L^4_{[0,T^*_j]}C^{\delta}_x}
		\leq & (T^*_j)^{-(\frac34+{\delta})}\|d\widetilde{\bv}_j, d\widetilde{\rho}_j\|_{L^4_{[0,1]}C^{\delta}_x}.
	\end{split}
\end{equation}
Using \eqref{Dsee5}, \eqref{Dsee6}, \eqref{DTJ}, and ${\delta} \in (0,s-\frac74)$, it follows
\begin{equation}\label{Dsee7}
	\begin{split}
		\|d{\bv}_j, d{\rho}_j\|_{L^4_{[0,T^*_j]}C^{\delta}_x}
		\leq	&C (T^*_j)^{-(\frac34+{\delta})}(\|\bar{\bv}_{0j}\|_{H_x^s}+ \|\bar{\rho}_{0j}\|_{H_x^s}+ \| \bar{w}_{0j} \|_{H_x^{\frac32}})
		\\
		\leq & C (T^*_j)^{-(\frac34+{\delta})} \left\{  (T^*_j)^{s-1}\|{\bv}_{0j}, {\rho}_{0j}\|_{\dot{H}_x^s}
		+ (T^*_j)^{\frac32}\| {w}_{0j} \|_{\dot{H}_x^{\frac32}} + (T^*_j)^{\frac74}\|\nabla {w}_{0j} \|_{L_x^{8}} \right\}
		\\
		\leq & C (\|{\bv}_{0j}\|_{H^s}+ \|{\rho}_{0j}\|_{H^s}+ \| {w}_{0j} \|_{H^{\frac32}}+ \| \nabla {w}_{0j} \|_{L^{8}}).
	\end{split}
\end{equation}
Comining \eqref{Dsee7} and \eqref{pu0}, we obtain
\begin{equation}\label{yz4}
	\begin{split}
		\|d{\bv}_j, d{\rho}_j\|_{L^4_{[0,T^*_j]}C^{\delta}_x}
		\leq & C E(0).
	\end{split}
\end{equation}
Set
\begin{equation*}\label{Etr}
	E(T^*_{j})=\|\bv_{j}\|_{L^\infty_{[0,T^*_{j}]} H^{s}_x}+\|\rho_{j}\|_{L^\infty_{[0,T^*_{j}]} H^{s}_x}+
	\|w_{j}\|_{L^\infty_{[0,T^*_{j}]} H^{\frac32}_x}+
	\|\nabla w_{j}\|_{L^\infty_{[0,T^*_{j}]} L^{8}_x}.
\end{equation*}
Due to \eqref{yz4}, \eqref{pp8}, and \eqref{DTJ}, we get
\begin{equation*}\label{Aab}
	\begin{split}
	\|d{\bv}_j, d{\rho}_j\|_{L^4_{[0,T^*_j]}C^{\delta}_x} \leq (T_j^*)^{\frac34}	\|d{\bv}_j, d{\rho}_j\|_{L^4_{[0,T^*_j]}C^{\delta}_x}
		\leq 2.
	\end{split}
\end{equation*}
By using Theorem \ref{bBe}, we have
\begin{align*}
	\label{AMM3}
	& E(T^*_j) \leq C E(0) \mathrm{e}^2\exp\{ C E(0) \mathrm{e}^2\} .
\end{align*}
Similarly, for $1\leq r \leq s+1$, there exists a unique solution for the Cauchy problem
\begin{equation}\label{ru03}
	\begin{cases}
		\square_{{g}_j} F=0, \quad (0,T^*_j]\times \mathbb{R}^2,
		\\
		(F,\partial_t F)|_{t=0}=(F_0,F_1)\in H_x^{r} \times H^{r-1}_x.
	\end{cases}
\end{equation}
Moreover, for $a<r-1$, we have
\begin{equation}\label{ru04}
	\begin{split}
		\|\left<\nabla \right>^{a-1} dF\|_{L^4_{[0,T^*_j]} L^\infty_x}
		\leq & C(\|{F}_0\|_{{H}_x^{r}}+ \| {F}_1 \|_{{H}_x^{r-1}}),
	\end{split}
\end{equation}
and
\begin{equation*}\label{ru05}
	\begin{split}
		\|{F}\|_{L^\infty_{[0,T^*_j]} H^{r}_x}+ \|\partial_t {F}\|_{L^\infty_{[0,T_j]} H^{r-1}_x} \leq  C(\| {F}_0\|_{H_x^{r}}+ \| {F}_1\|_{H_x^{r-1}}).
	\end{split}
\end{equation*}
Gathering the results above, there is a uniform bound for $\bv_{j}, \rho_{j}$ and \(w_j\). However, the length of the interval \([0, T_j^*]\) is not uniform and depends on \(j\). To construct a solution as the limit of these smooth solutions, we must extend them over a uniformly regular time interval.

\subsubsection{A loss of Strichartz estimates on a short time-interval.} \label{esest}
We will discuss it into two cases\footnote{Here, we mainly inspired by Tataru \cite{T1}, Bahouri-Chemin \cite{BC2}, and Ai-Ifrim-Tataru \cite{AIT}. Of course, our work based on the property of Strichartz estimates in Proposition \ref{r6} and careful analysis on vorticity. Moreover, we eventually conclude it by induction method.}, the high frequency and low frequency for $d{\bv}_j$ and $d{\rho}_j$.

\textbf{Case 1: high frequency($k \geq j$).} Due to \eqref{yz4}, we get
\begin{equation}\label{yz6}
	\begin{split}
		& \| d \bv_j, d \rho_j \|_{L^4_{[0,T^*_j]}C^a_x} \leq CE(0), \quad  a \in [0,s-\frac74).
	\end{split}
\end{equation}
For $k \geq j$, using Bernstein inequality and \eqref{yz6}, we have
\begin{equation}\label{yz9}
	\begin{split}
		\| P_{k} d \bv_j, P_{k} d \rho_j \|_{L^4_{[0,T^*_j]}L^\infty_x}
		\leq & 2^{-ka}\| P_k d \bv_j, P_k d \rho_j \|_{L^4_{[0,T^*_j]}C^a_x}
		\\
		\leq & C2^{-ja} \| d \bv_j, d \rho_j \|_{L^4_{[0,T^*_j]}C^a_x}
		\\
		\leq &  C2^{-ja} E(0).
	\end{split}
\end{equation}
Taking $a=9\delta_1$ in \eqref{yz9}, we obtain
\begin{equation}\label{yz10}
	\begin{split}
		\| P_{k} d \bv_j, P_{k} d \rho_j \|_{L^4_{[0,T^*_j]}L^\infty_x}
		\leq & C   2^{-\delta_{1}k} \cdot [1+E(0)]^3 2^{-7\delta_{1}j}, \quad k \geq j,
	\end{split}
\end{equation}
\quad \textbf{Case 2: low frequency($k<j$).} In this case, it's much different from the high frequency. Fortunately, there is some good estimates for difference terms $P_k(d{\rho}_{m+1}-d{\rho}_m)$ and $P_k(d{\bv}_{m+1}-d{\bv}_{m})$. Following \eqref{dvc} and \eqref{etad}, we set
\begin{equation*}\label{etad0}
\begin{split}
	& \bv_m=\bv_{+m}+ \bv_{-m}, \quad \bv_{\pm m}=(v^1_{\pm m},v^2_{\pm m}),
	\\
	& -\Delta v^i_{-m}=\epsilon^{ia}\mathrm{e}^{\rho_m}\partial_a w_m .
\end{split}	
\end{equation*}
We will obtain some good estimates of $P_k(d{\bv}_{m+1}-d{\bv}_m)$ and $P_k(d{\rho}_{m+1}-d{\rho}_m)$ by using \eqref{ru03}-\eqref{ru04}. Indeed, this good estimate is from a loss of derivatives of Strichartz estimate. We claim that
\begin{align}\label{yu0}
	& \|\bv_{m+1}-\bv_{m}, {\rho}_{m+1}-{\rho}_{m}\|_{L^\infty_{ [0,T^*_{m+1}] } L^2_x} \leq 2^{-(s- \delta_1)m}E(0),
	\\\label{yu1}
	& \|w_{m+1}-w_{m}\|_{L^\infty_{ [0,T^*_{m+1}]  } L^2_x} \leq 2^{- (\frac32- \delta_1)m}E(0).
\end{align}
To verify \eqref{yu0} and \eqref{yu1}, we start from the initial data. Applying Bernstein's inequality, we get
\begin{equation}\label{yz12}
	\begin{split}
		\|{\bv}_{0(m+1)}-{\bv}_{0m}\|_{L^2}
		\lesssim & 2^{-sm} \|{\bv}_{0(m+1)}-{\bv}_{0m}\|_{\dot{H}^s}.
	\end{split}	
\end{equation}
Similarly, we obtain
\begin{equation}\label{yz13}
	\begin{split}
		& \|{\rho}_{0(m+1)}-{\rho}_{0m}\|_{L^2} \lesssim   2^{-sm}\|{\rho}_{0(m+1)}-{\rho}_{0m}\|_{\dot{H}^s},
		\\
		& \|{w}_{0(m+1)}-{w}_{0m}\|_{L^2} \lesssim   2^{-{\frac32}m}\|{w}_{0(m+1)}-{w}_{0m}\|_{\dot{H}^{\frac32}}.
	\end{split}	
\end{equation}
Based on \eqref{yz12} and \eqref{yz13}, we can use the structure of the system \eqref{fc} to prove \eqref{yu0}-\eqref{yu1}. Let ${\bU}_{m}=(\bv_{m},p(\rho_{m}))^\mathrm{T} $. Then ${\bU}_{m+1}-{\bU}_{m}$ satisfies
\begin{equation}\label{yz14}
	\begin{cases}
		& A^0({\bU}_{m+1}) \partial_t ( {\bU}_{m+1}- {\bU}_{m}) + A^i({\bU}_{m+1}) \partial_i ( {\bU}_{m+1}- {\bU}_{m})=\Pi_m,
		\\
		& ( {\bU}_{m+1}-{\bU}_{m} )|_{t=0}= {\bU}_{0(m+1)}-{\bU}_{0m},
	\end{cases}
\end{equation}
where
\begin{equation}\label{Fhz}
	\Pi_m=-[A^0({\bU}_{m+1})-A^0({\bU}_{m}) ]\partial_t  {\bU}_{m}- [A^i({\bU}_{m+1})-A^i({\bU}_{m}) ]\partial_i  {\bU}_{m}.
\end{equation}
From \eqref{Fhz}, we can see
\begin{equation*}\label{Fhz0}
	|\Pi_m|\lesssim |{\bU}_{m+1} - {\bU}_{m}| \cdot |  d{\bU}_{m} |.
\end{equation*}
Multiplying ${\bU}_{m+1}- {\bU}_{m}$ on \eqref{yz14} and integrating it on $\mathbb{R}^2$, we have
\begin{equation}\label{yz15}
	\frac{d}{dt}	\|{\bU}_{m+1}-{\bU}_{m} \|^2_{L^2_x} \lesssim   \| (d {\bU}_{m+1}, d{\bU}_{m}) \|_{L^\infty_x}\| {\bU}_{m+1}-{\bU}_{m} \|^2_{L^2_x}.
\end{equation}
Integrating \eqref{yz15} on $[0,T^*_{m+1}]$, using Gronwall's inequality and Strichartz estimate \eqref{yz4}, it yields
\begin{equation*}
	\begin{split}
		\| {\bU}_{m+1}-{\bU}_{m} \|_{L^\infty_{[0,T^*_{m+1} ]}L^2_x} \lesssim &  \| {\bU}_{0(m+1)}-{\bU}_{0m} \|_{L^2_x} .
	\end{split}
\end{equation*}
Due to \eqref{yz12} and \eqref{yz13}, we then have
 \begin{equation*}
 	\begin{split}
 		\| {\bU}_{m+1}-{\bU}_{m} \|_{L^\infty_{[0,T^*_{m+1} ]}L^2_x} \leq 2^{-(s- \delta_1)m}E(0).
 	\end{split}
 \end{equation*}
By using Lemma \ref{jiaohuan0}, we show that
\begin{equation*}
	\begin{split}
		\| {\bv}_{m+1}-{\bv}_{m}, {\rho}_{m+1}-{\rho}_{m} \|_{L^\infty_{[0,T^*_{m+1}]}L^2_x} \leq 2^{-(s- \delta_1)m}E(0).
	\end{split}
\end{equation*}
Thus, the estimate \eqref{yu0} holds. To prove \eqref{yu1}, we consider the transport equation of ${w}_{m+1} - {w}_m$:
\begin{equation}\label{AMM7}
		\partial_t ({w}_{m+1}- {w}_{m}) + ({\bv}_{m+1} \cdot \nabla) ({w}_{m+1}- {w}_{m})
		= -({\bv}_{m+1}-{\bv}_{m}) \cdot \nabla {w}_{m}.
\end{equation}
Multiplying with ${w}_{m+1}- {w}_{m}$ on \eqref{AMM7} and integrating it on $[0,t]\times \mathbb{R}^2$, we derive
\begin{equation}\label{AMM8}
	\begin{split}
		\|{w}_{m+1}- {w}_{m} \|^2_{L^2_x} \leq  & \| {w}_{0(m+1)}- {w}_{0m} \|^2_{L^2_x}+ C\int^t_{0}  \| \nabla {\bv}_{m+1}\|_{L^\infty_x}\| {w}_{m+1}- {w}_{m} \|^2_{L^2_x}d\tau
		\\
		& + C\int^t_{0} \| {\bv}_{m+1}-{\bv}_{m} \|_{L^{\frac83}_x}\| {w}_{m+1}-{w}_{m} \|_{L^2_x}\|\nabla {w}_{m} \|_{L^8_x}d\tau
		\\
		\leq  & \| {w}_{0(m+1)}- {w}_{0m} \|^2_{L^2_x}+ C\int^t_{0}  \| \nabla {\bv}_{m+1}\|_{L^\infty_x}\| {w}_{m+1}- {w}_{m} \|^2_{L^2_x}d\tau
		\\
		& + C\int^t_{0} \| {\bv}_{m+1}-{\bv}_{m} \|_{H^{\frac14}_x}\| {w}_{m+1}-{w}_{m} \|_{L^2_x}\|\nabla {w}_{m} \|_{L^8_x}d\tau.
	\end{split}
\end{equation}
We also note that
\begin{equation}\label{kkk}
	\| {\bv}_{m+1}-{\bv}_{m} \|_{H^{\frac14}_x} \lesssim \| {\bv}_{m+1}-{\bv}_{m} \|^{1-\frac{1}{4s}}_{L^{2}_x} \| {\bv}_{m+1}-{\bv}_{m} \|^\frac{1}{4s}_{H^{s}_x}  .
\end{equation}
Taking advantage of \eqref{yz13}, \eqref{yu0}, \eqref{yz6}, \eqref{AMM8}, and \eqref{kkk}, and using Gronwall's inequality gives
\begin{equation*}
	\begin{split}
		\|{w}_{m+1}- {w}_{m} \|_{L^\infty_{[0,T^*_{m+1}]} L^2_x}  \lesssim   2^{-(\frac32-\delta_1)m} E(0).
	\end{split}
\end{equation*} 
This implies that \eqref{yu1} holds. On the other hand, for $\delta_1<s-\frac74$, using Strichartz estimates \eqref{ru04}(similar to \eqref{SR0}), it yields
\begin{equation}\label{see80}
	\begin{split}
		& \|\nabla({\rho}_{m+1}-{\rho}_{m}) \|_{L^4_{[0,T^*_{m+1}]} C^{\delta_{1}}_x}+\|\nabla ({\bv}_{+(m+1)}-{\bv}_{+m}) \|_{L^4_{[0,T^*_{m+1}]} C^{\delta_{1}}_x}
		\\
		\leq  &  C\| {\rho}_{m+1}-{\rho}_{m}, {\bv}_{m+1}-{\bv}_{m} \|_{L^\infty_{[0,T^*_{m+1}]} H^{\frac74+2\delta_{1}}_x}+\| {w}_{m+1}-{w}_{m} \|_{L^\infty_{[0,T^*_{m+1}]} H^{1+2\delta_{1}}_x}.
	\end{split}
\end{equation}
Noting $s=2+10\delta_{1}$ and using \eqref{yu0}-\eqref{yu1}, we can bound \eqref{see80} by
\begin{equation}\label{see99}
	\begin{split}
		& \|\nabla({\rho}_{m+1}-{\rho}_{m}) \|_{L^4_{[0,T^*_{m+1}]} C^{\delta_{1}}_x}+\|\nabla ({\bv}_{+(m+1)}-{\bv}_{+m}) \|_{L^4_{[0,T^*_{m+1}]} C^{\delta_{1}}_x}
		\leq   2^{-7\delta_1m} E(0).
	\end{split}
\end{equation}
Applying \eqref{yu1} and Sobolev imbeddings, we can establish
\begin{equation}\label{siq}
	\begin{split}
		\|\nabla ({\bv}_{-(m+1)}-{\bv}_{-m}) \|_{L^4_{[0,T^*_{m+1}]} C^{\delta_1}_x}
		\leq & \| \nabla ({\bv}_{-(m+1)}-{\bv}_{-m}) \|_{ L^\infty_{[0,T^*_{m+1}]} C^{\delta_1}_x}
		\\
		\leq  & C \|{w}_{m+1}- {w}_{m} \|_{L^\infty_{[0,T^*_{m+1}]} H^{1+2\delta_{1}}_x}
		\\
		\leq & C  2^{-7\delta_{1} m} [1+E(0)]^2.
	\end{split}
\end{equation}
Adding \eqref{see99} and \eqref{siq}, it yields
\begin{equation}\label{siw}
	\begin{split}
		& \|\nabla( {\rho}_{m+1}-{\rho}_{m} ) \|_{L^4_{[0,T^*_{m+1}]} C^{\delta_{1}}_x}+\| \nabla ({\bv}_{m+1}-{\bv}_{m}) \|_{L^4_{[0,T^*_{m+1}]} C^{\delta_{1}}_x}\leq  C 2^{-7\delta_{1} m} [1+E(0)]^2.
	\end{split}
\end{equation}
By using \eqref{fc} and \eqref{siw}, we obtain
\begin{equation}\label{sie}
	\begin{split}
		& \|\partial_t({\rho}_{m+1}-{\rho}_{m})\|_{L^4_{[0,T^*_{m+1}]} C^{\delta_{1}}_x}+\| \partial_t({\bv}_{m+1}-{\bv}_{m}) \|_{L^2_{[0,T^*_{m+1}]} C^{\delta_{1}}_x}
		\\
		\leq  & \|\nabla({\rho}_{m+1}-{\rho}_{m}) , \nabla ({\bv}_{m+1}-{\bv}_{m}) \|_{L^4_{[0,T^*_{m+1}]} C^{\delta_{1}}_x} \cdot (1+ \|\bv_m, \rho_m\|_{L^\infty_{[0,T^*_{m+1}]} H^s_x} )
		\\
		\leq & C [1+E(0)]^3 2^{-6\delta_{1} m}.
	\end{split}
\end{equation}
Due to \eqref{siw} and \eqref{sie}, for $k<m$, we get
\begin{equation}\label{Sia}
	\|P_k d({\rho}_{m+1}-{\rho}_{m}), P_k d ({\bv}_{m+1}-{\bv}_{m}) \|_{L^4_{[0,T^*_{m+1}]} L^\infty_x}
	\leq   C2^{-\delta_1k}\cdot   2^{-6\delta_1m}[1+E(0)]^3,
\end{equation}
and
\begin{equation*}\label{kz4}
	\|P_k d({\rho}_{m+1}-{\rho}_{m}), P_k d ({\bv}_{m+1}-{\bv}_{m}) \|_{L^1_{[0,T^*_{m+1}]} L^\infty_x}
	\leq   C2^{-\delta_1k} \cdot  2^{-6\delta_1m} [1+E(0)]^3.
\end{equation*}
\subsubsection{Uniform energy and Strichartz estimates on a regular time-interval $[0,T_{N_0}^*]$.}\label{finalk}
Recall \eqref{DTJ}, then $T_{N_0}^*=[E(0)]^{-3}2^{-\delta_{1}N_0}$. Recall \eqref{yz10} and \eqref{Sia}. Therefore, when $k\geq j$ and $j\geq N_0$, using \eqref{yz10} and H\"older's inequality, we have
\begin{equation}\label{kf01}
	\begin{split}
		\| P_{k} d \bv_j, P_{k} d \rho_j \|_{L^1_{[0,T^*_j]}L^\infty_x} \leq & (T^*_j)^{\frac34}\| P_{k} d \bv_j, P_{k} d \rho_j \|_{L^4_{[0,T^*_j]}L^\infty_x}
		\\
		\leq & (T^*_{N_0})^{\frac34} \cdot C  [1+E(0)]^3 2^{-\delta_{1}k} 2^{-7\delta_1j}.
	\end{split}
\end{equation}
When $k< j$ and $m\geq N_0$, using \eqref{Sia} and H\"older's inequality, we obtain
\begin{equation}\label{kf02}
	\|P_k d({\rho}_{m+1}-{\rho}_{m}), P_k d({\bv}_{m+1}-{\bv}_{m}) \|_{L^1_{[0,T^*_{m+1}]} L^\infty_x}
	\leq    (T^*_{N_0})^{\frac34} \cdot C  [1+E(0)]^3 2^{-\delta_{1}k} 2^{-6\delta_1m}.
\end{equation}
On the other side, when $k\geq j$ and $j< N_0$, using \eqref{yz10}, we get
\begin{equation}\label{kf03}
	\begin{split}
		\| P_{k} d \bv_j, P_{k} d \rho_j \|_{L^1_{[0,T^*_{N_0}]}L^\infty_x} \leq & (T^*_{N_0})^{\frac34} \| P_{k} d \bv_j, P_{k} d \rho_j \|_{L^4_{[0,T^*_{N_0}]}L^\infty_x}
		\\
		\leq & (T^*_{N_0})^{\frac34} \| P_{k} d \bv_j, P_{k} d \rho_j \|_{L^4_{[0,T^*_j]}L^\infty_x}
		\\
		\leq & (T^*_{N_0})^{\frac34} \cdot C  [1+E(0)]^3 2^{-\delta_{1}k} 2^{-7\delta_1j}.
	\end{split}
\end{equation}
When $k< j$ and $m< N_0$, using \eqref{Sia} and H\"older's inequality, it yields
\begin{equation}\label{kf04}
	\begin{split}
		& \|P_k d({\rho}_{j+1}-{\rho}_{j}), P_k d({\bv}_{j+1}-{\bv}_{j}) \|_{L^1_{[0,T^*_{N_0}]} L^\infty_x}
		\\
		\leq & (T^*_{N_0})^{\frac34} \|P_k d({\rho}_{j+1}-{\rho}_{j}), P_k d({\bv}_{j+1}-{\bv}_{j}) \|_{L^4_{[0,T^*_{N_0}]} L^\infty_x}
		\\
		\leq & (T^*_{N_0})^{\frac34} \|P_k d({\rho}_{j+1}-{\rho}_{j}), P_k d({\bv}_{j+1}-{\bv}_{j}) \|_{L^4_{[0,T^*_{j+1}]} L^\infty_x}
		\\
		\leq    & (T^*_{N_0})^{\frac34} \cdot C  [1+E(0)]^3 2^{-\delta_{1}k} 2^{-6\delta_1j}.
	\end{split}
\end{equation}
Due to a different time-interval for the sequence  $(\bv_j,\rho_j, w_j)$, we will discuss the solutions $(\bv_j,\rho_j, w_j)$ if  $j \leq N_0$ or  $j \geq N_0+1$ as follows.

\textbf{Case 1: $j \leq N_0$.} In this case, $(\bv_j, \rho_j, w_j)$ exists on $[0,T_j^*]$, and $[0,T^*_{N_0}]\subseteq [0,T^*_{j}]$. As a result, we don't need to extend solutions $(\bv_j, \rho_j, w_j)$ if  $j \leq N_0$. Using \eqref{yz4} and H\"older's inequality, we get
\begin{equation*}
	\| d\bv_j, d\rho_j\|_{L^1_{[0,T^*_{N_0}]}L^\infty_x} \leq (T^*_{N_0})^{\frac34} \| d\bv_j, d\rho_j\|_{L^4_{[0,T^*_{N_0}]}L^\infty_x} \leq C(T^*_{N_0})^{\frac34} [ 1+E(0) ].
\end{equation*}
By \eqref{pp8}, this yields
\begin{equation*}\label{ky0}
	\| d\bv_j, d\rho_j\|_{L^1_{[0,T^*_{N_0}]}L^\infty_x} \leq 2.
\end{equation*}
By using Newton-Leibniz's formula and \eqref{pu00}, it follows that
\begin{equation*}\label{ky1}
	\|\bv_j, \rho_j\|_{L^\infty_{[0,T^*_{N_0}] \times \mathbb{R}^3}}\leq |\bv_{0j}, \rho_{0j}|+ \| d\bv_j, d\rho_j\|_{L^1_{[0,T^*_{N_0}]}L^\infty_x} \leq 2+C_0.
\end{equation*}
Using Theorem \ref{bBe}, we get
\begin{equation*}\label{ky2}
	E(T^*_{N_0}) \leq CE(0)\mathrm{e}^2\exp\{  CE(0)\mathrm{e}^2\} .
\end{equation*}
\quad \textbf{Case 2: $j \geq N_0+1$.} In this case, we expect to extend the time interval $I_1=[0,T^*_j]$ to $[0,T^*_{N_0}]$. Our idea is to use \eqref{yz10}, \eqref{Sia}, and Theorem \ref{bBe} to calculate the energy at time $T_j^*$. Starting at $T_j^*$, we can also get a new time-interval.

We set
\begin{equation*}\label{ti1}
	I_1=[0,T^*_j]=[t_0,t_1], \quad \quad |I_1|=[E(0)]^{-3}2^{-\delta_{1} j}.
\end{equation*}
By using frequency decomposition, we get
\begin{equation}\label{kz03}
	\begin{split}
		d \bv_j= & \textstyle{\sum}^{\infty}_{k=j} d \bv_j+ \textstyle{\sum}^{j-1}_{k=1}P_k d \bv_j
		\\
		=& P_{\geq j}d \bv_j+ \textstyle{\sum}^{j-1}_{k=1} \textstyle{\sum}_{m=k}^{j-1} P_k  (d\bv_{m+1}-d\bv_m)+ \textstyle{\sum}^{j-1}_{k=1}P_k d \bv_k .
	\end{split}
\end{equation}
Similarly, we also have
\begin{equation}\label{kz04}
	\begin{split}
		d \rho_j= & P_{\geq j}d \rho_j+ \textstyle{\sum}^{j-1}_{k=1} \textstyle{\sum}_{m=k}^{j-1} P_k  (d \rho_{m+1}-d \rho_m)+ \textstyle{\sum}^{j-1}_{k=1}P_k d \rho_k.
	\end{split}
\end{equation}
When $k\geq j$, using \eqref{kf01} and \eqref{kf03}, we have\footnote{In the case $j\geq N_0$ we use \eqref{kf01}. In the case $j < N_0$, we take $T^*_j = T^*_{N_0}$ and use \eqref{kf03} to give a bound on $[0,T^*_{N_0}]$.}
\begin{equation}\label{kz01}
	\begin{split}
		\| P_{k} d \bv_j, P_{k} d \rho_j \|_{L^1_{[0, T^*_j]}L^\infty_x} \leq   &   (T^*_{N_0})^{\frac34} \cdot C [ 1+ E(0)]^3 2^{-\delta_{1}k} 2^{-7\delta_{1}j}.
	\end{split}
\end{equation}
When $k< j$, using \eqref{kf02} and \eqref{kf04}, it follows\footnote{In the case $m\geq N_0$ we use \eqref{kf02}. In the case $m < N_0$, we take $T^*_{m+1} = T^*_{N_0}$ and use \eqref{kf04} to give a bound on $[0,T^*_{N_0}]$.}
\begin{equation}\label{kz02}
	\|P_k d({\rho}_{m+1}-{\rho}_{m}), P_k d({\bv}_{m+1}-{\bv}_{m}) \|_{L^1_{[0,T^*_{m+1}]} L^\infty_x}
	\leq    (T^*_{N_0})^{\frac34} \cdot C  [ 1+ E(0)]^3 2^{-\delta_{1}k} 2^{-6\delta_1m}.
\end{equation}
We will use  and \eqref{kz01}-\eqref{kz02} to give a precise analysis of \eqref{kz03}-\eqref{kz04} and get some new time-intervals, and then we try to extend $\rho_j$ from $[0,T^*_j]$ to $[0,T^*_{N_0}]$. Our strategy is as follows. In step 1, we extend it from $[0,T^*_j]$ to $[0,T^*_{j-1}]$. In step 2, we use induction methods to conclude these estimates and also extend it to $[0,T_{N_0}^*]$.

\textbf{Step 1: Extending $[0,T^*_j]$ to $[0,T^*_{j-1}] \ (j \geq N_0+1)$.} To start, referring \eqref{DTJ}, we need to calculate $E(T^*_j)$ for obtaining a length of a new time-interval. Then we shall calculate $\|d \bv_j, d \rho_j\|_{L^1_{[0,T^*_j]}L^\infty_x}$. Using \eqref{kz03} and \eqref{kz04}, we derive that
\begin{equation*}
	\begin{split}
		\|d \bv_j, d \rho_j\|_{L^1_{[0,T^*_j]} L^\infty_x }
		\leq & 	\|P_{\geq j}d\bv_j, P_{\geq j}d \rho_j\|_{L^1_{[0,T^*_j]} L^\infty_x }  + \textstyle{\sum}^{j-1}_{k=1} \|P_k d\bv_k, P_k d\rho_k\|_{L^1_{[0,T^*_j]}L^\infty_x}\\
		& + \textstyle{\sum}^{j-1}_{k=1} \textstyle{\sum}_{m=k}^{j-1}  \|P_k  (d\bv_{m+1}-d\bv_m), P_k  (d\rho_{m+1}-d\rho_m)\|_{L^1_{[0,T^*_j]}L^\infty_x}
		\\
		\leq & 	\|P_{\geq j}d\bv_j, P_{\geq j}d \rho_j\|_{L^1_{[0,T^*_j]}L^\infty_x} + \textstyle{\sum}^{j-1}_{k=1} \|P_k d\bv_k, P_k d\rho_k\|_{L^1_{[0,T^*_k]}L^\infty_x}\\
		& + \textstyle{\sum}^{j-1}_{k=1} \textstyle{\sum}_{m=k}^{j-1}\| P_k  (d\bv_{m+1}-d\bv_m), P_k  (d\rho_{m+1}-d\rho_m)\|_{L^1_{[0,T^*_{m+1}]}L^\infty_x}
		\\
		\leq & C[ 1+ E(0)]^3 (T^*_{N_0})^{\frac34}  \textstyle{\sum}_{k=j}^{\infty} 2^{-\delta_{1} k} 2^{-7\delta_{1} j}
		\\
		& + C[ 1+ E(0)]^3 (T^*_{N_0})^{\frac34} \textstyle{\sum}^{j-1}_{k=1} 2^{-\delta_{1} k} 2^{-7\delta_{1} k}
		\\
		& +  C[ 1+ E(0)]^3 (T^*_{N_0})^{\frac34} \textstyle{\sum}^{j-1}_{k=1} \textstyle{\sum}_{m=k}^{j-1} 2^{-\delta_{1} k} 2^{-6\delta_{1} m}
		\\
		\leq & C[ 1+ E(0)]^3 (T^*_{N_0})^{\frac34} [\frac{1}{3}(1-2^{-\delta_{1}})]^{-2}.
	\end{split}
\end{equation*}
we get
\begin{equation}
	\begin{split}\label{kz05}
		\|d \bv_j, d \rho_j\|_{L^1_{[0,T^*_j]}L^\infty_x}
		\leq  C[ 1+ E(0)]^3(T^*_{N_0})^{\frac34} [\frac{1}{3}(1-2^{-\delta_{1}})]^{-2} \leq 2.
	\end{split}
\end{equation}
By using \eqref{kz05} and \eqref{pu00}, we get
\begin{equation*}
	\begin{split}\label{kp05}
		\| \bv_j,  \rho_j\|_{L^\infty_{[0,T^*_j]}L^\infty_x} \leq \| \bv_{0j},  \rho_{0j}\|+	\|d \bv_j, d \rho_j\|_{L^1_{[0,T^*_j]}L^\infty_x}
		\leq  C_0+2.
	\end{split}
\end{equation*}
By \eqref{kz05}, \eqref{pp8} and Theorem \ref{bBe}, we have
\begin{equation}\label{kz06}
	\begin{split}
		E(T^*_{j}) \leq CE(0)\mathrm{e}^2\exp\{  CE(0) \mathrm{e}^2 \} =C_*.
	\end{split}
\end{equation}
Above, $C_*$ is stated in \eqref{Cstar}. Starting at the time $T^*_j$, seeing \eqref{DTJ} and \eqref{kz06}, we can obtain an extending time-interval with a length of $(C_*)^{-3}2^{-\delta_{1} j}$. But, if $T^*_j + (C_*)^{-3}2^{-\delta_{1} j} \geq T^*_{j-1}$, we have finished this step. Else, we need to extend it again.

\textbf{Case 1:} $T^*_j + (C_*)^{-3}2^{-\delta_{1} j} \geq T^*_{j-1} $. In this case, we can get a new interval
\begin{equation*}\label{deI2}
	I_2=[T^*_j, T^*_{j-1}], \quad |I_2|= (2^{\delta_{1}}-1) [E(0)]^{-3}2^{-\delta_{1} j}.
\end{equation*}
Moreover, referring \eqref{kz01} and \eqref{kz02}, we derive
\begin{equation}\label{kz07}
	\| P_{k} d \bv_j, P_{k} d \rho_j \|_{L^1_{[T^*_j, T^*_{j-1}]}L^\infty_x}
	\leq  (T^*_{N_0})^{\frac34}  \cdot C  (1+C_*)^3 2^{-\delta_{1}k} 2^{-7\delta_{1}j}, \quad k \geq j,
\end{equation}
and $k<j$,
\begin{equation}\label{kz08}
	\|P_k (d{\rho}_{j}-d{\rho}_{j-1}), P_k (d{\bv}_{j}-d{\bv}_{j-1}) \|_{L^1_{[T^*_j, T^*_{j-1}]} L^\infty_x}
	\leq   (T^*_{N_0})^{\frac34}  \cdot C  (1+C_*)^3 2^{-\delta_{1}k} 2^{-6\delta_1(j-1)}.
\end{equation}
Using \eqref{pp8} and $j \geq N_0+1$, \eqref{kz07} and \eqref{kz08} yields
\begin{equation}\label{kz09}
	\| P_{k} d \bv_j, P_{k} d \rho_j \|_{L^1_{[T^*_j, T^*_{j-1}]}L^\infty_x}
	\leq  (T^*_{N_0})^{\frac34}  \cdot C  [1+E(0)]^3 2^{-\delta_{1}k} 2^{-6\delta_{1}j}, \quad k \geq j,
\end{equation}
and $k<j$,
\begin{equation}\label{kz10}
	\|P_k (d{\rho}_{j}-d{\rho}_{j-1}), P_k (d{\bv}_{j}-d{\bv}_{j-1}) \|_{L^1_{[T^*_j, T^*_{j-1}]} L^\infty_x}
	\leq    (T^*_{N_0})^{\frac34}  \cdot C   [1+E(0)]^3 2^{-\delta_{1}k} 2^{-5\delta_1(j-1)}.
\end{equation}
Therefore, we obtain the following estimate
\begin{equation}\label{kz11}
	\begin{split}
		& \|d \bv_j, d\rho_j\|_{L^1_{[0,T^*_{j-1}]} L^\infty_x}
		\\
		\leq & 	\|P_{\geq j}d\bv_j, P_{\geq j}d \rho_j\|_{L^1_{[0,T^*_{j-1}]} L^\infty_x}  + \textstyle{\sum}^{j-1}_{k=1} \|P_k d\bv_k, P_k d\rho_k\|_{L^1_{[0,T^*_{j-1}]} L^\infty_x}
		\\
		& + \textstyle{\sum}^{j-1}_{k=1} \|P_k  (d\bv_{j}-d\bv_{j-1}), P_k  (d\rho_{j}-d\rho_{j-1})\|_{L^1_{[0,T^*_{j-1}]} L^\infty_x}
		\\
		& + \textstyle{\sum}^{j-2}_{k=1} \textstyle{\sum}_{m=k}^{j-2}  \|P_k  (d\bv_{m+1}-d\bv_m), P_k  (d\rho_{m+1}-d\rho_m)\|_{L^1_{[0,T^*_{j-1}]} L^\infty_x}.
	\end{split}
\end{equation}
Due to \eqref{kz01} and \eqref{kz09}, it yields
\begin{equation}\label{kz12}
	\begin{split}
		\|P_{\geq j}d\bv_j, P_{\geq j}d\rho_j\|_{L^1_{[0,T^*_{j-1}]} L^\infty_x}
		\leq & C   [1+E(0)]^3 (T^*_{N_0})^{\frac34}  \textstyle{\sum}^{\infty}_{k=j}2^{-\delta_{1}k} (2^{-7\delta_1 j}+ 2^{-6\delta_{1} j})
		\\
		\leq & C   [1+E(0)]^3  (T^*_{N_0})^{\frac34}  \textstyle{\sum}^{\infty}_{k=j} 2^{-\delta_{1}k} 2^{-6\delta_{1} j}\times 2.
	\end{split}
\end{equation}
Due to \eqref{kz02} and \eqref{kz10}, it follows
\begin{equation}\label{kz13}
	\begin{split}
		& \textstyle{\sum}^{j-1}_{k=1} \|P_k  (d\bv_{j}-d\bv_{j-1}), P_k  (d\rho_{j}-d\rho_{j-1})\|_{L^1_{[0,T^*_{j-1}]} L^\infty_x}
		\\
		\leq & C   [1+E(0)]^3 (T^*_{N_0})^{\frac34}  \textstyle{\sum}^{j-1}_{k=1} 2^{-\delta_{1}k} (2^{-6\delta_1 (j-1)}+ 2^{-5\delta_{1} (j-1)})
		\\
		\leq & C   [1+E(0)]^3 (T^*_{N_0})^{\frac34}  \textstyle{\sum}^{j-1}_{k=1} 2^{-\delta_{1}k} 2^{-5\delta_{1} (j-1)} \times 2.
	\end{split}
\end{equation}
Inserting \eqref{kz12}-\eqref{kz13} into \eqref{kz11}, we derive that
\begin{equation}\label{kz14}
	\begin{split}
		\|d \bv_j, d \rho_j\|_{L^1_{[0,T^*_{j-1}]} L^\infty_x}
		\leq & 	C[1+E(0)]^3 (T^*_{N_0})^{\frac34}   \textstyle{\sum}_{k=j}^{\infty} 2^{-\delta_{1} k} 2^{-6\delta_{1} j}\times 2
		\\
		& + C   [1+E(0)]^3 (T^*_{N_0})^{\frac34}  \textstyle{\sum}^{j-1}_{k=1} 2^{-\delta_{1}k} 2^{-5\delta_{1} (j-1)} \times 2
		\\
		& + C[1+E(0)]^3 (T^*_{N_0})^{\frac34}  \textstyle{\sum}^{j-1}_{k=1} 2^{-\delta_{1} k} 2^{-7\delta_{1} k}
		\\
		& +  C[1+E(0)]^3 (T^*_{N_0})^{\frac34}  \textstyle{\sum}^{j-2}_{k=1} \textstyle{\sum}_{m=k}^{j-2} 2^{-\delta_{1} k} 2^{-6\delta_{1} m}
		\\
		\leq & 	C(1+E^3(0))(T^*_{N_0})^{\frac34}  [\frac13(1-2^{-\delta_{1}})]^{-2}.
	\end{split}
\end{equation}
By using \eqref{kz14}, \eqref{pp8} and Theorem \ref{bBe}, we can establish
\begin{equation}\label{kz15}
	\begin{split}
		&|\bv_j, \rho_j| \leq 2+C_0,
		\\
		& E(T^*_{j-1}) \leq C_*.
	\end{split}
\end{equation}
\quad \textbf{Case 2:} $T^*_j + (C_*)^{-3}2^{-\delta_{1} j} < T^*_{j-1} $. In this case, we record
\begin{equation*}
	I_2=[T^*_j, t_2], \quad |I_2| = (C_*)^{-3}2^{-\delta_{1} j}.
\end{equation*}
Referring \eqref{kz01} and \eqref{kz02}, we can derive
\begin{equation}\label{kz16}
	\| P_{k} d \bv_j, P_{k} d \rho_j \|_{L^1_{I_2}L^\infty_x}
	\leq  (T^*_{N_0})^{\frac34} \cdot C  (1+C_*)^3 2^{-\delta_{1}k} 2^{-7\delta_{1}j}, \quad k \geq j,
\end{equation}
and
\begin{equation}\label{kz17}
	\|P_k (d{\rho}_{j}-d{\rho}_{j-1}), P_k  (d{\bv}_{j}-d{\bv}_{j-1}) \|_{L^1_{I_2} L^\infty_x}
	\leq    (T^*_{N_0})^{\frac34} \cdot C  (1+C_*)^3 2^{-\delta_{1}k} 2^{-6\delta_1(j-1)}, \quad k<j.
\end{equation}
Similarly, on $I_1 \cup I_2$, we have
\begin{equation*}
	\begin{split}
		\|d\bv_j, d\rho_j\|_{L^1_{ I_1 \cup I_2 } L^\infty_x }
		\leq & 	\|P_{\geq j}d\bv_j, P_{\geq j}d \rho_j\|_{L^1_{I_1 \cup I_2} L^\infty_x }   + \textstyle{\sum}^{j-1}_{k=1} \|P_k d\bv_k, P_k d\rho_k\|_{L^1_{[0,T^*_k]} L^\infty_x }
		\\
		& + \textstyle{\sum}^{j-1}_{k=1} \|P_k  (d\bv_{j}-d\bv_{j-1}), P_k  (d\rho_{j}-d\rho_{j-1})\|_{L^1_{I_1 \cup I_2 } L^\infty_x }
		\\
		& + \textstyle{\sum}^{j-2}_{k=1} \textstyle{\sum}_{m=k}^{j-2}  \|P_k  (d\bv_{m+1}-d\bv_m), P_k  (d\rho_{m+1}-d\rho_m)\|_{L^1_{I_1 \cup I_2} L^\infty_x } .
	\end{split}
\end{equation*}
Noting that $ I_1 \cup I_2 \subseteq T^*_k$ when $k \leq j-1$, we therefore obtain
\begin{equation}\label{kz11A}
	\begin{split}
		& \|d \bv_j, d \rho_j\|_{L^1_{ I_1 \cup I_2 } L^\infty_x }
		\\
		\leq & 	\|P_{\geq j}d\bv_j, P_{\geq j}d \rho_j\|_{L^1_{I_1 \cup I_2} L^\infty_x }   + \textstyle{\sum}^{j-1}_{k=1} \|P_k d\bv_k, P_k d\rho_k\|_{L^1_{I_1 \cup I_2 } L^\infty_x }
		\\
		& + \textstyle{\sum}^{j-1}_{k=1} \|P_k  (d \bv_{j}-d \bv_{j-1}), P_k  (d\rho_{j}-d\rho_{j-1})\|_{L^1_{I_1 \cup I_2 } L^\infty_x }
		\\
		& + \textstyle{\sum}^{j-2}_{k=1} \textstyle{\sum}_{m=k}^{j-2}  \|P_k  (d\bv_{m+1}-d\bv_m), P_k  (d\rho_{m+1}-d\rho_m)\|_{L^1_{[0,T^*_{m+1}]} L^\infty_x }
	\end{split}
\end{equation}
Inserting \eqref{kz01}, \eqref{kz02} and \eqref{kz16} and \eqref{kz17} to \eqref{kz11A}, it follows that
\begin{equation}\label{kz11a}
	\begin{split}		
		\|d \bv_j, d \rho_j\|_{L^1_{ I_1 \cup I_2 } L^\infty_x }
		\leq & 	C[1+E(0)]^3 (T^*_{N_0})^{\frac34} \textstyle{\sum}_{k=j}^{\infty} 2^{-\delta_{1} k} 2^{-6\delta_{1} j}\times 2
		\\
		& + C   [1+E(0)]^3 (T^*_{N_0})^{\frac34} \textstyle{\sum}^{j-1}_{k=1} 2^{-\delta_{1}k} 2^{-5\delta_{1} (j-1)} \times 2
		\\
		& + C[1+E(0)]^3 (T^*_{N_0})^{\frac34} \textstyle{\sum}^{j-1}_{k=1} 2^{-\delta_{1} k} 2^{-7\delta_{1} k}
		\\
		& +  C[1+E(0)]^3 (T^*_{N_0})^{\frac34} \textstyle{\sum}^{j-2}_{k=1} \textstyle{\sum}_{m=k}^{j-2} 2^{-\delta_{1} k} 2^{-6\delta_{1} m}
		\\
		\leq & 	C[1+E(0)]^3 (T^*_{N_0})^{\frac34} [\frac13(1-2^{-\delta_{1}})]^{-2}.
	\end{split}
\end{equation}
Applying \eqref{kz11a} and Theorem \ref{bBe}, we get
\begin{equation*}\label{kz15f}
	\begin{split}
		& |\bv_j,\rho_j| \leq 2+C_0,
		\\
		& E(t_2) \leq C_*.
	\end{split}
\end{equation*}
Therefore, we can repeat the process with a length with $(C_*)^{-1}2^{-\delta_{1} j}$ till extending it to $T^*_{j-1}$. Moreover, on every new time-interval with $(C_*)^{-1}2^{-\delta_{1} j}$, the estimates \eqref{kz16} and \eqref{kz17} hold. Set
\begin{equation}\label{times1}
	X_1= \frac{T^*_{j-1}-T^*_j}{C^{-3}_*2^{-\delta_{1} j}}= (2^{\delta_{1}}-1)C^3_* [E(0)]^{-3}.
\end{equation}
Then we need a maximum of $X_1$-times to reach the time $T^*_{j-1}$ both in case 2 (it's also adapt to case 1 for calculating the times). As a result, we can calculate
\begin{equation*}
	\begin{split}
		\|d \bv_j, d \rho_j\|_{L^1_{[0,T^*_{j-1}]} L^\infty_x}
		\leq & 	\|P_{\geq j}d\bv_j, P_{\geq j}d \rho_j\|_{L^1_{[0,T^*_{j-1}]} L^\infty_x}
		+ \textstyle{\sum}^{j-1}_{k=1} \|P_k d\bv_k, P_k d\rho_k\|_{L^1_{[0,T^*_{k}]} L^\infty_x}
		\\
		+ & \textstyle{\sum}^{j-1}_{k=1} \|P_k  (d\bv_{j}-d\bv_{j-1}), P_k  (d\rho_{j}-d\rho_{j-1})\|_{L^1_{[0,T^*_{j-1}]} L^\infty_x}.
		\\
		& + \textstyle{\sum}^{j-2}_{k=1} \textstyle{\sum}_{m=k}^{j-2}  \|P_k  (d\bv_{m+1}-d\bv_m), P_k  (d\rho_{m+1}-d\rho_m)\|_{L^1_{[0,T^*_{m}]} L^\infty_x}
	\end{split}
\end{equation*}	
Due to \eqref{kz01}, \eqref{kz02}, \eqref{kz16}, \eqref{kz17}, and \eqref{pp8}, we obtain
\begin{equation}\label{kzqt}
	\begin{split}
		& \|d \bv_j, d \rho_j\|_{L^1_{[0,T^*_{j-1}]} L^\infty_x}
		\\
		\leq & C[1+E(0)]^3 (T^*_{N_0})^{\frac34} \textstyle{\sum}_{k=j}^{\infty} 2^{-\delta_{1} k} 2^{-6\delta_{1} j}\times (2^{\delta_{1}}-1)C^3_* [E(0)]^{-3}
		\\
		& + C   [1+E(0)]^3 (T^*_{N_0})^{\frac34} \textstyle{\sum}^{j-1}_{k=1} 2^{-\delta_{1}k} 2^{-5\delta_{1} (j-1)} \times (2^{\delta_{1}}-1)C^3_* [E(0)]^{-3}
		\\
		& + C[1+E(0)]^3(T^*_{N_0})^{\frac34} \textstyle{\sum}^{j-1}_{k=1} 2^{-\delta_{1} k} 2^{-7\delta_{1} k}
		\\
		& +  C[1+E(0)]^3(T^*_{N_0})^{\frac34} \textstyle{\sum}^{j-2}_{k=1} \textstyle{\sum}_{m=k}^{j-2} 2^{-\delta_{1} k} 2^{-6\delta_{1} m}
		\\
		\leq & 	C[1+E(0)]^3 (T^*_{N_0})^{\frac34}[\frac13(1-2^{-\delta_{1}})]^{-2}.
	\end{split}
\end{equation}
Therefore, by using \eqref{kzqt} and Theorem \eqref{bBe}, we get
\begin{equation}\label{kz270}
	\begin{split}
		& \| \bv_j, \rho_j \|_{L^\infty_{ [0,T^*_{j-1}]\times \mathbb{R}^3}} \leq 2+C_0,
		\\
		& E(T^*_{j-1}) \leq C_*.
	\end{split}
\end{equation}
At this stage, both in case 1 or case 2, seeing from \eqref{kz270}, \eqref{kzqt}, \eqref{kz14}, \eqref{kz15}, through a maximum of $X_1=(2^{\delta_{1}}-1)C^3_* [E(0)]^{-3}$ times with each length $C_*^{-3}2^{-\delta_{1}j}$ or $(2^{\delta_{1}}-1)[E(0)]^{-3}2^{-\delta_{1} j}$, we shall extend the solutions $(\bv_j,\rho_j,\bw_j)$ from $[0,T^*_j]$ to $[0,T^*_{j-1}]$, and
\begin{equation}\label{kz27}
	\begin{split}
		& \| \bv_j, \rho_j \|_{L^\infty_{ [0,T^*_{j-1}]\times \mathbb{R}^3}} \leq 2+C_0,, \quad E(T^*_{j-1}) \leq C_*,
		\\
		& \|d \bv_j, d \rho_j\|_{L^1_{[0,T^*_{j-1}]} L^\infty_x}
		\leq  	C[1+E(0)]^3 (T^*_{N_0})^{\frac34}[\frac13(1-2^{-\delta_{1}})]^{-2}.
	\end{split}		
\end{equation}
From \eqref{kz27}, we have extended the solutions $(\bv_m,\rho_m,\bw_m)$ ($m\in[N_0, j-1]$) from $[0,T^*_m]$ to $[0,T^*_{m-1}]$. Moreover, referring \eqref{kz01} and \eqref{kz02}, \eqref{kz27}, and \eqref{times1}, we get
\begin{equation}\label{kz29}
	\begin{split}
		& \| P_{k} d \bv_m, P_{k} d \rho_m \|_{L^1_{[0,T^*_{m-1}]}L^\infty_x}
		\\
		\leq  & \| P_{k} d\bv_m, P_{k} d \rho_m \|_{L^1_{[0,T^*_{m}]}L^\infty_x}
		\\
		& +\| P_{k} d \bv_m, P_{k} d \rho_m \|_{L^1_{[T^*_{m},T^*_{m-1}]}L^\infty_x}
		\\
		\leq  & C[1+E(0)]^3(T^*_{N_0})^{\frac34} 2^{-\delta_{1}k} 2^{-7\delta_{1}m} + C(1+C_*)^3 2^{-\delta_{1}k} 2^{-7\delta_{1}m} \times (2^{\delta_{1}}-1)C^3_* [E(0)]^{-3}
		\\
		\leq  & C[1+E(0)]^3(T^*_{N_0})^{\frac34} 2^{-\delta_{1}k} 2^{-7\delta_{1}m} + C[1+E(0)]^3 2^{-\delta_{1}k} 2^{-6\delta_{1}m} \times (2^{\delta_{1}}-1).
	\end{split}	
\end{equation}
For $k \geq m\geq N_0+1$, due to \eqref{pp8} and \eqref{kz29}, it yields
\begin{equation}\label{kz31}
	\begin{split}
		\| P_{k} d \bv_m, P_{k} d \rho_m \|_{L^1_{[0,T^*_{m-1}]}L^\infty_x}
		\leq & C[1+E(0)]^3(T^*_{N_0})^{\frac34} 2^{-\delta_{1}k} 2^{-5\delta_{1}m} \times 2^{\delta_{1}},
	\end{split}	
\end{equation}
Similarly, if $m\geq N_0+1$, using \eqref{pp8}, for $k<j$, the following estimate holds:
\begin{equation}\label{kz34}
	\begin{split}
		& \|P_k (d{\rho}_{m}-d{\rho}_{m-1}), P_k (d{\bv}_{m}-d{\bv}_{m-1}) \|_{L^1_{[0,T^*_{m-1}]} L^\infty_x}
		\\
		\leq    &  \|P_k (d{\rho}_{m}-d{\rho}_{m-1}), P_k  (d{\bv}_{m}-d{\bv}_{m-1}) \|_{L^1_{[0,T^*_{m}]} L^\infty_x}
		\\
		& +\|P_k (d {\rho}_{m}-d {\rho}_{m-1}), P_k (d {\bv}_{m}-d {\bv}_{m-1}) \|_{L^1_{[T^*_m,T^*_{m-1}]} L^\infty_x}
		\\
		\leq & C  [1+E(0)]^3 (T^*_{N_0})^{\frac34} 2^{-\delta_{1}k} 2^{-6\delta_1(m-1)}
		\\
		& + C  (1+C_*)^3 2^{-\delta_{1}k} 2^{-6\delta_1(m-1)} \times (2^{\delta_{1}}-1)C^3_* [E(0)]^{-3}
		\\
		\leq & C[1+E(0)]^3(T^*_{N_0})^{\frac34} 2^{-\delta_{1}k} 2^{-5\delta_{1}(m-1)} \times 2^{\delta_{1}}.
	\end{split}	
\end{equation}
\quad \textbf{Step 2: Extending time interval $[0,T^*_j]$ to $[0,T^*_{N_0}]$}. Based the above analysis in Step 1, we can give a induction by achieving the goal. We assume the solutions $(\bv_j, \rho_j, w_j)$ can be extended from $[0,T^*_j]$ to $[0,T^*_{j-l}]$ through a maximam $X_l$ times and
\begin{equation}\label{kz41}
	\begin{split}
		X_l=& \frac{T_{j-l}^*- T_{j}^*}{C^{-3}_* 2^{-\delta_{1} j}}
		\\
		=& \frac{E(0)^{-1}( 2^{-\delta_{1}(j-l)} - 2^{-\delta_{1}j} )}{C^{-3}_* 2^{-\delta_{1} j}}
		\\
		=& \frac{C^3_*}{[E(0)]^3} (2^{\delta_{1}l}-1).
	\end{split}	
\end{equation}
Moreover, the following bounds
\begin{equation}\label{kz42}
	\begin{split}
		\| P_{k} d \bv_j, P_{k} d \rho_j \|_{L^1_{[0,T^*_{j-l}]}L^\infty_x}
		\leq & C[1+E(0)]^3 (T^*_{N_0})^{\frac34} 2^{-\delta_{1}k} 2^{-5\delta_{1}j} \times 2^{\delta_{1}l}, \qquad k \geq j,
	\end{split}	
\end{equation}
and
\begin{equation}\label{kz43}
	\begin{split}
		& \|P_k (d{\rho}_{m+1}-d{\rho}_{m}), P_k (d{\bv}_{m+1}-d{\bv}_{m}) \|_{L^1_{[0,T^*_{m-l}]} L^\infty_x}
		\\
		\leq  & C[1+E(0)]^3 (T^*_{N_0})^{\frac34} 2^{-\delta_{1}k} 2^{-5\delta_{1}m } \times 2^{\delta_{1}l}, \quad k<j,
	\end{split}	
\end{equation}
and
\begin{equation}\label{kz44}
	\|d \bv_j, d \rho_j\|_{L^1_{[0,T^*_{j-l}]} L^\infty_x}
	\leq  	C[1+E(0)]^3 (T^*_{N_0})^{\frac34}[\frac13(1-2^{-\delta_{1}})]^{-2}.
\end{equation}
and
\begin{equation}\label{kz45}
	|\bv_j, \rho_j| \leq 2+C_0, \quad E(T^*_{j-l}) \leq C_*.
\end{equation}
In the following, we will check the estimates \eqref{kz41}, \eqref{kz42}, \eqref{kz43}, \eqref{kz44}, and \eqref{kz45} hold when $l=1$, and it also holds when we replace $l$ by $l+1$.

Using \eqref{kz27}, \eqref{kz31}, \eqref{times1}, and \eqref{kz34}, then \eqref{kz42}-\eqref{kz45} hold by taking $l=1$. Let us now check it for $l+1$. In this case, it implies that  $T^*_{j-l} \leq T^*_{N_0}$. Therefore, $j-l\geq N_0+1$ should hold. Starting at the time $T^*_{j-l}$, seeing \eqref{DTJ} and \eqref{kz45}, we shall get an extending time-interval of $(\bv_j,\rho_j, w_j)$ with a length of $(C_*)^{-3}2^{-\delta_{1} j}$. Hence, we can go to the case 2 in step 1, and the length every new time-interval is $C_*^{-3} 2^{-\delta_{1} j}$. Therefore, the times is
\begin{equation}\label{kz46}
	X= \frac{T^*_{j-(l+1)}-T^*_{j-l}}{C^{-3}_*2^{-\delta_{1} j}}= 2^{\delta_{1} l}(2^{\delta_{1}}-1)C^3_* [E(0)]^{-3}.
\end{equation}
Thus, we can deduce that
\begin{equation}\label{kz40}
	X_{l+1}=X_l+X= (2^{\delta_{1}(l+1)}-1)C^3_* [E(0)]^{-3}.
\end{equation}
Moreover, for $k \geq j$, we have
\begin{equation*}\label{kz48}
	\begin{split}
		\| P_{k} d \bv_j, P_{k} d\rho_j \|_{L^1_{[0, T^*_{j-(l+1)}]}L^\infty_x}\leq & \| P_{k} d \bv_j, P_{k} d\rho_j \|_{L^1_{[0,T^*_{j-l}]}L^\infty_x}
		\\
		& +	\| P_{k} d \bv_j, P_{k} d\rho_j \|_{L^1_{[T^*_{j-l}, T^*_{j-(l+1)}]}L^\infty_x}.
	\end{split}
\end{equation*}
Using \eqref{kz01} and \eqref{kz45}
\begin{equation}\label{kz49}
	\begin{split}
		& \| P_{k} d \bv_j, P_{k} d\rho_j \|_{L^1_{[T^*_{j-l}, T^*_{j-(l+1)}]}L^\infty_x}
		\\
		\leq   & C  (1+C_*)^3(T^*_{N_0})^{\frac34} 2^{-\delta_{1}k} 2^{-7\delta_{1}j}\times 2^{\delta_{1} l}(2^{\delta_{1}}-1)C^3_* [E(0)]^{-3}, \quad k\geq j.
	\end{split}
\end{equation}
Due to \eqref{pp8}, it yields
\begin{equation*}
	\begin{split}
		(1+C_*)^3 C^3_* [E(0)]^{-3} [1+E(0)]^{-3}2^{-\delta_{1} N_0} \leq 1.
	\end{split}
\end{equation*}
Hence, from \eqref{kz49} we have
\begin{equation}\label{kz50}
	\begin{split}
		& \| P_{k} d \bv_j, P_{k} d\rho_j \|_{L^1_{[T^*_{j-l}, T^*_{j-(l+1)}]}L^\infty_x}
		\\
		\leq   & C  [1+E(0)]^3(T^*_{N_0})^{\frac34} 2^{-\delta_{1}k} 2^{-5\delta_{1}j}\times 2^{\delta_{1} l}(2^{\delta_{1}}-1), \quad k\geq j.
	\end{split}
\end{equation}
Using \eqref{kz42} and \eqref{kz50}, so we get
\begin{equation}\label{kz51}
	\begin{split}
		\| P_{k} d \bv_j, P_{k} d \rho_j \|_{L^1_{[0,T^*_{j-(l+1)}]}L^\infty_x}
		\leq & C[1+E(0)]^3 (T^*_{N_0})^{\frac34} 2^{-\delta_{1}k} 2^{-5\delta_{1}j} \times 2^{\delta_{1}(l+1)}, \qquad k \geq j.
	\end{split}	
\end{equation}
If $k<j$, we can derive
\begin{equation}\label{kz52}
	\begin{split}
		& \|P_k (d{\rho}_{m+1}-d{\rho}_{m}), P_k (d{\bv}_{m+1}-d{\bv}_{m}) \|_{L^1_{[0,T^*_{m-(l+1)}]} L^\infty_x}
		\\
		\leq  & \|P_k (d{\rho}_{m+1}-d{\rho}_{m}), P_k (d{\bv}_{m+1}-d{\bv}_{m}) \|_{L^1_{[0,T^*_{m-l}]} L^\infty_x}
		\\
		& + \|P_k (d{\rho}_{m+1}-d{\rho}_{m}), P_k (d{\bv}_{m+1}-d{\bv}_{m}) \|_{L^1_{[T^*_{m-l},T^*_{m-(l+1)}]} L^\infty_x}.
	\end{split}	
\end{equation}
When we extend the solutions $(\bv_j, \rho_j, w_j)$ from $[0,T^*_{j-l}]$ to $[0,T^*_{j-(l+1)}]$, then the solutions $(\bv_m, \rho_m, w_m)$ is also extended from $[0,T^*_{m-l}]$ to $[0,T^*_{m-(l+1)}]$. Seeing \eqref{kz02} and \eqref{kz45}, we can obtain
\begin{equation}\label{kz53}
	\begin{split}
		& \|P_k (d{\rho}_{m+1}-d{\rho}_{m}), P_k (d{\bv}_{m+1}-d{\bv}_{m}) \|_{L^1_{[T^*_{m-l},T^*_{m-(l+1)}]} L^\infty_x}
		\\
		\leq  & C(1+C_*)^3 (T^*_{N_0})^{\frac34} 2^{-\delta_{1}k} 2^{-6\delta_{1}m } \times 2^{\delta_{1} l}(2^{\delta_{1}}-1)C^3_* [E(0)]^{-3}.
	\end{split}	
\end{equation}
Hence, for $m-l \geq N_0+1$, using \eqref{pp8}, it yields
\begin{equation*}
	(1+C_*)^3  C^3_* [E(0)]^{-3} [1+E(0)]^{-3} 2^{-\delta_{1}N_0 } \leq 1.
\end{equation*}
Based on the above results, \eqref{kz53} becomes
\begin{equation}\label{kz54}
	\begin{split}
		& \|P_k (d{\rho}_{m+1}-d{\rho}_{m}), P_k (d{\bv}_{m+1}-d{\bv}_{m}) \|_{L^1_{[T^*_{m-l},T^*_{m-(l+1)}]} L^\infty_x}
		\\
		\leq  & C[1+E(0)]^3(T^*_{N_0})^{\frac34} 2^{-\delta_{1}k} 2^{-5\delta_{1}m } \times 2^{\delta_{1} l}(2^{\delta_{1}}-1).
	\end{split}	
\end{equation}
Inserting \eqref{kz43} and \eqref{kz54} to \eqref{kz52}, we have
\begin{equation}\label{kz55}
	\begin{split}
		& \|P_k (d{\rho}_{m+1}-d{\rho}_{m}), P_k (d{\bv}_{m+1}-d{\bv}_{m}) \|_{L^1_{[0,T^*_{m-(l+1)}]} L^\infty_x}
		\\
		\leq  & C[1+E(0)]^3(T^*_{N_0})^{\frac34} 2^{-\delta_{1}k} 2^{-5\delta_{1}m } \times 2^{\delta_{1}(l+1)}, \quad k<j,
	\end{split}	
\end{equation}
Next, we can bound 
\begin{equation}\label{kz56}
	\begin{split}
		& \|d \bv_j, d \rho_j\|_{L^1_{[0,T^*_{j-(l+1)}]} L^\infty_x}
		\\
		\leq & \|P_{\geq j}d\bv_j, P_{\geq j}d \rho_j\|_{L^1_{[0,T^*_{j-(l+1)}]} L^\infty_x}
		\\
		& + \textstyle{\sum}^{j-1}_{k=j-l}\textstyle{\sum}^{j-1}_{m=k} \|P_k  (d\bv_{m+1}-d\bv_{m}), P_k  (d\rho_{m+1}-d\rho_{m})\|_{L^1_{[0,T^*_{j-(l+1)}]} L^\infty_x}
		\\
		& + \textstyle{\sum}^{j-1}_{k=1} \|P_k d\bv_k, P_k d\rho_k\|_{L^1_{[0,T^*_{j-(l+1)}]} L^\infty_x}
		\\
		=& 	\|P_{\geq j}d\bv_j, P_{\geq j}d \rho_j\|_{L^1_{[0,T^*_{j-(l+1)}]} L^\infty_x}
		\\
		& + \textstyle{\sum}^{j-1}_{k=j-l} \|P_k d\bv_k, P_k d\rho_k\|_{L^1_{[0,T^*_{j-(l+1)}]} L^\infty_x}
		+ \textstyle{\sum}^{j-(l+1)}_{k=1} \|P_k d\bv_k, P_k d\rho_k\|_{L^1_{[0,T^*_{j-(l+1)}]} L^\infty_x}
		\\
		& + \textstyle{\sum}^{j-1}_{k=1} \textstyle{\sum}^{j-1}_{m=j-(l+1)} \|P_k  (d\bv_{m+1}-d\bv_{m}), P_k  (d\rho_{m+1}-d\rho_{m})\|_{L^1_{[0,T^*_{j-(l+1)}]} L^\infty_x}
		\\
		& + \textstyle{\sum}^{j-(l+2)}_{k=1} \textstyle{\sum}_{m=k}^{j-(l+2)}  \|P_k  (d\bv_{m+1}-d\bv_m), P_k  (d\rho_{m+1}-d\rho_m)\|_{L^1_{[0,T^*_{j-(l+1)}]} L^\infty_x}
		\\
		=& \Theta_1+  \Theta_2+ \Theta_3+ \Theta_4+ \Theta_5,
	\end{split}
\end{equation}
where
\begin{equation}\label{Theta}
	\begin{split}
		\Theta_1= &  \|P_{\geq j}d\bv_j, P_{\geq j}d \rho_j\|_{L^1_{[0,T^*_{j-(l+1)}]} L^\infty_x} ,
		\\
		\Theta_2= & \textstyle{\sum}^{j-(l+2)}_{k=1}\textstyle{\sum}^{j-(l+2)}_{m=k} \|P_k  (d\bv_{m+1}-d\bv_{m}), P_k  (d\rho_{m+1}-d\rho_{m})\|_{L^1_{[0,T^*_{j-(l+1)}]} L^\infty_x},
		\\
		\Theta_3= & \textstyle{\sum}^{j-1}_{k=1}\textstyle{\sum}^{j-1}_{m=j-(l+1)} \|P_k  (d\bv_{m+1}-d\bv_{m}), P_k  (d\rho_{m+1}-d\rho_{m})\|_{L^1_{[0,T^*_{j-(l+1)}]} L^\infty_x},
		\\
		\Theta_4 =& \textstyle{\sum}^{j-(l+1)}_{k=1}	\|P_{k}d\bv_k, P_{k}d\rho_k\|_{L^1_{[0,T^*_{j-(l+1)}]} L^\infty_x},
		\\
		\Theta_5 =& \textstyle{\sum}^{j-1}_{k=j-l}	\|P_{k}d\bv_k, P_{k}d\rho_k\|_{L^1_{[0,T^*_{j-(l+1)}]} L^\infty_x}.
	\end{split}
\end{equation}
On time-interval $[0,T^*_{j-(l+1)}]$, we note that there is no growth for $\Theta_2$ and $\Theta_4$ in this extending process. For example, considering $\Theta_2$, the existing time-interval of $P_k  (d\bv_{m+1}-d\bv_{m})$ is actually $[0,T^*_{m+1}]$, and $[0,T^*_{j-(l+1)}] \subseteq [0,T^*_{m+1}]$ if $m \geq j-(l+2)$. Therefore, we can use the bounds \eqref{kz01} and \eqref{kz02} to handle $\Theta_2$ and $\Theta_4$. While, considering $\Theta_1, \Theta_3$, and $\Theta_5$, we need to calculate the growth in Strichartz estimates. Based on this idea, let us give a precise analysis on \eqref{Theta}.

According to \eqref{kz51}, we can estimate $\Theta_1$ by
\begin{equation}\label{kz57}
	\begin{split}
		\Theta_1 \leq & C[1+E(0)]^3(T^*_{N_0})^{\frac34} \textstyle{\sum}^{\infty}_{k=j} 2^{-\delta_{1}k} 2^{-5\delta_{1}j} \times 2^{\delta_{1}(l+1)}.
	\end{split}
\end{equation}
Due to \eqref{kz01}, we have
\begin{equation}\label{kz58}
	\begin{split}
		\Theta_2\leq &	\textstyle{\sum}^{j-(l+2)}_{k=1}\textstyle{\sum}^{j-(l+2)}_{m=k} \|P_k  (d\bv_{m+1}-d\bv_{m}), P_k  (d\rho_{m+1}-d\rho_{m})\|_{L^1_{[0,T^*_{m+1}]} L^\infty_x},
		\\
		\leq &  C[1+E(0)]^3 (T^*_{N_0})^{\frac34} \textstyle{\sum}^{j-(l+2)}_{k=1}\textstyle{\sum}^{j-(l+2)}_{m=k} 2^{-\delta_{1}k} 2^{-6\delta_{1}m}.
	\end{split}	
\end{equation}
For $1 \leq k \leq j-1$ and $m \leq j-1$, using \eqref{kz55}, it follows
\begin{equation}\label{kz59}
	\begin{split}
		\Theta_3\leq & \textstyle{\sum}^{j-1}_{k=1}\textstyle{\sum}^{j-1}_{m=j-(l+1)} \|P_k  (d\bv_{m+1}-d\bv_{m}), P_k  (d\rho_{m+1}-d\rho_{m})\|_{L^1_{[0,T^*_{m-(l+1)}]} L^\infty_x}
		\\
		\leq &  C  [1+E(0)]^3 (T^*_{N_0})^{\frac34} \textstyle{\sum}^{j-1}_{k=1}\textstyle{\sum}^{j-1}_{m=j-(l+1)}  2^{-\delta_{1}k} 2^{-5\delta_{1}m} \times 2^{\delta_{1}(l+1)}.
	\end{split}
\end{equation}
Due to \eqref{kz01}, it yields
\begin{equation}\label{kz60}
	\begin{split}
		\Theta_4 \leq & \textstyle{\sum}^{j-(l+1)}_{k=1}	\|P_{k}d\bv_k, P_{k}d\rho_k\|_{L^1_{[0,T^*_{k}]} L^\infty_x}
		\\
		\leq & C  [1+E(0)]^3 (T^*_{N_0})^{\frac34} \textstyle{\sum}^{j-(l+1)}_{k=1} 2^{-\delta_{1}k} 2^{-7\delta_{1}k}.
	\end{split}
\end{equation}
If $j-l \leq k\leq j-1$, then $k+l+1-j \leq l$. By using \eqref{kz42}, we can estimate
\begin{equation}\label{kz61}
	\begin{split}
		\Theta_5 =& \textstyle{\sum}^{j-1}_{k=j-l}	\|P_{k}d\bv_k, P_{k}d\rho_k\|_{L^1_{[0,T^*_{j-(l+1)}]} L^\infty_x}
		\\
		= & \textstyle{\sum}^{j-1}_{k=j-l}	\|P_{k}d\bv_k, P_{k}d\rho_k\|_{L^1_{[0,T^*_{k-(k+l+1-j)}]} L^\infty_x}
		\\
		\leq	& C  [1+E(0)]^3 (T^*_{N_0})^{\frac34} \textstyle{\sum}^{j-1}_{k=j-l} 2^{-\delta_{1}k} 2^{-5\delta_{1}j}2^{\delta_{1}(k+l+1-j)}.
	\end{split}
\end{equation}
Inserting \eqref{kz57}, \eqref{kz58}, \eqref{kz59}, \eqref{kz60}, \eqref{kz61} to \eqref{kz56}, it follows
\begin{equation}\label{kz62}
	\begin{split}
		& \|d \bv_j, d \rho_j\|_{L^1_{[0,T^*_{j-(l+1)}]} L^\infty_x}
		\\
		\leq & C[1+E(0)]^3 (T^*_{N_0})^{\frac34} \textstyle{\sum}^{\infty}_{k=j} 2^{-\delta_{1}k} 2^{-5\delta_{1}j} \times 2^{\delta_{1}(l+1)}
		\\
		& + C[1+E(0)]^3(T^*_{N_0})^{\frac34} \textstyle{\sum}^{j-(l+2)}_{k=1}\textstyle{\sum}^{j-(l+2)}_{m=k} 2^{-\delta_{1}k} 2^{-6\delta_{1}m}
		\\
		&+C  [1+E(0)]^3(T^*_{N_0})^{\frac34} \textstyle{\sum}^{j-1}_{k=1}\textstyle{\sum}^{j-1}_{m=j-(l+1)}  2^{-\delta_{1}k} 2^{-5\delta_{1}m} \times 2^{\delta_{1}(l+1)}
		\\
		& + C  [1+E(0)]^3(T^*_{N_0})^{\frac34} \textstyle{\sum}^{j-(l+1)}_{k=1} 2^{-\delta_{1}k} 2^{-7\delta_{1}k}
		\\
		& + C  [1+E(0)]^3(T^*_{N_0})^{\frac34} \textstyle{\sum}^{j-1}_{k=j-l} 2^{-\delta_{1}k} 2^{-5\delta_{1}k}2^{\delta_{1}(k+l+1-j)}
	\end{split}
\end{equation}
In the case of $j-l \geq N_0+1$ and $j\geq N_0+1$, the estimate \eqref{kz62} yields
\begin{equation}\label{kz63}
	\begin{split}
		& \|d \bv_j, d \rho_j\|_{L^1_{[0,T^*_{j-(l+1)}]} L^\infty_x}
		\\
		\leq & C[1+E(0)]^3(T^*_{N_0})^{\frac34}(1-2^{-\delta_{1}})^{-2} \big\{
		2^{-\delta_{1}j} 2^{\delta_{1}(l+1)}+ 2^{-6\delta_{1}}
		\\
		& \quad +2^{-5\delta_{1}[j-(l+1)]}  2^{\delta_{1}(l+1)}+ 2^{-6\delta_{1}}+ 2^{-5\delta_{1}(j-l)}2^{\delta_{1}(l+1-j)}
		\big\}
		\\
		\leq & C[1+E(0)]^3(T^*_{N_0})^{\frac34}(1-2^{-\delta_{1}})^{-2} \left\{
		2^{-6\delta_{1}N_0} + 2^{-6\delta_{1}} +	2^{-5\delta_{1}N_0} + 2^{-6\delta_{1}}+ 	2^{-6\delta_{1}N_0} 	\right\}
		\\
		\leq & C[1+E(0)]^3(T^*_{N_0})^{\frac34}[\frac13(1-2^{-\delta_{1}})]^{-2}.
	\end{split}
\end{equation}
By using \eqref{pu00}, \eqref{pp8}, \eqref{kz63}, and Theorem \ref{bBe}, we have proved
\begin{equation}\label{kz64}
	\begin{split}
		|\bv_j, \rho_j|\leq 2+C_0, \quad  E(T^*_{j-(l+1)}) \leq C_*.
	\end{split}
\end{equation}
Gathering \eqref{kz40}, \eqref{kz51}, \eqref{kz55}, \eqref{kz63} and \eqref{kz64}, the estimates \eqref{kz41}-\eqref{kz45} hold for $l+1$. Thus, our induction hold \eqref{kz41}-\eqref{kz45} for $l=1$ to $l=j-N_0$. Therefore, we can extend the solutions $(\bv_j,\rho_j,\bw_j)$ from $[0,T^*_j]$ to $[0,T^*_{N_0}]$ when $j \geq N_0$. We denote
\begin{equation}\label{Tstar}
	T^*=\min\{ 1, T^*_{N_0} \}=\min\{ 1,  [E(0)]^{-3}2^{-\delta_{1} N_0} \}.
\end{equation}
Taking $l=j-N_0$ in \eqref{kz44}-\eqref{kz45}, we therefore get
\begin{equation}\label{kz65}
	\begin{split}
		& E(T^*) \leq  C_*, \quad \|\bv_j, \rho_j\|_{L^\infty_{[0,T^*]} L^\infty_x} \leq 2+C_0,
		\\
		& \|d \bv_j, d \rho_j\|_{L^1_{[0,T^*]} L^\infty_x}
		\leq  C[1+E(0)]^3(T^*_{N_0})^{\frac34}[\frac13(1-2^{-\delta_{1}})]^{-2} \leq 2,
	\end{split}
\end{equation}
where $N_0$ and $C_*$ (depending on $C_0, c_0, s, M_0$) are denoted in \eqref{pp8} and \eqref{Cstar}. Similarly, we conclude
\begin{equation}\label{kz66}
	\begin{split}
		\|d \bv_j, d \rho_j\|_{L^4_{[0,T^*]} L^\infty_x}
		\leq & C[1+E(0)]^3 [\frac13(1-2^{-\delta_{1}})]^{-2}.
	\end{split}
\end{equation}
Due to \eqref{Tstar}, \eqref{kz65}, \eqref{kz66}, we have proved \eqref{Duu0}, \eqref{Duu00}, and \eqref{Duu2}.

It still remains for us to prove \eqref{Duu21} and \eqref{Duu22}. We present the proof in the following subsection.
\subsubsection{Strichartz estimates of linear wave equation on time-interval $[0,T^*_{N_0}]$.}\label{finalq}
We still expect the behaviour of a linear wave equation endowed with $g_j=g(\bv_j,\rho_j)$. So we claim a theorem as follows
\begin{proposition}\label{rut}
	For $s-\frac34 \leq r \leq \frac{11}{4}$, there is a solution $f_j$ on $[0,T^*_{N_0}]\times \mathbb{R}^2$ satisfying the following linear wave equation
	\begin{equation}\label{ru01}
		\begin{cases}
			\square_{{g}_j} f_j=0,
			\\
			(f_j,\partial_t f_j)|_{t=0}=(f_{0j},f_{1j}),
		\end{cases}
	\end{equation}
	where $(f_{0j},f_{1j})=(P_{\leq j}f_0,P_{\leq j}f_1)$ and $(f_0,f_1)\in H_x^r \times H^{r-1}_x$. Moreover, for $a\leq r-(s-1) $, we have
	\begin{equation}\label{ru02}
		\begin{split}
			&\|\left< \nabla \right>^{a-1} d{f}_j\|_{L^4_{[0,T^*_{N_0}]} L^\infty_x}
			\leq  C_{M_0}(\|{f}_0\|_{{H}_x^r}+ \|{f}_1 \|_{{H}_x^{r-1}}),
			\\
			&\|{f}_j\|_{L^\infty_{[0,T^*_{N_0}]} H^{r}_x}+ \|\partial_t {f}_j\|_{L^\infty_{[0,T^*_{N_0}]} H^{r-1}_x} \leq  C_{M_0}(\| {f}_0\|_{H_x^r}+ \| {f}_1\|_{H_x^{r-1}}).
		\end{split}
	\end{equation}
\end{proposition}
\begin{proof}[Proof of Proposition \ref{rut}.] Firstly, there are some Strichartz estimates for short time intervals. By applying the extension method described in subsection \ref{esest} and summing the short-time estimates corresponding to these intervals, we can derive a specific type of Strichartz estimate that incurs a loss of derivatives. Due to	
\begin{equation}\label{ru060}
	a+9\delta_{1} \leq r-(s-1)+9\delta_{1} < r-(s-1)+(s-\frac74) <r-\frac34,
\end{equation}	
using \eqref{ru01}, \eqref{ru03} and \eqref{ru04}, we therefore have
	\begin{equation}\label{ru06}
		\begin{split}
			\| \left< \nabla \right>^{a-1+9\delta_{1}} df_j\|_{L^4_{[0,T^*_j]} L^\infty_x} \leq & C( \| f_{0j} \|_{H^r_x} + \| f_{1j} \|_{H^{r-1}_x}  )
			\\
			\leq & C ( \| f_{0} \|_{H^r_x} + \| f_{1} \|_{H^{r-1}_x}  ).
		\end{split}
	\end{equation}
Next, we will discuss the cases of high frequency and low frequency. For $k\geq j$, which corresponds to the high frequency, we apply  Bernstein's inequality to obtain
	\begin{equation}\label{ru07}
		\begin{split}
			\| \left< \nabla \right>^{a-1} P_k df_j\|_{L^4_{[0,T^*_j]} L^\infty_x}
			=& C2^{-9\delta_{1} k}\| \left< \nabla \right>^{a-1+9\delta_{1}} P_k df_j\|_{L^4_{[0,T^*_j]} L^\infty_x}
			\\
			\leq & C2^{-\delta_{1} k} 2^{-8\delta_{1} j} \| \left< \nabla \right>^{a-1+9 \delta_{1}} df_j\|_{L^4_{[0,T^*_j]} L^\infty_x}.
		\end{split}
	\end{equation}
	Combining \eqref{ru06} and \eqref{ru07}, for $a \leq r-(s-1)$, we obtain
	\begin{equation}\label{ru070}
		\begin{split}
			\| \left< \nabla \right>^{a-1} P_k df_j\|_{L^4_{[0,T^*_j]} L^\infty_x}
			\leq & C2^{-\delta_{1} k} 2^{-8\delta_{1} j} ( \| f_{0} \|_{H^r_x} + \| f_{1} \|_{H^{r-1}_x}  ), \quad k\geq j.
		\end{split}
	\end{equation}
For $k\leq j$, which corresponds to the low frequency, we need to discuss the difference terms. For any integer $m\geq 1$, we have
	\begin{equation*}
		\begin{cases}
			\square_{{g}_m} f_m=0, \quad [0,T^*_{m}]\times \mathbb{R}^2,
			\\
			(f_m,\partial_t f_m)|_{t=0}=(f_{0m},f_{1m}),
		\end{cases}
	\end{equation*}
	and
	\begin{equation*}
		\begin{cases}
			\square_{{g}_{m+1}} f_{m+1}=0, \quad [0,T^*_{m+1}]\times \mathbb{R}^2,
			\\
			(f_{m+1},\partial_t f_{m+1})|_{t=0}=(f_{0(m+1)},f_{1(m+1)}).
		\end{cases}
	\end{equation*}
Then the difference term $f_{m+1}-f_m$ is satisfied by
	\begin{equation*}\label{ru08}
		\begin{cases}
			\square_{{g}_{m+1}} (f_{m+1}-f_m)=({g}^{\alpha i}_{m+1}-{g}^{\alpha i}_{m} )\partial^2_{\alpha i} f_m, \quad [0,T^*_{m+1}]\times \mathbb{R}^2,
			\\
			(f_{m+1}-f_m,\partial_t (f_{m+1}-f_m))|_{t=0}=(f_{0(m+1)}-f_{0m},f_{1(m+1)}-f_{1m}).
		\end{cases}
	\end{equation*}
Due to \eqref{ru04}, \eqref{ru060}, and Duhamel's principle, which yields
	\begin{equation}\label{ru09}
		\begin{split}
			& \| \left< \nabla \right>^{a-1} P_k (d f_{m+1}-df_m)\|_{L^4_{[0,T^*_{m+1}]} L^\infty_x}
			\\
			\leq & C\| f_{0(m+1)}-f_{0m}\|_{H_x^{r-9\delta_{1}}} + C\|f_{1(m+1)}-f_{1m}\|_{H_x^{r-1-9\delta_{1}}}
			\\
			& \ + C\|({g}_{m+1}-{g}_{m}) \cdot\nabla d f_m\|_{L^1_{[0,T^*_{m+1}]} H^{r}_x}
			\\
			\leq & C2^{-9\delta_{1} m}( \| f_{0} \|_{H^r_x} + \| f_{1} \|_{H^{r-1}_x}  ) + C \| {g}_{m+1}-{g}_{m}\|_{L^1_{[0,T^*_{m+1}]} L^\infty_x} \| \nabla d f_m\|_{L^\infty_{[0,T^*_{m+1}]} H^{r-1}_x} .
		\end{split}
	\end{equation}
By applying energy estimates and the Bernstein inequality, it follows that
	\begin{equation}\label{ru10}
		\begin{split}
			\| \nabla d f_m\|_{L^\infty_{[0,T^*_{m+1}]} H^{r-1}_x} \leq & C\| d f_m\|_{L^\infty_{[0,T^*_{m+1}]} H^{r}_x}
			\\
			\leq & C(\|f_{0m}\|_{H^{r+1}_x}+\|f_{1m}\|_{H^{r}_x})
			\\
			\leq & C2^m (\|f_{0m}\|_{H^r_x}+\|f_{1m}\|_{H^{r-1}_x}) .
		\end{split}
	\end{equation}
By utilizing the Strichartz estimates \eqref{ru04} and Lemma \ref{LD}, we obtain
	\begin{equation}\label{ru11}
		\begin{split}
			& 2^m \| {g}_{m+1}-{g}_{m}\|_{L^1_{[0,T^*_{m+1}]} L^\infty_x}
			\\
			\leq & 2^m (T^*_{m+1})^{\frac34}\| {\bv}_{m+1}-{\bv}_{m},  {\rho}_{m+1}-{\rho}_{m}\|_{L^4_{[0,T^*_{m+1}]} L^\infty_x}
			\\
			\leq & C2^m ( \| {\bv}_{m+1}-{\bv}_{m},  {\rho}_{m+1}-{\rho}_{m}\|_{L^\infty_{[0,T^*_{m+1}]} H^{\frac34+\delta_{1}}_x} + \| {w}_{m+1}-{w}_{m}\|_{L^\infty_{[0,T^*_{m+1}]} H^{\delta_{1}}_x})
			\\
			\leq & C2^{-8\delta_{1}m } \left[  1+E(0) \right]^3.
		\end{split}
	\end{equation}
Above, we also use \eqref{yu0}, \eqref{yu1}, and \eqref{DTJ}. According to \eqref{ru10} and \eqref{ru11}, for $k<j$ and $k\leq m$, \eqref{ru09} tells us that
	\begin{equation}\label{ru12}
		\begin{split}
			\| \left< \nabla \right>^{a-1} P_k (df_{m+1}-df_m)\|_{L^4_{[0,T^*_{m+1}]} L^\infty_x}
			\leq  C2^{-\delta_{1} k}2^{-8\delta_{1} m} \| (f_{0},f_1) \|_{H^r_x\times H^{r-1}_x}  \left[  1+E(0) \right]^3.
		\end{split}
	\end{equation}
Although we have established some estimates for difference terms, we also need to discuss the energy estimates. This is necessary because the key point of the extension method (in subsection \ref{esest}) crucially relies on the uniform bound of the energy. To achieve this goal, we take the operator $\left< \nabla \right>^{r-1}$ on \eqref{ru01}, which yields the following result:
	\begin{equation}\label{ru15}
		\begin{cases}
			\square_{{g}_j} \left< \nabla \right>^{r-1}f_j=[ \square_{{g}_j} ,\left< \nabla \right>^{r-1}] f_j,
			\\
			(\left< \nabla \right>^{r-1}f_j,\partial_t \left< \nabla \right>^{r-1}f_j)|_{t=0}=(\left< \nabla \right>^{r-1}f_{0j},\left< \nabla \right>^{r-1}f_{1j}).
		\end{cases}
	\end{equation}
	To estimate $[ \square_{{g}_j} ,\left< \nabla \right>^{r-1}] f_j$, we will divide it into two cases: $s-\frac34<r\leq s$ and $s<r\leq \frac{11}{4}$.
	
\textbf{Case 1: $s-\frac34<r\leq s$.} In this situation, we note that
	\begin{equation}\label{ru14}
		\begin{split}
			[ \square_{{g}_j} ,\left< \nabla \right>^{r-1}] f_j
			=& [{g}^{\alpha i}_j-\mathbf{m}^{\alpha i}, \left< \nabla \right>^{r-1} \partial_i ] \partial_{\alpha}f_j + \left< \nabla \right>^{r-1}( \partial_i g_j \partial_{\alpha}f_j )
			\\
			= &[{g}_j-\mathbf{m}, \left< \nabla \right>^{r-1} \nabla] df_j + \left< \nabla \right>^{r-1}( \nabla g_j df_j).
		\end{split}
	\end{equation}
	By \eqref{ru14} and Kato-Ponce estimates, we have\footnote{If $r=s$, then $L^{\frac{3}{s-r}}_x=L^\infty_x$.}
	\begin{equation}\label{ru16}
		\| [ \square_{{g}_j} ,\left< \nabla \right>^{r-1}] f_j \|_{L^2_x} \leq C ( \|dg_j\|_{L^\infty_x} \|d f_j \|_{H^{r-1}_x} + \|\left< \nabla \right>^{r} (g_j-\mathbf{m}) \|_{L^{\frac{2}{1+r-s}}_x} \| df_j \|_{L^{\frac{2}{s-r}}_x} )
	\end{equation}
By Sobolev's inequality, it follows that
	\begin{equation}\label{ru17}
		\|\left< \nabla \right>^{r} (g_j-\mathbf{m}) \|_{L^{\frac{2}{1+r-s}}_x} \leq C\|g_j-\mathbf{m}\|_{H^s_x}.
	\end{equation}
Taking advantage of the Gagliardo-Nirenberg and Young's inequalities, which yields
	\begin{equation}\label{ru18}
		\begin{split}
			\| df_j \|_{L^{\frac{3}{s-r}}_x} \leq & C\|\left< \nabla \right>^{r-1} df_j \|^{\frac{3/4}{3/4+(2-s)}}_{L^{2}_x} \|\left< \nabla \right>^{r-(s-\frac34)} df_j \|^{\frac{2-s}{3/4+(2-s)}}_{L^{\infty}_x}
			\\
			\leq & C( \|df_j \|_{H^{r-1}_x}+ \|\left< \nabla \right>^{r-(s-\frac34)} df_j \|_{L^{\infty}_x})
		\end{split}
	\end{equation}
\quad	\textbf{Case 2: $s<r\leq \frac{11}{4}$.} For $s<r\leq \frac{11}{4}$, applying Kato-Ponce estimates, we have
	\begin{equation}\label{ru150}
		\begin{split}
			 \| [ \square_{{g}_j} ,\left< \nabla \right>^{r-1}] f_j \|_{L^2_x}
			= & \| {g}^{\alpha i}_j-\mathbf{m}^{\alpha i},\left< \nabla \right>^{r-1}] \partial^2_{\alpha i}f_j \|_{L^2_x}
			\\
			\leq & C ( \|dg_j\|_{L^\infty_x} \|d f_j \|_{H^{r-1}_x} + \|\left< \nabla \right>^{r-1} (g_j-\mathbf{m}) \|_{L^{\frac{2}{1+s-r}}_x} \|\nabla df_j \|_{L^{\frac{2}{r-s}}_x} ) .
		\end{split}
	\end{equation}
	By Sobolev's inequality, it follows that
	\begin{equation}\label{ru151}
		\|\left<\nabla \right>^{r-1} (g_j-\mathbf{m}) \|_{L^{\frac{2}{1+s-r}}_x} \leq C\|g_j-\mathbf{m}\|_{H^s_x}.
	\end{equation}
	Note that $\frac74<s\leq 2$. By applying the Gagliardo-Nirenberg and Young's inequalities, we derive:
	\begin{equation}\label{ru152}
		\begin{split}
			\| \nabla df_j \|_{L^{\frac{3}{1+s-r}}_x} \leq & C\|\left< \nabla \right>^{r-1} df_j \|^{\frac{3/4}{3/4+(2-s)}}_{L^{2}_x} \|\left< \nabla \right>^{r-(s-\frac34)} df_j \|^{\frac{2-s}{3/4+(2-s)}}_{L^{\infty}_x}
			\\
			\leq & C( \|df_j \|_{H^{r-1}_x}+ \|\left< \nabla \right>^{r-(s-\frac34)} df_j \|_{L^{\infty}_x}).
		\end{split}
	\end{equation}
Therefore, in both cases, for \( s - \frac{3}{4} < r \leq \frac{11}{4} \), using \eqref{ru16}-\eqref{ru18} and  \eqref{ru150}-\eqref{ru152}, we can derive
	\begin{equation*}\label{ru19}
		\begin{split}
			\| [ \square_{{g}_j} ,\left< \nabla \right>^{r-1}] f_j \|_{L^2_x}\|d f_j \|_{H^{r-1}_x} \leq & C  \|dg_j\|_{L^\infty_x} \|d f_j \|^2_{H^{r-1}_x}
			+ C  \| g_j-\mathbf{m} \|_{H^s_x}\|d f_j \|^2_{H^{r-1}_x}
			\\
			& + C  \|g_j-\mathbf{m}\|_{H^s_x} \|\left< \nabla \right>^{r-\frac{s}{2}-1} df_j \|_{L^{\infty}_x} \|d f_j \|_{H^{r-1}_x}
			\\
			\leq & C  ( \|dg_j\|_{L^\infty_x}+ \|g_j-\mathbf{m}\|_{H^s_x}+ \|\left< \nabla \right>^{r-(s-\frac34)} df_j \|_{L^{\infty}_x} ) \|d f_j \|^2_{H^{r-1}_x}
			\\
			&+ C  \|g_j-\mathbf{m}\|^2_{H^s_x}  \|\left< \nabla \right>^{r-(s-\frac34)} df_j \|_{L^{\infty}_x} .
		\end{split}	
	\end{equation*}
The above estimate, when combined with \eqref{ru15}, yields
	\begin{equation*}\label{ru20}
		\begin{split}
			\frac{d}{dt}\|d f_j \|^2_{H^{r-1}_x}
			\leq & C  (\|dg_j\|_{L^\infty_x}+ \|g_j-\mathbf{m}\|_{H^s_x}+ \|\left< \nabla \right>^{r-(s-\frac34)} df_j \|_{L^{\infty}_x} ) \|d f_j \|^2_{H^{r-1}_x}
			\\
			&+ C  \|g_j-\mathbf{m}\|^2_{H^s_x}  \|\left< \nabla \right>^{r-(s-\frac34)} df_j \|_{L^{\infty}_x} .
		\end{split}	
	\end{equation*}
Using Gronwall's inequality, we obtain
	\begin{equation}\label{ru21}
		\begin{split}
			\|d f_j(t) \|^2_{H^{r-1}_x}
			\leq & C  (\|d f_j(0) \|^2_{H^{r-1}_x} + C \int^t_0 \|g_j-\mathbf{m}\|^2_{H^s_x}  \|\left< \nabla \right>^{r-(s-\frac34)} df_j \|_{L^{\infty}_x}d\tau)
			\\
			& \qquad \cdot\exp\left\{\int^t_0  \|dg_j\|_{L^\infty_x}+ \|g_j-\mathbf{m}\|_{H^s_x}+ \|\left< \nabla \right>^{r-(s-\frac34)} df_j \|_{L^{\infty}_x} d\tau \right\} .
		\end{split}	
	\end{equation}
On the other hand, the Newton-Leibniz formula tells us
	\begin{equation}\label{ru22}
		\begin{split}
			\|f_j(t) \|_{L^{2}_x}
			\leq & \| f_j(0) \|_{L^2_x} + \int^t_0 \| \partial_t f_j \|_{L^2_x} d\tau.
		\end{split}	
	\end{equation}
Summarizing the results from \eqref{ru21}, \eqref{ru22} and \eqref{kz65}, for \( t \in [0, T^*_{N_0}] \), we obtain
	\begin{equation}\label{ru23}
		\begin{split}
			& \| f_j(t) \|^2_{H^{r}_x} + \| d f_j(t) \|^2_{H^{r-1}_x}
			\\
			\leq & C  \textrm{e}^{C_*}(\|f_0 \|^2_{H^{r}_x}+\|f_1 \|^2_{H^{r-1}_x} + C_*)
			\\
			& \quad \cdot\exp \left\{\int^t_0   \|\left< \nabla \right>^{r-\frac{s}{2}-1} df_j \|_{L^{\infty}_x} d\tau \cdot 
			\exp \left( \int^t_0   \|\left< \nabla \right>^{r-(s-\frac34)} df_j \|_{L^{\infty}_x} d\tau \right)  \right\} .
		\end{split}	
	\end{equation}
For \( a \leq r - (s - 1) \), based on \eqref{ru070}, \eqref{ru12}, and \eqref{ru23}, and following the extension method outlined in subsection \ref{esest}, we can obtain the following bounds:
	\begin{equation}\label{ru234}
		\begin{split}
			\| \left< \nabla \right>^{a-1} P_k df_j\|_{L^4_{[0,T^*_{N_0}]} L^\infty_x}
			\leq & C2^{-\delta_{1} k} 2^{-8\delta_{1} j} ( \| f_{0} \|_{H^r_x} + \| f_{1} \|_{H^{r-1}_x}  ) \times \left\{ 2^{\delta_{1} (j-N_0)} \right\}^{\frac34}
			\\
			\leq & C2^{-\frac34\delta_{1}N_0} 2^{-\delta_{1} k} 2^{-4\delta_{1} j} ( \| f_{0} \|_{H^r_x} + \| f_{1} \|_{H^{r-1}_x}  ), \quad k \geq j,
		\end{split}
	\end{equation}
	and
	\begin{equation}\label{ru25}
		\begin{split}
			& \| \left< \nabla \right>^{a-1} P_k (df_{m+1}-df_m)\|_{L^4_{[0,T^*_{N_0}]} L^\infty_x}
			\\
			\leq & C2^{-\delta_{1} k}2^{-8\delta_{1} m}( \| f_{0} \|_{H^r_x} + \| f_{1} \|_{H^{r-1}_x}  )(1+E^3(0)) \times \left\{ 2^{\delta_{1} (m+1-N_0)} \right\}^{\frac34}
			\\
			\leq & C2^{-\frac34\delta_{1}N_0} 2^{-\delta_{1} k}2^{-4\delta_{1} m}( \| f_{0} \|_{H^r_x} + \| f_{1} \|_{H^{r-1}_x}  ) \left[  1+E(0) \right]^3 , \quad k<j.
		\end{split}
	\end{equation}
	By decomposing the frequency, we obtain
	\begin{equation*}
		df_j= P_{\geq j} df_j+ \textstyle{\sum}^{j-1}_{k=1} \textstyle{\sum}^{j-1}_{m=k} P_k (df_{m+1}-df_m)+ \textstyle{\sum}^{j-1}_{k=1} P_k df_k.
	\end{equation*}
	By using \eqref{ru234} and \eqref{ru25}, we get
	\begin{equation}\label{ru230}
		\begin{split}
			\| \left< \nabla \right>^{a-1} df_{j} \|_{L^4_{[0,T^*_{N_0}]} L^\infty_x}
			\leq & C( \| f_{0} \|_{H^r_x} + \| f_{1} \|_{H^{r-1}_x}  )\left[  1+E(0) \right]^3  [\frac13(1-2^{-\delta_{1}})]^{-2}.
		\end{split}
	\end{equation}
	Applying \eqref{ru23} and \eqref{ru230}, we obtain
	\begin{equation*}\label{ru26}
		\begin{split}
			\|f_j\|_{L^\infty_{[0,T^*_{N_0}]} H^{r}_x}+ \|d f_j\|_{L^\infty_{[0,T^*_{N_0}]} H^{r-1}_x }
			\leq & A_* \exp ( B_* \mathrm{e}^{B_*} ),
		\end{split}	
	\end{equation*}
	where $C_*$ is as stated in \eqref{Cstar} and 
	\begin{equation*}
		\begin{split}
			A_{*}=& C  \mathrm{e}^{C_*}(\|f_0 \|_{H^{r}_x}+\|f_1 \|_{H^{r-1}_x} + C_*),
			\\
			B_*=& C(\| f_{0} \|_{H^r_x} + \| f_{1} \|_{H^{r-1}_x})\left[  1+E(0) \right]^3 [\frac13(1-2^{-\delta_{1}})]^{-2}.
		\end{split}
	\end{equation*}
At this stage, we have proved \eqref{ru02}.
\end{proof}

\section*{Acknowledgments} The author thanks Daniel Tataru for many hours of discussion. The author is supported by Natural Science Foundation of Hunan Province, China (Grant No. 2025JJ40003) and the Fundamental Research Funds for the Central Universities (Grant No. 531118010867).

\section*{Conflicts of interest and Data Availability Statements}
The authors declared that this work does not have any conflicts of interest. The authors also confirm that the data supporting the findings of this study are available within the article.

\end{sloppypar}

\begin{thebibliography}{4}

\bibitem{AS}
L.Abbrescia, J. Speck. Remarkable localized integral identities for 3D compressible Euler flow and the double-null framework,  Arch. Ration. Mech. Anal., 249(1), Paper No. 10, 137 pp (2025).


\bibitem{AIT}
A. Ai, M. Ifrim, D. Tataru. The time-like minimal surface equation in Minkowski space: low regularity solutions, Invent. Math., 235 (3), 745-891 (2024).


\bibitem{AAR}
P.T. Allen, L. Andersson, A. Restuccia. Local well-posedness for membranes in the light
cone gauge, Comm. Math. Phys., 301, 383-410 (2011).

\bibitem{ACY1}
X.~L. An, H.~Y. Chen, S.~L. Yin. The Cauchy problems for the 2D compressible Euler equations and ideal MHD system are ill-posed in $H^{\frac74}(\mathbb{R}^2)$, 
{\tt arXiv:2206.14003} (2022).

\bibitem{ACY2}
X.~L. An, H.~Y. Chen, S.~L. Yin. Low regularity ill-posedness and shock formation for 3D ideal compressible
MHD, {\tt arXiv:2110.10647} (2021).



\bibitem{AM}
L. Andersson, V. Moncrief. Elliptic-hyperbolic systems and the Einstein equations, Ann. Henri Poincar\'e, 4, 1-34(2003).

\bibitem{AZ}
L. Andersson, H.L. Zhang, Well-posedness for rough solutions of the 3D compressible Euler equations, arXiv:2208.10132 (2023).

\bibitem{Av}
O.N. Avadanei. Counterexamples to Strichartz estimates and gallery waves for the irrotational compressible Euler equation in a vacuum setting, {\tt arXiv:2504.17932} (2025).

\bibitem{BC2}
H. Bahouri and J.Y. Chemin. E\'quations d\'ondes quasilineaires et estimation de Strichartz, Amer. J. Math., 121, 1337-1377 (1999).
\bibitem{BCD}
H. Bahouri, J. Y. Chemin, and R. Danchin. Fourier analysis and nonlinear partial differential equations,
Grundlehren der Mathematischen Wissenschaften, vol. 343, Springer,
Heidelberg, 2011.

\bibitem{BL}
J. Bourgain, D. Li. Strong ill-posedness of the incompressible Euler
equation in borderline Sobolev spaces, Invent. Math., 201:97-157 (2015).

\bibitem{BL2}
J. Bourgain, L. Dong. Strong ill-posedness of the 3D incompressible Euler equation in borderline
spaces, Int. Math. Res. Notices, 2021(16), 1-110 (2019).
\bibitem{Chae}
D. Chae. On the well-posedness of the Euler equations in the Triebel-Lizorkin spaces. Comm. Pure Appl. Math. 55(5), 654-678 (2002).
\bibitem{Ch2}
D. Chae. On the Euler equations in the critical Triebel-Lizorkin spaces, Arch. Ration. Mech. Anal., 170(3), 185-210 (2003).


\bibitem{CM}
D. Christodoulou and S. Miao. Compressible flow and Euler's equations, Surveys of Modern Mathematics,
vol. 9, International Press, Somerville, MA; Higher Education Press, Beijing, 2014.

\bibitem{CLS}
D. Coutand, H. Lindblad, and S. Shkoller. A priori estimates for the free-boundary 3D compressible Euler
equations in physical vacuum, Comm. Math. Phys. 296(2), 559-587 (2010).




\bibitem{DLS}
M. Disconzi, C. Luo, G. Mazzone and J. Speck. Rough sound waves in 3D compressible Euler flow with
vorticity,  Selecta Math., 28(2), Paper No. 41, 153 pp (2022).
\bibitem{EL}
B. Ettinger, H. Lindblad. A sharp counterexample to local existence
 of low regularity solutions to Einstein
 equations in wave coordinates, Ann. Math., 185, 311-330 (2017).

\bibitem{GL}
Z.H. Guo, K.L. Li. Remarks on the well-posedness of the Euler equations in the Triebel-Lizorkin spaces,  J. Fourier Anal. Appl., 27(2), Paper No. 29, 24 pp (2021).

\bibitem{JM}
J. Jang and N. Masmoudi. Well-posedness for compressible Euler equations with physical vacuum singularity, Comm.
Pure Appl. Math., 62(10), 1327-1385 (2009).

\bibitem{HKWX}
F. Huang, J. Kuang, D. Wang, X. Wei. Stability of transonic contact discontinuity for two-dimensional steady compressible Euler flows in a finitely long nozzle, Ann. PDE, 7(2), Paper No. 23, 96 pp (2021).

\bibitem{HKM}
T. Hughes, T. Kato and J.E. Marsden. Well-posed quasi-linear second-order hyperbolic systems with
applications to nonlinear electrodynamics and general relativity, Arch. Rat. Mech. Anal., 63, 273-
294 (1977).

\bibitem{IT}
M. Ifrim, D. Tataru. The compressible Euler equations in a physical vacuum:a comprehensive Eulerian approach, Ann. Inst. H. Poincar$\acute{\textrm{e}}$ C Anal. Non Lin$\acute{\textrm{e}}$aire, 41(2), 405-495 (2024). 
\bibitem{IT1}
M. Ifrim, D. Tataru. Local well-posedness for quasilinear problems:a primer, Bull. Am. Math. Soc., New Ser., 60(2), 167-194 (2023).


\bibitem{Geba}
D.G. Geba. A local well-posedness result for the quasilinear
wave equation in $\mathbb{R}^{2+1}$,  Comm. Part. Diff. Equ., 29, 323-360 (2004).



\bibitem{KP}
T. Kato, G. Ponce. Commutator estimates and the Euler and Navier-Stokes equations, Comm. Pure Appl. Math., 41(7), 891-907 (1988).

\bibitem{KT}
M. Keel, T. Tao. Endpoint Strichartz estimates, Am. J. Math., 120, 955-980 (1998).

\bibitem{KJ}
J. Kim, I. Jeong. A simple ill-posedness proof for incompressible Euler equations in critical Sobolev spaces, J. Funct. Anal., 283(10), Paper No. 109673, 34 pp (2022).

\bibitem{K}
S. Klainerman. A commuting vectorfield approach to Strichartz type inequalities and applications to quasilinear wave equations, Int. Math. Res. Notices, 5, 221-274 (2001).

\bibitem{KR2}
S. Klainerman, I. Rodnianski. Improved local well-posedness for quasilinear wave equations in dimension three, Duke Math. J., 117, 1-124 (2003).
\bibitem{KR}
S. Klainerman, I. Rodnianski. Rough solutions of the Einstein vacuum equations, Ann. Math., 161, 1143-1193 (2005).

\bibitem{KR1}
S. Klainerman, I. Rodnianski and J. Szeftel. The bounded $\textrm{L}^2$ curvature conjecture, Invent. Math., 202(1), 91216 (2015).



\bibitem{LDZ}
Z. Lei, Y. Du, Q.T. Zhang. Singularities of solutions to compressible Euler equations with
vacuum, Math. Res. Lett., 20, 41-50 (2013).

\bibitem{LU}
P. G. LeFloch and S. Ukai. A symmetrization of the relativistic Euler equations in
several spatial variables, Kinet. Relat. Mod., 2, 275–292 (2009).

\bibitem{Li}
T. Li, T. Qin. Physics and Partial Differential Equations (second edition), Higher Education Press, Beijing, (2005).







\bibitem{L}
H. Lindblad. Counterexamples to local existence for quasilinear wave equations, Math. Res. Letters, 5(5), 605-622 (1998).
\bibitem{LS1}
J. Luk, J. Speck. Shock formation in solutions to the 2D compressible Euler equations in the presence of non-zero vorticity, Invent. Math., 214, 1-169 (2018).
\bibitem{LS2}
J. Luk, J. Speck. The hidden null structure of the compressible Euler equations and a prelude to applications,  J. Hyperbolic Differ. Equ., 17(1), 1-60 (2020).

\bibitem{LY1}
T. Luo, P. Yu. On the stability of multi-dimensional rarefaction waves I: the energy estimates, Ann. of Math., 202(2), 631-752 (2025).

\bibitem{LY2}
T. Luo, P. Yu. On the stability of multi-dimensional rarefaction waves II: existence of solutions and applications to the Riemann problem, Ann. of Math., 202(2), 753-855 (2025).


\bibitem{M}
A. Majda. Compressible fluid flow and systems of conservation laws in several space variables, Applied
Mathematical Sciences, 53. Springer, New York, 1984.


\bibitem{MRR}
F. Merle, P. Rapha\"el, I. Rodnianski, J. Szeftel. On smooth self similar solutions to the compressible Euler equations, 	{\tt arXiv:1912.10998} (2019).


\bibitem{ML}
C. Miao, L. Xue. On the global well-posedness of a class of Boussinesq-Navier-Stokes systems, NoDEA Nonlinear Differential Equations Appl., 18 (6), 707-735 (2011).
\bibitem{MSS}
G. Mockenhaupt, A. Seeger, and C. D. Sogge. Local smoothing of Fourier integral
operators and Carleson-Sj\"olin estimates, J. Amer. Math. Soc. 6, 65-130 (1993).

\bibitem{MR}
G. Moschidis, I. Rodnianski. On well-posedness for the timelike minimal surface equation, {\tt arXiv:2504.01244 } (2025).

\bibitem{Olm}
G. Ohlmann, Ill-posedness of a quasilinear wave equation in two dimensions for data in $H^{\frac74}$, Pure Appl. Anal., 5(3), 507-540 (2023).



\bibitem{QX}
P. Qu, Z. Xin. Long time existence of entropy solutions to the one-dimensional non-isentropic Euler equations with periodic initial data, Arch. Ration. Mech. Anal., 216(1), 221-259 (2015).

\bibitem{S}
T.C. Sideris. Formation of singularities in three dimensional
compressible fluids, Commun. Math. Phys., 101, 475-485 (1985).

\bibitem{ST}
H.F. Smith, D. Tataru. Sharp local well-posedness results for the nonlinear wave equation, Ann.
Math., 162, 291-366 (2005).

\bibitem{ST0}
H.F. Smith and D. Tataru. Sharp counterexamples for Strichartz estimates for low regularity metrics,
Mathematical Research Letters, 199-204 (2002).

\bibitem{Sm}
H. F. Smith, A parametrix construction for wave equations with $C^{1,1}$ coefficients, Ann.
Inst. Fourier (Grenoble) 48, 797-835 (1998).

\bibitem{SS}
H. F. Smith and C. D. Sogge, On Strichartz and eigenfunction estimates for low regularity metrics, Math. Res. Lett., 1, 729-737 (1994).


\bibitem{S1}
J. Speck. Shock formation for 2D quasilinear wave systems featuring multiple speeds: Blowup for the fastest wave, with non-trivial interactions up to the singularity, Ann. PDE, 4(1), Paper No. 6, 131 pp (2018).




\bibitem{Tao}
T. Tao. Global regularity of wave maps. $\Pi$. Small energy in two dimensions, Comm. Math. Phys., 224(2), 443-544 (2001).


\bibitem{T1}
D. Tataru. Strichartz estimates for operators with nonsmooth coefficients and the nonlinear wave
equation, Am. J. Math., 122, 349-376 (2000).
\bibitem{T2}
D. Tataru. Strichartz estimates for second order hyperbolic operators with nonsmooth coefficients II,
Am. J. Math., 123, 385-423 (2001).
\bibitem{T3}
D. Tataru. Strichartz estimates for second order hyperbolic operators with nonsmooth coefficients III.
J. Am. Math. Soc., 15, 419-442 (2002).

\bibitem{T4}
D. Tataru. Rough solutions for the wave maps equation. Amer. J. Math., 127(2), 293-377 (2005).


\bibitem{WCB}
C.B. Wang. Sharp local well-posedness for quasilinear wave equations with spherical symmetry, J. Eur. Math. Soc., 25(11), 4459-4520 (2023).

\bibitem{WQRough}
Q. Wang. Rough Solutions of Einstein vacuum equations in CMCSH gauge, Comm. Math. Phys. 328, 1275-1340 (2014).


\bibitem{WQSharp}
Q. Wang. A geometric approach for sharp local well-posedness
of quasilinear wave equations, Ann. PDE, 3:12 (2017).
\bibitem{WQEuler}
Q. Wang. Rough solutions of the 3-D compressible Euler equations, Ann.
of Math., 195(2), 509-654 (2022).

\bibitem{WZ}
S. Wang, Y. Zhou. Global well-posedness for radial extremal hypersurface equation in 
-dimensional Minkowski space-time in critical Sobolev space, {\tt arXiv:2212.08828} (2022).

\bibitem{Yin}
H.C. Yin. Formation and construction of a shock wave for 3-D compressible Euler equations with the spherical
initial data, Nagoya Math. J., 175, 125-164 (2004).

\bibitem{ZH}
D. Zha, K. Hidano. Global solutions to systems of quasilinear wave equations with low regularity data and applications,
J. Math. Pures Appl. 142(9), 146-183 (2020).

\bibitem{Z1}
H.L. Zhang. Local existence theory for 2D compressible Euler equations with low regularity, J. Hyperbolic Differ. Equ., 18(3), 701-728 (2021).

\bibitem{Z2}
H.L. Zhang. Low regularity solutions of two-dimensional compressible Euler equations with dynamic vorticity,  Comm. Anal. Geom., 33(2), 453-529 (2025).

\bibitem{ZL}
Y. Zhou, Z. Lei. Global low regularity solutions of quasi-linear wave equations. Adv. Differential Equations,
13(1-2), 55-104 (2008).

\end{thebibliography}
\end{document}